\documentclass[letterpaper,11pt,reqno]{amsart}
\usepackage[active]{srcltx}
\usepackage{longtable}
\usepackage{bm}
\usepackage{amsmath,amssymb,amscd, booktabs,mathtools,amsthm}
\usepackage[all]{xy}
\usepackage[table]{xcolor}
\usepackage[headings]{fullpage}

\let\oldtocsection=\tocsection

\let\oldtocsubsection=\tocsubsection

\let\oldtocsubsubsection=\tocsubsubsection

\renewcommand{\tocsection}[2]{\hspace{0em}\oldtocsection{#1}{#2}}
\renewcommand{\tocsubsection}[2]{\hspace{1em}\oldtocsubsection{#1}{#2}}
\renewcommand{\tocsubsubsection}[2]{\hspace{2em}\oldtocsubsubsection{#1}{#2}}
\newcommand{\mat}[4]{{\setlength{\arraycolsep}{0.5mm}\left[
\begin{array}{cc}#1&#2\\#3&#4\end{array}\right]}}
\newcommand{\forget}[1]{}

\allowdisplaybreaks

\newtheorem{lemma}{Lemma}[section]
\newtheorem{theorem}[lemma]{Theorem}
\newtheorem{corollary}[lemma]{Corollary}

\newtheorem*{theorem*}{Theorem ($\GL(2)$ twisting)}

\usepackage[pdftex, colorlinks, linkcolor={blue}, citecolor={blue} ,pdfpagelabels, 
hyperindex, plainpages, pageanchor ]{hyperref}

\newcommand{\Q}{{\mathbb Q}}
\newcommand{\Z}{{\mathbb Z}}
\newcommand{\N}{{\mathbb N}}
\newcommand{\R}{{\mathbb R}}
\newcommand{\C}{{\mathbb C}}

\newcommand{\p}{\mathfrak p}

\newcommand{\GL}{{\rm GL}}

\newcommand{\GSp}{{\rm GSp}}

\newcommand{\trace}{{\rm tr}}

\newcommand{\SSp}{{\rm Sp}}

\begin{document}
 
\title{Fourier coefficients for twists of Siegel paramodular forms (expanded version)}
\author[Johnson-Leung and Roberts]{Jennifer Johnson-Leung\\
Brooks Roberts}
\thanks{The first author was partially supported by an NSA Young Investigator's Award.}

\begin{abstract}
In this paper, we calculate the Fourier coefficients of the paramodular twist of a Siegel modular form of paramodular level $N$ by a nontrivial quadratic Dirichlet character mod $p$ for $p$ a prime not dividing $N$.  As an application, these formulas can be used to verify the nonvanishing of the twist for particular examples.  We also deduce that the twists of Maass forms are identically zero.
\end{abstract}
\maketitle
\section{Introduction}
Let $S_k(\Gamma^{\mathrm{para}}(N))$ be the space of Siegel cusp forms of paramodular level $N$, weight $k$, and degree two; if $N=1$, then this is the space of Siegel cups forms, $S_k(\SSp(4,\Z))$.  Let $p\nmid N$ be an odd prime, and let $\chi=\big(\frac{\cdot}{p}\big)$ be the nontrivial quadratic Dirichlet character of conductor $p$.  In \cite{JR3}, we introduced a twisting map $\mathcal{T}_\chi:
S_k(\Gamma^{\mathrm{para}}(N))\rightarrow S_k(\Gamma^{\mathrm{para}}(Np^4))$.  This map is induced by a corresponding local twisting map, which we studied in \cite{JR2}.  The nature of this local map implies that $\mathcal{T}_\chi$ is an analog, at the level of Siegel modular forms, of the map sending a cuspidal automorphic representation of $\GSp(4)/\Q$ to its twist by $\chi$.  

In this paper, we explicitly calculate $\mathcal{T}_\chi$ in terms of Fourier coefficients.  Let $F\in S_k(\Gamma^{\mathrm{para}}(N))$ have the Fourier expansion 
$$
F(Z)=\sum_{S\in A(N)^+}a(S)e^{2\pi i \trace{(SZ)}},
$$
and write the Fourier expansion of $\mathcal{T}_\chi(F)$ as
$$
\mathcal{T}_\chi(F)(Z)=\sum_{S\in A(Np^4)^+}W(\chi)a_\chi(S)e^{2\pi i \trace{(SZ)}}.
$$
Here, $A(N)^+$ and $A(Np^4)^+$ are defined as in \eqref{A(N)definition}, $W(\chi)$ is the Gauss sum of $\chi$, and $S=\left[\begin{smallmatrix}\alpha&\beta\\\beta&\gamma\end{smallmatrix}\right]$ with $\alpha, 2\beta, \gamma\in\Z$.   The main result of this work calculates the coefficients $a_\chi(S)$ in terms of $a(S)$.  In particular, we prove that if $p\nmid 2\beta$ then
$$
a_\chi(S)=p^{1-k}\chi(2\beta)\sum_{b=1}^{p-1}\chi(b)a(S[\begin{bmatrix}1&-bp^{-1}\\&p\end{bmatrix}]).
$$
Here $S[A]={}^tASA$ for a $2\times2$ matrix $A$.  This formula is analogous to the formula for the twist of elliptic modular forms $a_\chi(n)=W(\chi)\chi(n)a(n)$ given, for example, in \cite{Sh}.  When $p\mid 2\beta$, the formula for $a_\chi(S)$ is more complicated.  The full statement appears in Theorem \ref{maintheorem} below.

As a corollary of the theorem, we also prove that if $N=1$ and $F$ is in the Maass space, then $\mathcal{T}_\chi(F)=0$.  This result may be viewed as an additional check of the formulas of the theorem since, as explained below, the vanishing on the Maass space can also be proven by a different argument.  In the more complicated expressions for $a_\chi(S)$, the vanishing requires cancellations between terms coming from different pieces of the operator as expressed in \cite{JR3}. In our view, this nontrivial cancellation indicates that the expressions in the theorem do not enjoy further nontrivial reductions. 

On the other hand, $\mathcal{T}_\chi(F)$ is generally nonzero.  For example, if $N=1$, $p=3$, and $F=\Upsilon20$, we verify below that $\mathcal{T}_\chi(F)\neq0$ using the formula of Theorem \ref{maintheorem} and the coefficients provided in the database \cite{LMFDB} .  More generally, if $F$ is a paramodular newform and the local component at $p$ of the corresponding automorphic representation is generic, then $\mathcal{T}_\chi(F)$ is a nonzero newform \cite{JR2}.  Therefore, our result provides an additional source of explicit examples of paramodular newforms.  These examples might be used to investigate the paramodular conjecture of Brumer and Kramer \cite{BK} or the paramodular B\"ocherer's conjecture \cite{RT}.

We note that in \cite{An} Andrianov studies a twist, depending on $\chi$,  on Fourier coefficients of Siegel modular forms with respect to $\Gamma_0(N)$ and the principal congruence subgroups.  The motivation for this map appears to be rather different from the twist considered in our work.  Indeed, it is easy to see from the formulas below that this twist does not agree with $\mathcal{T}_\chi$ in the case when the groups coincide ($N=1$).
\section{Notation}\label{notation}
Let 
$$
J=\begin{bmatrix} &\mathbf{1_2}\\-\mathbf{1_2}&\end{bmatrix}.
$$
The algebraic group $\GSp(4)$ is defined as the subgroup of $g\in\GL(4)$ such that ${}^tgJg=\lambda(g)J$ for some $\lambda(g)\in\GL(1)$ called the similitude factor of $g$.  We let $\SSp(4)$ be the kernel of the homomorphism defined by $g\mapsto\lambda(g)$. define $\GSp(4,\R)^+$ as the subgroup of $\GSp(4,\R)$ such that $\lambda(g)>0$ and  For $N$ a positive integer, we define the paramodular group of level $N$ as 
$$
\Gamma^{\mathrm{para}}(N)=\SSp(4,\Q)\cap\begin{bmatrix}\Z&\Z&N^{-1}\Z&\Z\\ N\Z&\Z&\Z&\Z\\N\Z&N\Z&\Z&N\Z\\N\Z&\Z&\Z&\Z\end{bmatrix}.
$$
Let $\mathfrak{H}_2$ be the Siegel upper half plane of degree 2.  The group $\GSp(4,\R)^+$ acts on $\mathfrak{H}_2$ by
$$
h\langle Z \rangle=(AZ+B)(CZ+D)^{-1}, \qquad h=\begin{bmatrix}A&B\\C&D\end{bmatrix},\qquad Z\in\mathfrak{H}_2.
$$
For a function $F:\mathfrak{H}_2\rightarrow \C$, $h\in\GSp(4,\R)^+$ as above, and a positive integer $k$, we define the function $F|_k:\mathfrak{H}_2\rightarrow \C$ via
$$
(F|_kh)(Z)=\lambda(h)^kj(h,Z)^{-k}F(h\langle Z\rangle),\qquad Z\in\mathfrak{H}_2,
$$
where $j(h,Z)=\det(CZ+D)$ is the factor of automorphy.
We use the notation of \cite{PY} as regards Siegel modular forms defined with respect to the paramodular groups. In particular, we let $S_k(\Gamma^{\mathrm{para}}(N))$ be the space of Siegel modular cusp forms of weight $k$, degree two, and paramodular level $N$.  Let $F \in S_k(\Gamma^{\mathrm{para}}(N))$. $F$ has a Fourier expansion
\begin{equation}\label{fouriereq}
F(Z) = \sum_{S \in A(N)^+} a(S) e^{2\pi i \mathrm{tr} (SZ)}
\end{equation}
for $Z \in \mathfrak{H}_2$. Here, $A(N)^+$ is the set of all  $2 \times 2$ matrices $S$ of the form:
\begin{equation}\label{A(N)definition}
S=\mat{\alpha}{\beta}{\beta}{\gamma},\quad \alpha\in N\Z,\quad\gamma \in \Z,  \quad \beta \in \frac{1}{2}\Z, \quad \alpha>0, \quad \det \mat{\alpha}{\beta}{\beta}{\gamma} = \alpha\gamma-\beta^2>0. 
\end{equation}
Since $F|_kg=F$ for the elements $g \in \Gamma^{\mathrm{para}}(N)$ of the form
$$
g
=\begin{bmatrix} 1&&mN^{-1}& \\ &1&& \\ &&1& \\ &&&1 \end{bmatrix}, \qquad m \in \Z
$$
or of the form
$$
g = \begin{bmatrix} A& \\ & {}^t A^{-1} \end{bmatrix}, \qquad A = \mat{a_1}{a_2}{a_3}{a_4} \in \GL(2,\Z), \ a_3 \equiv 0 (N)
$$
we have for all $S$ in $A(N)^+$, 
$$
a(S) \neq 0 \implies N\mid\alpha
$$
and
$$
a(\,{}^t\hspace{-1mm}AS A) = \det (A)^k a(S), \qquad A = \mat{a_1}{a_2}{a_3}{a_4} \in \GL(2,\Z), \ a_3 \equiv 0 (N).
$$
We fix $p$ to be an odd prime with $p\nmid N$. For $a\in \Z$ and $\xi$ a Dirichlet character mod $p$, we define the Gauss sum
$$
W(\xi,a)=\sum_{b\in(\Z/p\Z)^\times}\xi(b)e^{2\pi i abp^{-1}},
$$
and we let $W(\xi)=W(\xi,1)$.   Throughout this work, $\chi$ is the nontrivial quadratic Dirichlet character mod $p$.  Note that $\chi=\big(\frac{\cdot}{p}\big)$ is the Legendre symbol, and $W(\chi,a)=\chi(a)W(\chi)$.  For $S$ as in \eqref{A(N)definition}, we let 
\begin{equation}\label{fSpolynomialeq}
f_S(X)=\alpha p^{-4}X^2-2\beta p^{-2}X+\gamma.
\end{equation}
\section{Results}
In this section, we present the main results of this paper.  The proofs appear in Section \ref{maintheoremproof}.
\begin{theorem}\label{maintheorem}
Let $N$ and $k$ be positive integers, $p$ an odd prime with $p\nmid N$, and $\chi$ the nontrivial quadratic Dirichlet character mod $p$, i.e. $\chi=\left(\frac{\cdot}{p}\right)$. Let $\mathcal{T}_\chi:S_k(\Gamma^{\mathrm{para}}(N))\rightarrow S_k(\Gamma^{\mathrm{para}}(Np^4))$ be the twisting map from Theorem 3.1 of \cite{JR3}.  Let $F\in S_k(\Gamma^{\mathrm{para}}(N))$ have the Fourier expansion
$$
F(Z) = \sum_{S \in A(N)^+} a(S) e^{2\pi i \mathrm{tr} (SZ)}.
$$
Then the twist of $F$ has the Fourier expansion
$$
\mathcal{T}_\chi(F)(Z)=\sum_{S\in A(Np^4)^+}W(\chi)a_\chi(S) e^{2\pi i \mathrm{tr}(SZ)}
$$
where $a_\chi(S)$ for $S=\begin{bmatrix}\alpha&\beta\\ \beta&\gamma\end{bmatrix}$ is given as follows:
\renewcommand{\theenumi}{\roman{enumi}}%
\begin{enumerate}
\item If $p\nmid 2\beta$ and $p^4\mid\alpha$, then 
\begin{equation}\label{case1eq}
a_\chi(S)=p^{1-k}\chi(2\beta)\sum_{\substack{b\in(\Z/p\Z)^\times}}\chi(b)a(S[\begin{bmatrix}1&-bp^{-1}\\&p\end{bmatrix}]).
\end{equation}
\item If $p\mid\mid2\beta$ and $p^4\mid\mid \alpha$, then
\begin{align}\nonumber
a_\chi(S)&=p^{-1}\sum_{\substack{a,b\in(\Z/p\Z)^\times}}\chi\big(ab(2\beta p^{-1}a -\gamma)\big)a(S[\begin{bmatrix}1&-(a+b)p^{-1}\\&1\end{bmatrix}])\\\nonumber
 &+p^{-1}\sum_{\substack{a,z\in(\Z/p\Z)^\times\\z\not\equiv1(p)}} \chi\big(az(1-z)(az\alpha p^{-4}-2\beta p^{-1})\big)a(S[\begin{bmatrix}p^{-1}&-ap^{-2}\\&p\end{bmatrix}])\\\nonumber
&-p^{-1}\chi(\alpha p^{-4})a(S[\begin{bmatrix}p^{-2}&\\&p^2\end{bmatrix}])+p^{k-2}\chi(-\alpha p^{-4})a(S[\begin{bmatrix}p^{-1}&-xp^{-3}\\&1\end{bmatrix}]) \\
&+p^{k-2}\chi(-\gamma)a(S[\begin{bmatrix}1&yp^{-2}\\&p^{-1}\end{bmatrix}]),\label{case2eq}
\end{align}
where $x\alpha p^{-4}\equiv2\beta p^{-1}(p^2)$ and $y2\beta p^{-1}\equiv -\gamma(p^2)$.\\
\item If $p\mid\mid 2\beta$ and $p^5\mid\alpha$, then
\begin{align}\nonumber
a_\chi(S) =&p^{-1}\sum_{\substack{a,b\in(\Z/p\Z)^\times}}\chi\big(ab( 2\beta p^{-1}a-\gamma)\big)a(S[\begin{bmatrix}1&-(a+b)p^{-1}\\&1\end{bmatrix}])\\
 &+p^{k-2}\chi(-\gamma)a(S[\begin{bmatrix}1&yp^{-2}\\&p^{-1}\end{bmatrix}])-p^{-1}\chi(2\beta p^{-1})\sum_{\substack{a\in(\Z/p\Z)^\times}}\chi(a)a(S[\begin{bmatrix}p^{-1}&-ap^{-2}\\&p\end{bmatrix}]),\label{case3eq}
 \end{align}
 where $y\in(\Z/p^2\Z)^\times$ is such that $2\beta p^{-1}y\equiv -\gamma(p^2).$\\
\item If $p^2\mid 2\beta$ and $p^4\mid\mid\alpha$, then
\begin{align}\label{case4eq}
a_\chi(S)&=b_\chi(S)+c_\chi(S)+d_\chi(S)
\end{align}
where
 \begin{align*}
 b_\chi(S) &=(1-p^{-1})\chi(\gamma)a(S)-p^{-1}\chi( \gamma)\sum_{\substack{b\in(\Z/p\Z)^\times}}a(S[\begin{bmatrix}1&-bp^{-1}\\&1\end{bmatrix}])\\
 &+p^{k-3}\sum_{\substack{b,x,y\in(\Z/p\Z)^\times\\x,y\not\equiv1(p)}}\chi\big(y(1-x)(\alpha p^{-4} \big(\frac{y-x}{1-y}\big)b^2 +2\beta  p^{-2}b -\gamma x^{-1})\big)a(S[\begin{bmatrix}p^{-1}&-bp^{-2}\\&1\end{bmatrix}])\\
&+p^{k-3}\big(\chi(\gamma)+p\chi(-4\det(S)p^{-4}\gamma)-\chi( -\alpha p^{-4})W(\mathbf{1},2\beta p^{-2})\big) a(S[\begin{bmatrix}p^{-1}&\\&1\end{bmatrix}])\\
&+p^{k-3}\chi(-\alpha p^{-4})\sum_{x\in\Z/p\Z}W(\mathbf{1},2\beta p^{-2}-x\alpha p^{-4})a(S[\begin{bmatrix}p^{-1}&-xp^{-2}\\&1\end{bmatrix})\\
&+p^{2k-4}\sum\limits_{\substack{b\in(\Z/p^2\Z)^\times\\p^2\mid(\alpha p^{-4}b^2-2\beta p^{-2}b+\gamma)}}\sum\limits_{\substack{z\in(\Z/p\Z)^\times\\z\not\equiv1(p)}}\chi\big(z(1-z)(\gamma -z\alpha p^{-4}b^2)\big)a(S[\begin{bmatrix}p^{-1}&-bp^{-3}\\&p^{-1}\end{bmatrix}])\\
&-p^{-1}\chi(\alpha p^{-4})\sum_{\substack{a\in(\Z/p\Z)^\times}} a(S[\begin{bmatrix}p^{-1}&-ap^{-2}\\&p\end{bmatrix}])\\
&+p^{k-3}\chi(\alpha p^{-4})(p\chi(-4\det(S)p^{-4})-W(\mathbf{1},\gamma))a(S[\begin{bmatrix}p^{-2}&\\&p\end{bmatrix}])\\
&+(1-p^{-1})\chi(\alpha p^{-4})a(S[\begin{bmatrix}p^{-2}&\\&p^2\end{bmatrix}]),
 \end{align*}

\begin{align*}
c_\chi(S)=&\begin{cases}p^{2k-4}\chi(\alpha p^{-4})W(\mathbf{1},4\det(S)p^{-5})a(S[\begin{bmatrix}p^{-2}&\\&1\end{bmatrix}])& \text{if}\,\, p^5\mid \det(S)\,\, \text{and}\,\,\chi(\gamma\alpha p^{-4})=1\\
0 &\text{otherwise,}
\end{cases}
\end{align*}
and
\begin{align*}
d_\chi(S)&=\begin{cases}p^{3k-5}\chi(\alpha p^{-4})W(\mathbf{1},4p^{-6}\det(S))a(S[\begin{bmatrix}p^{-2}&-ap^{-3}\\&p^{-1}\end{bmatrix}])&\begin{array}[t]{l}\text{if}\,\,p^6\mid4\det(S)\\\text{and}\,\,p^8\nmid4\det(S)\end{array}\\
\begin{array}[t]{l}
\chi(\alpha p^{-4})\Big(p^{3k-5}(p-1)a(S[\begin{bmatrix}p^{-2}&-ap^{-3}\\&p^{-1}\end{bmatrix}])\\
\,\,+p^{4k-6}a(S[\begin{bmatrix}p^{-2}&-bp^{-4}\\&p^{-2}\end{bmatrix}])\Big)
\end{array}&\text{if}\,\,p^8\mid 4\det(S)\\
0&\text{otherwise},
\end{cases}
\end{align*}
where $a\in (\Z/p\Z)^\times$ satisfies $2\alpha p^{-4}a\equiv2\beta p^{-2}(p)$ and $b\in (\Z/p^2\Z)^\times$ satisfies $2\alpha p^{-4}b\equiv2\beta p^{-2}(p^2)$. (If $a$ or $b$ does not exist, then the corresponding term is 0.)\\
\item If $p^2\mid2\beta$ and $p^5\mid\alpha$, then
\begin{align}\nonumber
a_\chi(S) &=(1-p^{-1})\chi(\gamma)a(S)-p^{-1}\chi( \gamma)\sum_{\substack{b\in(\Z/p\Z)^\times}}a(S[\begin{bmatrix}1&-bp^{-1}\\&1\end{bmatrix}])\\
\nonumber
 &-p^{k-3}\sum_{b\in(\Z/p\Z)^\times}\sum_{\substack{x\in(\Z/p\Z)^\times\\x\not\equiv1(p)}}\chi\big((1-x) (2\beta  bp^{-2}-\gamma x^{-1})\big)a(S[\begin{bmatrix}p^{-1}&-bp^{-2}\\&1\end{bmatrix}])\\\nonumber
&+p^{k-3}\chi(\gamma)(1-p+p\chi(-4\det(S)p^{-4}))a(S[\begin{bmatrix}p^{-1}&\\&1\end{bmatrix}])\\
&-p^{2k-4}\chi( -\gamma)\sum\limits_{\substack{b\in(\Z/p^2\Z)^\times\\p^2\mid(\alpha p^{-4}b^2-2\beta p^{-2}b+\gamma)}}a(S[\begin{bmatrix}p^{-1}&-bp^{-3}\\&p^{-1}\end{bmatrix}]).\label{case5eq}
\end{align}
\end{enumerate}
\end{theorem}
As mentioned in the introduction, it is easy to verify that $\mathcal{T}_\chi$ is not identically zero.  For example, let $N=1$, $p=3$, and let $F$ be the Siegel cusp form $\Upsilon20$ in \cite{LMFDB}.  Then using \eqref{case1eq} the coefficient $a_\chi([\begin{smallmatrix}81&22\\22&6\end{smallmatrix}])$ is calculated to be nonzero as follows:
\begin{align*}
a_\chi(\begin{bmatrix}81&22\\22&6\end{bmatrix})&=3^{-19}\chi(44)\big(a(\begin{bmatrix}81&39\\39&19\end{bmatrix})-a(\begin{bmatrix}81&12\\12&2\end{bmatrix})\big)\\
&=-3^{-19}\big(a(\begin{bmatrix}1&0\\0&18\end{bmatrix})-a(\begin{bmatrix}2&0\\0&9\end{bmatrix})\big)\\
&=-3^{-19}(2256995864880+4329978670800)\\
&=-\frac{2256995864880}{1162261467}.
\end{align*}

On the other hand, when restricted to the Maass space, $\mathcal{T}_\chi$ is identically zero, as asserted by the following corollary.
\begin{corollary}\label{corollary}
Let the notation be as in Theorem \ref{maintheorem}, and assume that $k$ is even and $N=1$.  If $F$ is in the Maass space as defined in \cite{Ma}, then $\mathcal{T}_\chi(F)=0$.
\end{corollary}
The proof of this corollary is a calculation using the formulas given in Theorem \ref{maintheorem}.  This result is expected as a consequence of the  nature of the twisting map.  Indeed, we can give a representation theoretic proof of the corollary as follows.
Let $F$ be as in the corollary, and let $V$ be the subspace of the cuspforms on the ad\`eles of $\GSp(4)/\Q$ generated by $F$.  Then $V$ is a direct sum of finitely many irreducible subspaces, $V_i$.  We have $V_i\simeq\bigotimes_{v}\pi_{i,v}$ where $\pi_{i,v}$ is an irreducible, admissible representation of $\GSp(4,\Q_v)$ with trivial central character, and for finite $v$, $\pi_{i,v}$ is unramified of Saito-Kurokawa type (\cite{RS} Section 5.5, \cite{RS2}, \cite {Sc}).  By the construction in \cite{JR2} and \cite{JR3}, $\mathcal{T}_\chi(F)$ lies in a space that is the direct sum of irreducible subspaces isomorphic to $\bigotimes_v(\pi_{i,v}\otimes\chi_v)$.  Here, we regard $\chi$ as a character of the id\`eles as in \cite{JR3}.  However, by Lemma 5.5.2 of \cite{RS}, we see that $\pi_{i,p}\otimes\chi_p$ is not paramodular, and hence $\mathcal{T}_\chi(F)=0$.  In our view, the explicit proof of the corollary is strong evidence for the accuracy of the formulas given in the theorem.

\section{Proofs of the Theorem and Corollary}
We begin with a lemma that will be useful in evaluating character sums.
\label{maintheoremproof}
\begin{lemma}
\label{jlemma}
Let $A,B,C \in \Z$ and assume that $A \not \equiv 0\, (p)$ or $B \not \equiv 0\, (p)$. Set $D=B^2-4AC$. Then
$$
\sum_{x \in \Z/p\Z} \chi(Ax^2+Bx+C)= 
\begin{cases}
(p-1)\chi(A)&\text{if $D\equiv 0\, (p)$}\\
-\chi(A)&\text{if $D\not\equiv 0\, (p)$}.
\end{cases}
$$
In particular, if $a_1,a_2,b_1,b_2 \in \Z$ with $a_1 \not \equiv 0\, (p)$ and $a_2 \not \equiv 0\, (p)$, then
$$
\sum_{x \in \Z/p\Z} \chi(a_1x+b_1)\chi(a_2x+b_2)= 
\begin{cases}
(p-1)\chi(a_1a_2)&\text{if $a_1b_2 \equiv a_2b_1\, (p)$}\\
-\chi(a_1a_2)&\text{if $a_1b_2 \not\equiv a_2b_1\, (p)$}.
\end{cases}
$$
\end{lemma}
\begin{proof}
Assume first that $A \equiv 0\, (p)$; by assumption, $B \not \equiv 0\, (p)$.  We have
$$
\sum_{x \in \Z/p\Z} \chi(Ax^2+Bx+C)= \sum_{x \in \Z/p\Z} \chi(Bx+C)= \sum_{x \in \Z/p\Z} \chi(x)= 0,
$$
as claimed. 

Now assume that $A \not \equiv 0\, (p)$, so that we may consider $A\in(\Z/p\Z)^\times$. Then
$$
\sum_{x \in \Z/p\Z} \chi(Ax^2+Bx+C)= \chi(A) \sum_{x \in \Z/p\Z} \chi(x^2+A^{-1}Bx+A^{-1}C).
$$
Since $B^2-4AC \equiv 0\, (p)$ if and only if $(A^{-1}B)^2-4(A^{-1}C)\equiv 0\, (p)$ we may assume for the remainder of the proof that $A\equiv 1\,(p)$. 

Assume that $D\equiv 0\, (p)$. We have
$$
\sum_{x \in \Z/p\Z} \chi(x^2+Bx+C)
=\sum_{x \in \Z/p\Z} \chi((x+B/2)^2)
=p-1.
$$
Assume that $D\not\equiv 0\, (p)$ and there exists $s \in \Z$ such that $(s,p)=1$ and $s^2 \equiv B^2-4C\, (p)$. Then
\begin{align*}
\sum_{x \in \Z/p\Z} \chi(x^2+Bx+C)
&=\sum_{x \in \Z/p\Z} \chi((x+B/2)^2-(s/2)^2)\\
&=\sum_{x \in \Z/p\Z} \chi(x+B/2+s/2) \chi(x+B/2-s/2) \\
&=\sum_{x \in \Z/p\Z} \chi(x) \chi(x-s) \\
&=\sum_{x \in (\Z/p\Z)^\times} \chi(x) \chi(x-s) \\
&=\sum_{x \in (\Z/p\Z)^\times} \chi(1-sx^{-1}) \\
&=\sum_{x \in (\Z/p\Z)^\times} \chi(1-sx) \\
&=\sum_{x \in (\Z/p\Z)^\times} \chi(1-x) \\
&=-1+\sum_{x \in \Z/p\Z} \chi(1-x) \\
&=-1.
\end{align*}
Assume that $D \not\equiv 0\, (p)$ and $B^2-4C$ is not a square mod $p$.   We have
\begin{align*}
\sum_{x \in \Z/p\Z} \chi(x^2+Bx+C)
&=\sum_{x \in \Z/p\Z} \chi((x+B/2)^2 -D/4)\\
&=\sum_{x \in \Z/p\Z} \chi(x^2 -D/4)\\
&=\chi(-D/4)+2\sum_{y \in (\Z/p\Z)^{\times 2} } \chi(y -D/4)\\
&=\chi(-D/4)+2\sum_{y \in (\Z/p\Z)^\times } \chi(y -D/4) -2\sum_{y \in (D/4)(\Z/p\Z)^{\times 2} } \chi(y -D/4)\\
&=\chi(-D/4)-2\chi(-D/4)+2\sum_{y \in \Z/p\Z} \chi(y -D/4) -2\!\!\!\!\sum_{y \in (D/4)(\Z/p\Z)^{\times 2} }\!\!\! \chi(y -D/4)\\
&=-\chi(-D/4) -2\sum_{y \in (D/4)(\Z/p\Z)^{\times 2} } \chi(y -D/4)\\
&=-\chi(-D/4) -2\sum_{y \in (\Z/p\Z)^{\times 2} } \chi(D/4) \chi(y -1)\\
&=-\chi(-D/4) - \chi(D/4)\sum_{x \in (\Z/p\Z)^{\times } } \chi(x^2 -1)\\
&=-\chi(-D/4) - \chi(D/4)\big( -\chi(-1)+ \sum_{x \in \Z/p\Z} \chi(x^2 -1)\big) \\
&=-\chi(-D/4) - \chi(D/4)\big( -\chi(-1) -1\big) \\
&=\chi(D/4) \\
&=-1.
\end{align*}
This completes the proof. 
\end{proof}
\begin{proof}[Proof of Theorem \ref{maintheorem}]
Let $F \in S_k(\Gamma^{\mathrm{para}}(N))$ with Fourier expansion given in \eqref{fouriereq}. In \cite{JR3}  Theorem 3.1 it was proven that $\mathcal{T}_\chi F = \sum_{l=1}^{14} F|_k\mathcal{T}_\chi^l$, where each $\mathcal{T}_\chi^l$ is an explicit finite formal sum of the form
$$
\mathcal{T}_\chi^l=\sum_{i \in I, j \in J} c_i d_j \begin{bmatrix} 1 & Q_{ij} \\ & 1 \end{bmatrix} \begin{bmatrix} P_i & \\ & {}^tP_i^{-1} \end{bmatrix}. 
$$
Let $Z \in \mathfrak{H}_2$. It follows that each $(F|_k \mathcal{T}_\chi^l)(Z)$ has the form
\begin{align}\label{lpartexpansioneq}
(F|_k \mathcal{T}_\chi^l)(Z) =& \sum_{S \in A(N)^+} a(S) \sum_{i \in I} c_i \det(P_i)^k \big( \sum_{j \in J} d_j e^{2\pi i \mathrm{tr}(SQ_{ij})} \big) e^{2\pi i\mathrm{tr}({}^tP_iSP_iZ)}\\
=&\nonumber \sum_{S \in A(Np^4)^+} a_l(S) e^{2\pi i \mathrm{tr}(SZ)}.
\end{align}
The $a_l(S)$ for each $l\in\{1,\dots,14\}$ and $S$ as in \eqref{A(N)definition} are computed as follows:

\noindent{\bf The case $l=1$:}  From \cite{JR3} Theorem 3.1, we have
\begin{multline*}
(F|_k\mathcal{T}_\chi^1)(Z)=p^{-11} \sum_{S \in A(N)^+} a(S) \sum_{\substack{a,b,x\in(\Z/p^3\Z)^\times}}\chi(ab)\\ \big( \sum_{z\in\Z/p^4\Z} e^{2\pi i \mathrm{tr}(S\begin{bmatrix}zp^{-4}&-bp^{-2}\\-bp^{-2}&-x^{-1}p^{-1}\end{bmatrix})} \big) e^{2\pi i\mathrm{tr}(S[\begin{bmatrix}1&(a+xb)p^{-1}\\&1\end{bmatrix}]Z)}.
\end{multline*}
 Then the inner sum is calculated as
\begin{align*}
\sum_{z\in\Z/p^4\Z}e^{2\pi i (-\gamma p^{-1}x^{-1}-2b\beta p^{-2}+\alpha z p^{-4})}
&=\begin{cases} p^{4}e^{-2\pi i (\gamma px^{-1}+2b\beta)p^{-2}}&\text{if } p^4\mid \alpha,\\ 
0&\text{otherwise.}\end{cases}
\end{align*}
Now, $(F|_k \mathcal{T}_\chi^1)(Z)$ is computed as follows:
\begin{align*}
 &p^{-7}\sum_{S\in A(Np^4)^+}\sum_{\substack{a,b,x\in(\Z/p^3\Z)^\times}}a(S)\chi(ab)e^{-2\pi i (\gamma px^{-1}+2b\beta)p^{-2}}e^{2\pi i \trace (S[\begin{bmatrix}1&(a+bx)p^{-1}\\&1\end{bmatrix} ]Z)}\\
 &= p^{-7}\sum_{S\in A(Np^4)^+}\sum_{\substack{a,b,x\in(\Z/p^3\Z)^\times}}a(S)\chi(abx)e^{-2\pi i (\gamma p+2b\beta)x^{-1}p^{-2}}e^{2\pi i \trace (S[\begin{bmatrix}1&(a+b)p^{-1}\\&1\end{bmatrix} ]Z)}\\
  &= p^{-7}\sum_{S\in A(Np^4)^+}\sum_{\substack{a,b\in(\Z/p^3\Z)^\times\\x_0\in(\Z/p\Z)^\times}}\sum_{x_1\in\Z/p^2\Z}a(S)\chi(abx_0)e^{-2\pi i (\gamma p+2b\beta)(x_0+px_1)p^{-2}}e^{2\pi i \trace (S[\begin{bmatrix}1&(a+b)p^{-1}\\&1\end{bmatrix} ]Z)}\\
&= p^{-5}\sum_{\substack{S\in A(Np^4)^+\\p\mid 2\beta}}\sum_{\substack{a,b\in(\Z/p^3\Z)^\times}}a(S)\chi(ab)W(\chi, -\gamma-2\beta p^{-1}b)e^{2\pi i \trace (S[\begin{bmatrix}1&(a+b)p^{-1}\\&1\end{bmatrix} ]Z)}.\\
\end{align*}
First interchanging the order of summation, consider the map $S\mapsto {}^tPSP$, for $P=\left[\begin{smallmatrix}1&(a+b)p^{-1}\\&1\end{smallmatrix}\right]$.  Evidently, this map is a bijection from $\{S\in A(Np^4)^+ : p\mid 2\beta\}$ to itself.  Continuing the calculation, we have
\begin{align*}
&=p^{-5}\sum_{\substack{a,b\in(\Z/p^3\Z)^\times}}\sum_{\substack{S\in A(Np^4)^+\\p\mid 2\beta}}a(S)\chi(ab)W(\chi, -\gamma-2\beta p^{-1}b)e^{2\pi i \trace (S[\begin{bmatrix}1&(a+b)p^{-1}\\&1\end{bmatrix} ]Z)}\\
&=p^{-5}\sum_{\substack{a,b\in(\Z/p^3\Z)^\times}}\sum_{\substack{S\in A(Np^4)^+\\p\mid 2\beta}}a(S[\begin{bmatrix}1&-(a+b)p^{-1}\\&1\end{bmatrix}])\chi(ab)\\
&\qquad W(\chi, -\gamma+2\beta p^{-1}(a+b)-\alpha(a+b)^2p^{-2}-2(\beta-\alpha(a+b)p^{-1})p^{-1}b)e^{2\pi i \trace (SZ)}\\
&=p^{-1}\sum_{\substack{S\in A(Np^4)^+\\p\mid 2\beta}}\sum_{\substack{a,b\in(\Z/p\Z)^\times}}a(S[\begin{bmatrix}1&-(a+b)p^{-1}\\&1\end{bmatrix}])\chi(ab)W(\chi, -\gamma+2\beta p^{-1}a)e^{2\pi i \trace (SZ)}.
\end{align*}
\noindent{\bf The case $l=2$:}  
From \cite{JR3} Theorem 3.1, we have
\begin{multline*}
(F|_k\mathcal{T}_\chi^2)(Z)=p^{k-11} \sum_{S \in A(N)^+} a(S) \sum_{b\in(\Z/p^3\Z)^\times}\chi(b)\\ \big( \!\!\!\sum_{\substack{a,x,y\in(\Z/p^3\Z)^\times\\x,y\not\equiv1(p)}}\!\!\!\!\chi(axy) e^{2\pi i \mathrm{tr}(S\begin{bmatrix}-ab(1-(1-y)^{-1}x)p^{-3}&-ap^{-2}  \\-ap^{-2}& -ab^{-1}(1-x)^{-1}p^{-1}  \end{bmatrix})} \big) e^{2\pi i\mathrm{tr}(S[\begin{bmatrix}p&bp^{-1}\\&1\end{bmatrix}]Z)}.
\end{multline*}
The inner sum is calculated as
\begin{align*}
&\sum_{\substack{a,x,y\in(\Z/p^3\Z)^\times\\x,y\not\equiv1(p)}}\chi(axy) e^{2\pi i \big(-a\gamma b^{-1}p^{-1}(1-x)^{-1}-2a\beta p^{-2}-ab(1-x(1-y)^{-1})\alpha p^{-3}\big)}\\
=&\sum_{\substack{a_0\in(\Z/p\Z)^\times\\x,y\in(\Z/p^3\Z)^\times\\x,y\not\equiv1(p)}}\sum_{a_1\in\Z/p^2\Z}\chi(a_0xy) e^{2\pi i \big(-\gamma b^{-1}(1-x)^{-1}p^2-2\beta p-b(1-x(1-y)^{-1})\alpha\big)(a_0+pa_1)p^{-3}}.\\
\intertext{We see that this sum vanishes unless $p^2\mid-2\beta p-b(1-x(1-y)^{-1})\alpha$, which implies that $p\mid\alpha$. Then our sum is}
=&p^{2}\sum\limits_{\substack{a_0\in(\Z/p\Z)^\times\\x,y\in(\Z/p^3\Z)^\times\\x,y\not\equiv1(p)\\p^2\mid-2\beta p-b(1-x(1-y)^{-1})\alpha}}\chi(a_0xy) e^{2\pi i \big(-\gamma b^{-1}(1-x)^{-1}p^2-2\beta p-b(1-x(1-y)^{-1})\alpha\big)a_0 p^{-3}}\\
=&p^{2}\sum\limits_{\substack{a_0,x_0\in(\Z/p\Z)^\times\\x_1\in\Z/p^2\Z\\y\in(\Z/p^3\Z)^\times\\x_0,y\not\equiv1(p)\\p^2\mid-2\beta p-b(1-(x_0+px_1)(1-y)^{-1})\alpha}}\chi(a_0x_0y) e^{2\pi i \big(-\gamma b^{-1}(1-x_0)^{-1}p^2-2\beta p-b(1-(x_0+px_1)(1-y)^{-1})\alpha\big)a_0 p^{-3}}.\\
\intertext{We see that this sum vanishes unless $p^2\mid\alpha$.  Our two conditions are equivalent to the conditions that $p^2\mid\alpha$ and $p\mid 2\beta$.  Assume this.  Continuing the calculation, we find that our sum is}
=&p^{4}\sum\limits_{\substack{a_0,x_0\in(\Z/p\Z)^\times\\y\in(\Z/p^3\Z)^\times\\x_0,y\not\equiv1(p)}}\chi(a_0x_0y) e^{2\pi i \big(-\gamma b^{-1}(1-x_0)^{-1}p^2-2\beta p-b(1-x_0(1-y)^{-1})\alpha\big)a_0 p^{-3}}\\
=&p^{6}\sum\limits_{\substack{x,y\in(\Z/p\Z)^\times\\x,y\not\equiv1(p)}}\chi(xy) W(\chi,-\gamma b^{-1}(1-x)^{-1}-2\beta p^{-1}-b(1-x(1-y)^{-1})\alpha p^{-2}).
\end{align*}

We now calculate the full sum, 
\begin{multline*}
(F|_k \mathcal{T}_\chi^2)(Z)=p^{k-5}\sum_{\substack{S\in A(Np^2)^+\\p\mid2\beta}}\sum_{\substack{b\in(\Z/p^3\Z)^\times\\x,y\in(\Z/p\Z)^\times\\x,y\not\equiv1(p)}}a(S)\chi(bxy) \\W(\chi, -\gamma b^{-1}(1-x)^{-1}-2\beta p^{-1}-b(1-x(1-y)^{-1})\alpha p^{-2}) e^{2\pi i\trace(S[\begin{bmatrix}p&bp^{-1}\\&1\end{bmatrix}]Z)}.
\end{multline*}
First interchanging the order of summation, consider the map $S\mapsto {}^tPSP$, for $P=\left[\begin{smallmatrix}p&bp^{-1}\\&1\end{smallmatrix}\right]$.  Evidently, this map is a bijection between $\{S\in A(Np^2)^+ : p\mid2\beta\}$ and $\{S\in A(Np^4)^+: p^2\mid 2\beta\}$.  Continuing the calculation, we have
\begin{align*}
&=p^{k-5}\sum_{\substack{b\in(\Z/p^3\Z)^\times\\x,y\in(\Z/p\Z)^\times\\x,y\not\equiv1(p)}}\sum_{\substack{S\in A(Np^4)^+\\p^2\mid2\beta}}a(S[\begin{bmatrix}p^{-1}&-bp^{-2}\\&1\end{bmatrix}])\chi(bxy) e^{2\pi i\trace(SZ)}\\
&\qquad W(\chi, -(\gamma-b2\beta p^{-2}+b^2\alpha p^{-4}) b^{-1}(1-x)^{-1}-2(\beta p^{-2}-b\alpha p^{-4})-b(1-x(1-y)^{-1})\alpha p^{-4})\\
=&p^{k-5}\sum_{\substack{S\in A(Np^4)^+\\p^2\mid2\beta}}\sum_{\substack{b\in(\Z/p^3\Z)^\times\\x,y\in(\Z/p\Z)^\times\\x,y\not\equiv1(p)}}a(S[\begin{bmatrix}p^{-1}&-bp^{-2}\\&1\end{bmatrix}])\chi(bxy) e^{2\pi i\trace(SZ)}\\
&\qquad W(\chi, (1-x)^{-1}(-b^{-1}\gamma+x2\beta p^{-2}+x(1-y)^{-1}(y-x)b\alpha p^{-4} ))\\
=&p^{k-3}\sum_{\substack{S\in A(Np^4)^+\\p^2\mid2\beta}}\sum_{\substack{b,x,y\in(\Z/p\Z)^\times\\x,y\not\equiv1(p)}}a(S[\begin{bmatrix}p^{-1}&-bp^{-2}\\&1\end{bmatrix}])\chi((1-x)y) \\
&\qquad W(\chi, -\gamma x^{-1}+2\beta  bp^{-2}+\alpha (1-y)^{-1}(y-x)b^2 p^{-4} )e^{2\pi i\trace(SZ)}.
\end{align*}
\noindent{\bf The case $l=3$:} 
From \cite{JR3} Theorem 3.1, we have 
\begin{multline*}
(F|_k\mathcal{T}_\chi^3(Z)=p^{k-6}\sum_{S\in A(N)^+}a(S)\\\big(\sum_{\substack{a\in(\Z/p^2\Z)^\times\\ b\in(\Z/p^3\Z)^\times\\ z\in(\Z/p\Z)^\times\\z\not\equiv1(p)}}\chi(b(1-z))e^{2\pi i \mathrm{tr}(S\begin{bmatrix}
-bp^{-3}&ap^{-2}\\
ap^{-2}&-a^2b^{-1}zp^{-1}
\end{bmatrix})}\big)e^{2\pi i \mathrm{tr}(S[\begin{bmatrix}p&\\&1\end{bmatrix}]Z)}.
\end{multline*}
The inner sum is calculated as
\begin{align*}
&\sum_{\substack{a\in(\Z/p^2\Z)^\times\\ b\in(\Z/p^3\Z)^\times\\ z\in(\Z/p\Z)^\times\\z\not\equiv1(p)}}\chi(b(1-z))e^{2\pi i \big(- a^2b^{-1}z\gamma p^2+2a\beta p-b\alpha\big)p^{-3}}\\
&=\sum_{\substack{a_1\in\Z/p\Z\\ b_1\in\Z/p^2\Z\\a_0,b_0, z\in(\Z/p\Z)^\times\\z\not\equiv1(p)}}\chi(b_0(1-z))e^{2\pi i \big(- a_0^2b_0^{-1}z\gamma p^2+2(a_0+pa_1)\beta p-(b_0+pb_1)\alpha\big)p^{-3}}\\
\intertext{This is zero unless $p\mid 2\beta$ and $p^2\mid\alpha$.  Assume this.  Our sum is}
&=p^{3}\sum\limits_{\substack{a, b,z\in(\Z/p\Z)^\times\\z\not\equiv1(p)}}\chi(b(1-z))e^{2\pi i(- a^2z\gamma +2a\beta p^{-1}-\alpha p^{-2})bp^{-1}}\\
&=p^{3}\sum\limits_{\substack{a, b,z\in(\Z/p\Z)^\times\\z\not\equiv1(p)}}\chi(bz)e^{2\pi i(- a^2(1-z)\gamma +2a\beta p^{-1}-\alpha p^{-2})bp^{-1}}\\
&=p^{3}\sum\limits_{\substack{a, b\in(\Z/p\Z)^\times}}\chi(b)\big((\sum_{z\in(\Z/p\Z)^\times}\chi(z)e^{2\pi ia^2z\gamma bp^{-1}}) -e^{2\pi i a^2\gamma bp^{-1}}\big)e^{2\pi i(- a^2\gamma +2a\beta p^{-1}-\alpha p^{-2})bp^{-1}}\\
&=p^{3}\big(W(\chi,\gamma)\sum\limits_{\substack{a, b\in(\Z/p\Z)^\times}}e^{2\pi i(- a^2\gamma +2a\beta p^{-1}-\alpha p^{-2})bp^{-1}}\big)-p^3\big(\sum\limits_{\substack{a, b\in(\Z/p\Z)^\times}}\chi(b)e^{2\pi i(2a\beta p^{-1}-\alpha p^{-2})bp^{-1}}\big)\\
&=p^{3}\big(W(\chi,\gamma)\sum\limits_{\substack{a\in(\Z/p\Z)^\times}}W(\mathbf{1},- a^2\gamma +2a\beta p^{-1}-\alpha p^{-2})\big)-p^3W(\mathbf{1},2\beta p^{-1})W(\chi,-\alpha p^{-2}).\\
\end{align*}
We now calculate the full sum
\begin{multline*}
(F|_k\mathcal{T}_\chi^3(Z)=p^{k-3}\sum_{\substack{S\in A(Np^2)^+\\p\mid2\beta}}a(S)\Big(\big(W(\chi,\gamma)\sum\limits_{\substack{a\in(\Z/p\Z)^\times}}W(\mathbf{1},- a^2\gamma +2a\beta p^{-1}-\alpha p^{-2})\big)\\-W(\mathbf{1},2\beta p^{-1})W(\chi,-\alpha p^{-2})\Big)e^{2\pi i \mathrm{tr}(S[\begin{bmatrix}p&\\&1\end{bmatrix}]Z)}.
\end{multline*}
Now, consider the map $S\mapsto {}^tPSP$, for $P=\left[\begin{smallmatrix}p&\\&1\end{smallmatrix}\right]$.  Evidently, this map is a bijection between $\{S\in A(Np^2)^+ : p\mid2\beta\}$ and $\{S\in A(Np^4)^+: p^2\mid 2\beta\}$.  Continuing the calculation, we have
\begin{align*}
&=p^{k-3}\sum_{\substack{S\in A(Np^2)^+\\p\mid2\beta}}a(S)\Big(\big(W(\chi,\gamma)\sum\limits_{\substack{a\in(\Z/p\Z)^\times}}W(\mathbf{1},- a^2\gamma +2a\beta p^{-1}-\alpha p^{-2})\big)\\
&\qquad\qquad-W(\mathbf{1},2\beta p^{-1})W(\chi,-\alpha p^{-2})\Big)e^{2\pi i \mathrm{tr}(S[\begin{bmatrix}p&\\&1\end{bmatrix}]Z)}\\
&=p^{k-3}\sum_{\substack{S\in A(Np^4)^+\\p^2\mid2\beta}}a(S[\begin{bmatrix}p^{-1}&\\&1\end{bmatrix}])\Big(\big(W(\chi,\gamma)\sum\limits_{\substack{a\in(\Z/p\Z)^\times}}W(\mathbf{1},- a^2\gamma +2a\beta p^{-2}-\alpha p^{-4})\big)\\
&\qquad\qquad-W(\mathbf{1},2\beta p^{-2})W(\chi,-\alpha p^{-4})\Big)e^{2\pi i \mathrm{tr}(SZ)}.
\end{align*}
We calculate 
\begin{multline*}
W(\chi,\gamma)\sum\limits_{\substack{a\in(\Z/p\Z)^\times}}W(\mathbf{1},- a^2\gamma +2a\beta p^{-2}-\alpha p^{-4})\\
=W(\chi,\gamma)(-W(\mathbf{1},-\alpha p^{-4})+\sum\limits_{\substack{a\in\Z/p\Z}}W(\mathbf{1},- a^2\gamma +2a\beta p^{-2}-\alpha p^{-4})).
\end{multline*}
  Note that this is zero if $p\mid\gamma$; assume that $p\nmid\gamma$.  We notice that the Gauss sum $W(\mathbf{1},- a^2\gamma +2a\beta p^{-2}-\alpha p^{-4})$ is $p-1$ if $a$ is a root of the polynomial $- x^2\gamma +2x\beta p^{-2}-\alpha p^{-4}$ in $x$ modulo $p$ and -1 otherwise.  Thus, we need only count the solutions for a given $S$ modulo $p$.  The discriminant of the polynomial is $-4\det(S)p^{-4}$.  Hence, there are no solutions when $-4\det(S)p^{-4}$ is not a square mod $p$, one solution when $p\mid(-4\det(S)p^{-4})$, and two solutions when $-4\det(S)p^{-4}$ is a square mod $p$ and $p\nmid-4\det(S)p^{-4}$.  Thus, the final answer is 
\begin{multline*}
p^{k-3}\sum_{\substack{S\in A(Np^4)^+\\p^2\mid2\beta}}a(S[\begin{bmatrix}p^{-1}&\\&1\end{bmatrix}])\Big(W(\chi,\gamma)\big(-W(\mathbf{1},-\alpha p^{-4})+p\chi(-4\det(S)p^{-4})\big)\\-W(\chi, -\alpha p^{-4})W(\mathbf{1},2\beta p^{-2})\Big)e^{2\pi i \trace(SZ)}.
\end{multline*}
\noindent{\bf The case $l=4$:}
From \cite{JR3} Theorem 3.1, we have that
\begin{multline*}
(F|_k\mathcal{T}_\chi^4)(Z)=\\p^{k-10}\sum_{S\in A(N)^+}a(S)\sum_{x\in (\Z/p^4\Z)^\times}
\big(\sum_{\substack{a\in\Z/p^4\Z\\ b\in(\Z/p^3\Z)^\times}} \chi(b)e^{2\pi i \mathrm{tr}(S\begin{bmatrix}(ax-bp)p^{-4}& ap^{-2} \\ap^{-2} &\end{bmatrix})}\big)e^{2\pi i \mathrm{tr}(S[\begin{bmatrix}p&xp^{-2}\\&1\end{bmatrix}]Z)}.
\end{multline*}
The inner sum is calculated as
\begin{align*}
&\sum_{\substack{a\in\Z/p^4\Z\\ b\in(\Z/p^3\Z)^\times}} \chi(b)
e^{2\pi i \big(a(x\alpha+2\beta p^2)-b\alpha p\big)p^{-4}}=\begin{cases}p^{6}W(\chi,-\alpha p^{-2})& \text{if}\, p^4\mid(x\alpha+2\beta p^2),\\
0& \text{otherwise}.\end{cases}
\end{align*}
Now, consider the map $S\mapsto {}^tPSP$, for $P=\left[\begin{smallmatrix}p&xp^{-2}\\&1\end{smallmatrix}\right]$, first interchanging the order of summation.  Evidently, this map is a bijection between $\{S\in A(N)^+:p^4\mid(x\alpha+2\beta p^2)\}$ and $\{S\in A(Np^4)^+:p\mid 2\beta\, \text{and}\, p^2\mid(2\beta p^{-1}-x\alpha p^{-4})\}$. Continuing the calculation, we have
\begin{align*}
=&p^{k-4}\sum_{x\in(\Z/p^4\Z)^\times}\sum\limits_{\substack{S\in A(N)^+\\p^4\mid(x\alpha+2\beta p^2)}}a(S)W(\chi,-\alpha p^{-2}) e^{2\pi i \trace(S[\begin{bmatrix}p&xp^{-2}\\&1\end{bmatrix}]Z)}\\
=&p^{k-4}\sum_{x\in(\Z/p^4\Z)^\times}\sum\limits_{\substack{S\in A_(Np^4)^+\\p\mid2\beta,\,p^2\mid(2\beta p^{-1}-x\alpha p^{-4})}}a(S[\begin{bmatrix}p^{-1}&-xp^{-3}\\&1\end{bmatrix}])W(\chi,-\alpha p^{-4}) e^{2\pi i \trace(SZ)}\\
=&p^{k-2}\sum\limits_{\substack{S\in A(Np^4)^+\\p\mid2\beta}}\,\,\sum_{\substack{x\in(\Z/p^2\Z)^\times\\x\alpha p^{-4}\equiv2\beta p^{-1}(p^2)}}a(S[\begin{bmatrix}p^{-1}&-xp^{-3}\\&1\end{bmatrix}])W(\chi,-\alpha p^{-4}) e^{2\pi i \trace(SZ)}.\\
\end{align*}
Note that there is only one such $x$ for each $S$.\\

\noindent{\bf The case $l=5$:}
From \cite{JR3} Theorem 3.1 we have that
\begin{multline*}
(F|_k\mathcal{T}_\chi^5)(Z)=\\p^{k-9}\sum_{S\in A(N)^+}a(S)\sum_{x\in \Z/p^3\Z}
\big(\sum_{\substack{a,b\in(\Z/p^3\Z)^\times}} \chi(b)e^{2\pi i \mathrm{tr}(S\begin{bmatrix}(ax-b)p^{-3}&ap^{-2} \\ap^{-2} &\end{bmatrix})}\big)e^{2\pi i \mathrm{tr}(S[\begin{bmatrix} p& xp^{-1}\\&1\end{bmatrix}]Z)}.
\end{multline*}
Reasoning as in the previous cases, we conclude that 
\begin{multline*}
(F|_k\mathcal{T}_\chi^5)(Z)=\\p^{k-3}\sum_{\substack{S\in A(Np^4)^+\\p^2\mid2\beta}}\sum_{x\in\Z/p\Z}a(S[\begin{bmatrix}p^{-1}&-xp^{-2}\\&1\end{bmatrix})W(\mathbf{1},2\beta p^{-2}-x\alpha p^{-4})W(\chi,-\alpha p^{-4})e^{2\pi i\trace(SZ)}.
\end{multline*}

\noindent{\bf The case $l=6$:}
From \cite{JR3} Theorem 3.1, we have that
\begin{multline*}
(F|_k\mathcal{T}_\chi^6)(Z)=\\p^{2k-6}\sum_{S\in A(N)^+}a(S)
\big(\sum_{\substack{a,b\in(\Z/p^2\Z)^\times\\x\in(\Z/p\Z)^\times}} \chi(bx)e^{2\pi i \mathrm{tr}(S\begin{bmatrix}b(1+xp)p^{-2}&ap^{-2}\\
ap^{-2}&a^2b^{-1}p^{-2}\end{bmatrix})}\big)e^{2\pi i \mathrm{tr}(S[\begin{bmatrix} p^2&\\&1\end{bmatrix}]Z)}.
\end{multline*}
The inner sum is calculated as
\begin{align*}
&\sum_{\substack{a,b\in(\Z/p^2\Z)^\times\\ x\in(\Z/p\Z)^\times}} \chi(bx)
e^{2\pi i \big(a^2b^{-1}\gamma+2a\beta+\alpha b(px+1)\big)p^{-2}}\\
=&\sum_{\substack{a\in(\Z/p^2\Z)^\times\\ x,b_0\in(\Z/p\Z)^\times}} \sum_{b_1\in\Z/p\Z}\chi(b_0x)
e^{2\pi i \big(a^2(b_0^{-1}-b_1b_0^{-2}p)\gamma+2a\beta+\alpha (b_0+b_1p)(px+1)\big)p^{-2}}\\
=&p\sum_{\substack{a\in(\Z/p^2\Z)^\times\\ x,b_0\in(\Z/p\Z)^\times\\\alpha\equiv a^2b_0^{-2}\gamma  (p)}} \chi(b_0x)
e^{2\pi i \big(a^2b_0^{-1}\gamma+2a\beta+\alpha b_0(px+1)\big)p^{-2}}\\
=&p\sum_{\substack{ x,a_0,b_0\in(\Z/p\Z)^\times\\\alpha\equiv a_0^2b_0^{-2} \gamma (p)}}\sum_{a_1\in\Z/p\Z} \chi(b_0x)
e^{2\pi i \big((a_0+pa_1)^2b_0^{-1}\gamma+2(a_0+pa_1)\beta+\alpha b_0(px+1)\big)p^{-2}}\\
=&p^2\sum_{\substack{ x,a_0,b_0\in(\Z/p\Z)^\times\\\alpha\equiv a_0^2b_0^{-2}\gamma (p)\\2\beta\equiv-2a_0b_0^{-1}\gamma(p)}} \chi(b_0x)
e^{2\pi i \big(a_0^2b_0^{-1}\gamma+2a_0\beta+\alpha b_0(px+1)\big)p^{-2}}.\\
\intertext{If $p\mid\alpha$ then the sum on $x$ forces the expression to vanish.  So we assume that $p\nmid\alpha$. Now our sum is}
&p^{2}\sum\limits_{\substack{a_0,b_0\in(\Z/p\Z)^\times\\\alpha\equiv a_0^2b_0^{-2}\gamma (p)\\2\beta\equiv-2a_0b_0^{-1}\gamma(p)}} \sum\limits_{x\in(\Z/p\Z)^\times}\chi(b_0x)e^{2\pi i\alpha b_0 x p^{-1}}
e^{2\pi i \big(a_0^2b_0^{-1}\gamma+2a_0\beta+\alpha b_0\big)p^{-2}}\\
=&p^{2} W(\chi,\alpha)
\sum\limits_{\substack{a_0,b_0\in(\Z/p\Z)^\times\\\alpha\equiv a_0^2b_0^{-2}\gamma (p)\\2\beta\equiv-2a_0b_0^{-1}\gamma(p)}}e^{2\pi i \big(a_0^2b_0^{-1}\gamma+2\beta a_0+\alpha b_0\big)p^{-2}}\\
=&p^2W(\chi,\alpha)
\sum\limits_{\substack{a\in(\Z/p\Z)^\times\\\alpha\equiv a^2\gamma (p)\\2\beta\equiv-2a\gamma(p)}}W(\mathbf{1},(a^2\gamma+2\beta a+\alpha)p^{-1}).\\
\intertext{Indeed, we note that the expression $a^2\gamma+2\beta a+\alpha$ is well-defined modulo $p^2$.  Moreover, a calculation shows that if $a$ satisfies the congruence conditions above, then $4\gamma(a^2\gamma+2\beta a+\alpha)$ is congruent to $4\det(S)$ modulo $p^2$.  Hence our inner sum is}
&\begin{cases} p^{2} W(\chi,\alpha)W(\mathbf{1},4\det(S)p^{-1})&\text{if}\,p\nmid \alpha,\,\chi(\alpha)=\chi(\gamma),\, p\mid 4\det(S),\\
0&\text{otherwise.}
\end{cases}
\end{align*}
Continuing as in previous cases, we find that 
\begin{align*}
(F|_k\mathcal{T}_\chi^6)(Z)=p^{2k-4}\sum_{\substack{S\in A(Np^4)^+\\p^4\mid\mid\alpha,\, p^5\mid 4\det(S)\\\chi(\alpha p^{-4})=\chi(\gamma)}}a(S[\begin{bmatrix}p^{-2}&\\&1\end{bmatrix}])W(\chi,\alpha p^{-4})W(\mathbf{1},4\det(S)p^{-5})e^{2\pi i \trace(SZ)}.
\end{align*}
\noindent{\bf The case $l=7$:}
From \cite{JR3} Theorem 3.1, we have that
\begin{multline*}
(F|_k\mathcal{T}_\chi^7)(Z)=\\p^{k-7}\sum_{S\in A(N)^+}a(S)\sum_{a\in(\Z/p^3\Z)^\times}\chi(a)\big(\sum_{\substack{ b\in(\Z/p\Z)^\times\\z\in\Z/p^4\Z}}\chi(b)e^{2\pi i \trace(S\begin{bmatrix} zp^{-4} & b p^{-1}\\ bp^{-1} &\end{bmatrix})}\big)e^{2\pi i \trace(S[\begin{bmatrix} 1 & -ap^{-1}\\ &p\end{bmatrix}]Z)}.
\end{multline*}
The inner sum is calculated as 
\begin{align*}
\sum_{\substack{ b\in(\Z/p\Z)^\times\\z\in\Z/p^4\Z}}\chi(b)  e^{2\pi i \big(\alpha z p^{-4}+2\beta b p^{-1}\big)}
=\begin{cases}p^{4}W(\chi,2\beta) & p^4\mid\alpha,\\
0&\text{otherwise}.
\end{cases}
\end{align*}
Now, consider the map $S\mapsto {}^tPSP$, for $P=\left[\begin{smallmatrix}1 & -ap^{-1}\\ &p\end{smallmatrix}\right]$, first interchanging the order of summation.  Evidently, this map is a bijection between $A(Np^4)^+$ and $\{S\in A(Np^4)^+: p\mid\mid2\beta\,\text{and}\,\gamma\equiv-2\beta p^{-1}a(p^2)\}$.  Continuing the calculation, we have
 \begin{align*}
&=p^{k-3}\sum_{\substack{S\in A(Np^4)^+}}\sum_{a\in(\Z/p^3\Z)^\times}\chi(a)a(S)W(\chi,2\beta)e^{2\pi i \trace(S[\begin{bmatrix} 1 & -ap^{-1}\\ &p\end{bmatrix}]Z)}\\
&=p^{k-3}\sum_{a\in(\Z/p^3\Z)^\times}\sum_{\substack{S\in A(Np^4)^+\\p||2\beta,\,2\beta p^{-1}a\equiv -\gamma(p^2)}}\chi(a)a(S[\begin{bmatrix}1&ap^{-2}\\&p^{-1}\end{bmatrix}])W(\chi,2(\beta p^{-1}+a\alpha p^{-2}))e^{2\pi i \trace(SZ)}\\
&=p^{k-2}\sum_{\substack{S\in A(Np^4)^+\\p||2\beta}}\,\,\sum_{\substack{a\in(\Z/p^2\Z)^\times\\2\beta p^{-1}a\equiv -\gamma(p^2)}}a(S[\begin{bmatrix}1&ap^{-2}\\&p^{-1}\end{bmatrix}])W(\chi,-\gamma)e^{2\pi i \trace(SZ)}.\\
\end{align*}
Notice that there is exactly one $a$ which satisfies the congruence condition.\\

\noindent{\bf The case $l=8$:}
From \cite{JR3} we have that
\begin{multline*}
(F|_k\mathcal{T}_\chi^8)(Z)=p^{2k-9}\sum_{S\in A(N)^+}a(S)\sum_{b\in(\Z/p^3\Z)^\times}\chi(b)\\\big(\sum_{\substack{a,z\in(\Z/p^3\Z)^\times\\z\not\equiv1(p)}}\chi(az(1-z))e^{2\pi i \trace(S\begin{bmatrix}ab(1-z)p^{-3}  & ap^{-1}\\ap^{-1}&\end{bmatrix})}\big)e^{2\pi i \trace(S[\begin{bmatrix}
p&bp^{-1}\\
&p
\end{bmatrix}]Z)}.
\end{multline*}
We calculate the inner sum
\begin{align*}
&\sum_{\substack{a,z\in(\Z/p^3\Z)^\times\\z\not\equiv1(p)}}\chi(az(1-z))e^{2\pi i a(\alpha b (1-z)+2\beta p^2)p^{-3}}\\
=&\sum_{\substack{a_0\in (\Z/p\Z)^\times\\z\in(\Z/p^3\Z)^\times\\z\not\equiv1(p)}}\sum_{a_1\in\Z/p^2\Z}\chi(a_0z(1-z))e^{2\pi i (a_0+pa_1)(\alpha b (1-z)+2\beta p^2)p^{-3}}\\
=&\begin{cases}p^{2}\sum\limits_{\substack{z\in(\Z/p^3\Z)^\times\\z\not\equiv1(p)}}\chi(z(1-z))W(\chi, b (1-z)\alpha p^{-2}+2\beta)&\text{if}\, p^2\mid\alpha\\
0&\text{otherwise.}
\end{cases}
\end{align*}
Now, consider the map $S\mapsto {}^tPSP$, for $P=\left[\begin{smallmatrix}p&bp^{-1}\\&p\end{smallmatrix}\right]$, first interchanging the order of summation.  Evidently, this map is a bijection between $A(Np^2)^+$ and $\{S\in A(Np^4)^+ : p^2\mid2\beta,\, p^2\mid f_S(b)\}.$  Continuing the calculation, we have
\begin{align*}
=&p^{2k-7}\sum_{\substack{S\in A(Np^2)^+}}\sum\limits_{\substack{b,z\in(\Z/p^3\Z)^\times\\z\not\equiv1(p)}}\chi(bz(1-z))a(S)W(\chi, b (1-z)\alpha p^{-2}+2\beta)e^{2\pi i \trace(S[
\begin{bmatrix}p&bp^{-1}\\&p\end{bmatrix}]Z)}\\
=&p^{2k-7}\sum\limits_{\substack{b,z\in(\Z/p^3\Z)^\times\\z\not\equiv1(p)}}\sum_{\substack{S\in A(Np^4)^+\\p^2\mid 2\beta,\,p^2\mid f_S(b)}}\chi(bz(1-z))a(S[\begin{bmatrix}p^{-1}&-bp^{-3}\\&p^{-1}\end{bmatrix}])\\
&\quad W(\chi, -b (z+1)\alpha p^{-4}+2\beta p^{-2})e^{2\pi i \trace(SZ)}\\
=&p^{2k-4}\sum_{\substack{S\in A(Np^4)^+\\\, p^2\mid 2\beta}}\,\,\sum\limits_{\substack{b\in(\Z/p^2\Z)^\times\\z\in(\Z/p\Z)^\times\\z\not\equiv1(p)\\p^2\mid f_S(b)}}\chi(z(1-z))a(S[\begin{bmatrix}p^{-1}&-bp^{-3}\\&p^{-1}\end{bmatrix}])\\
&\quad W(\chi,  -z\alpha p^{-4}b^2+\gamma)e^{2\pi i \trace(SZ)}.
\end{align*}
\noindent{\bf The case $l=9$:}
From \cite{JR3} we have that 
\begin{multline*}
(F|_k \mathcal{T}^9_\chi)(Z)=\\p^{3k-6}\sum_{S\in A(N)^+}a(S)\sum_{a\in(\Z/p^2\Z)^\times}\big(\sum_{\substack{b,x\in(\Z/p\Z)^\times}} \chi(b)e^{2\pi i\trace(S \begin{bmatrix}bp^{-1}& \\&xp^{-1} \end{bmatrix})}\big)e^{2\pi i \trace(S[\begin{bmatrix}p^2&a\\&p \end{bmatrix}]Z)}.
\end{multline*}
The inner sum is calculated as 
\begin{align*}
&\sum_{\substack{b,x\in(\Z/p\Z)^\times}} \chi(b)
e^{2\pi i (\alpha b+\gamma x)p^{-1}}=W(\chi,\alpha)W(\mathbf{1},\gamma).
\end{align*}
Now, consider the map $S\mapsto {}^tPSP$, for $P=\left[\begin{smallmatrix}p^2&a\\&p\end{smallmatrix}\right]$, first interchanging the order of summation.  Evidently, this map is a bijection between $A(N)^+$ and $\{S\in A(Np^4)^+ : p^2\mid 2\beta,\, 2\alpha p^{-4}a\equiv2\beta p^{-2}(p),\,p^2\mid f_S(a)\}$.  Continuing the calculation, we have
\begin{align*}
=&p^{3k-6}\sum_{S\in A(N)^+}\sum_{a\in(\Z/p^2\Z)^\times}a(S)W(\chi,\alpha)W(\mathbf{1},\gamma)e^{2\pi i \trace(S[\begin{bmatrix}p^2&a\\&p\end{bmatrix}]Z)}\\
=&p^{3k-6}\sum_{a\in(\Z/p^2\Z)^\times}\sum_{\substack{S\in A(Np^4)^+\\p^2\mid2\beta,\,p^2\mid f_S(a)\\2\alpha p^{-4}a\equiv2\beta p^{-2}(p)}}a(S[\begin{bmatrix}p^{-2}&-ap^{-3}\\&p^{-1}\end{bmatrix}])W(\chi,\alpha p^{-4})W(\mathbf{1},f_S(a)p^{-2})e^{2\pi i \trace(SZ)}\\
=&p^{3k-5}\sum_{\substack{S\in A(Np^4)^+\\p^4\mid\mid\alpha,\,p^2\mid\mid 2\beta\\ p^6\mid4\det(S)}}\sum_{\substack{a\in(\Z/p\Z)^\times\\2\alpha p^{-4}a\equiv2\beta p^{-2}(p)}}a(S[\begin{bmatrix}p^{-2}&-ap^{-3}\\&p^{-1}\end{bmatrix}])W(\chi,\alpha p^{-4})\chi(-4\det(S)p^{-6})e^{2\pi i \trace(SZ)}\\
\end{align*}
Notice that at most one value of $a$ satisfies the condition in the final expression.\\

\noindent{\bf The case $l=10$:}
From \cite{JR3} we have that
\begin{multline*}
(F|_k \mathcal{T}^{10}_\chi)(Z)=\\p^{4k-6}\sum_{S\in A(N)^+}a(S)\sum_{a\in(\Z/p^2\Z)^\times}\big(\sum_{\substack{b\in(\Z/p\Z)^\times}} \chi(b)e^{2\pi i\trace(S \begin{bmatrix}bp^{-1}& \\& \end{bmatrix})}\big)e^{2\pi i \trace(S[\begin{bmatrix}p^2&a\\&p^2 \end{bmatrix}]Z)}.
\end{multline*}
The inner sum is calculated as
\begin{align*}
\sum_{b\in(\Z/p\Z)^\times} \chi(b)e^{2\pi i \alpha b p^{-1}}=W(\chi,\alpha).
\end{align*}
Now, consider the map $S\mapsto {}^tPSP$, for $P=\left[\begin{smallmatrix}p^2&a\\&p^2\end{smallmatrix}\right]$, first interchanging the order of summation.  Evidently, this map is a bijection between $A(N)^+$ and $\{S\in A(Np^4)^+:p^4\mid\mid\alpha,\,p^2\mid2\beta,\,2\alpha p^{-4}a\equiv2\beta p^{-2}(p^2),\,p^4\mid f_S(a)\}$.  Continuing the calculation, we have
\begin{align*}
&=p^{4k-6}\sum_{S\in A(N)^+}\sum_{a\in(\Z/p^2\Z)^\times}a(S)W(\chi,\alpha)e^{2\pi i \trace( S[\begin{bmatrix}
p^2&a\\
&p^2
\end{bmatrix}]Z)}\\
&=p^{4k-6}\sum_{a\in(\Z/p^2\Z)^\times}\sum_{\substack{S\in A(Np^4)^+\\p^4\mid\mid\alpha,\,p^2\mid2\beta,\,p^4\mid f_S(a)\\2\alpha p^{-4}a\equiv2\beta p^{-2}(p^2)}}a(S[\begin{bmatrix}p^{-2}&-ap^{-4}\\&p^{-2}\end{bmatrix}])W(\chi,\alpha p^{-4})e^{2\pi i \trace( SZ)}\\
&=p^{4k-6}\sum_{\substack{S\in A(Np^4)^+\\p^4\mid\mid\alpha,\,p^2\mid2\beta\\p^{8}\mid 4\det(S)}}\sum_{\substack{a\in(\Z/p^2\Z)^\times\\2\alpha p^{-4}a\equiv2\beta p^{-2}(p^2)}}a(S[\begin{bmatrix}p^{-2}&-ap^{-4}\\&p^{-2}\end{bmatrix}])W(\chi,\alpha p^{-4})e^{2\pi i \trace( SZ)}.
\end{align*}
There is at most one such value of $a$.\\

\noindent{\bf The case $l=11$:}
From \cite{JR3} we have 
\begin{multline*}
(F|_k \mathcal{T}^{11}_\chi)(Z)=p^{-k-10}\sum_{S\in A(N)^+}a(S)\sum_{b\in(\Z/p^4\Z)^\times}\chi(b)\\\big(\sum_{\substack{a\in(\Z/p^2\Z)^\times\\x\in\Z/p^3\Z\\z\in\Z/p^4\Z}} \chi(a)e^{2\pi i\trace(S \begin{bmatrix}zp^{-4} & (ap+xb)p^{-3} \\ (ap+xb)p^{-3}&xp^{-2} \end{bmatrix})}\big)e^{2\pi i \trace(S[\begin{bmatrix} 1 & bp^{-2}   \\ &p^{-1} \end{bmatrix}]Z)}
\end{multline*}
The inner sum is calculated as
\begin{align*}
\sum_{\substack{a\in(\Z/p^2\Z)^\times\\x\in\Z/p^3\Z\\z\in\Z/p^4\Z}}\chi(a)  e^{2\pi i (\alpha z+2\beta(ap+bx)p+\gamma x p^2)p^{-4}}=&\begin{cases}
p^8W(\chi,2\beta p^{-1})&\text{if}\, p^4\mid \alpha,\, p^3\mid(\gamma p+2\beta b),\\
0&\text{otherwise}.
\end{cases}\\
\end{align*}
Now, consider the map $S\mapsto {}^tPSP$, for $P=\left[\begin{smallmatrix} 1 & bp^{-2}   \\ &p^{-1}\end{smallmatrix}\right]$, first interchanging the order of summation.  Evidently, this map is a bijection between $\{S\in A(Np^4): p^2\mid\mid 2\beta,\,p^2\mid(\gamma p+2\beta b)\}$ and $\{S\in A(Np^4)^+:p\nmid 2\beta\}$.  Continuing the calculation, we have
\begin{align*}
=&p^{-k-2}\sum_{\substack{b\in(\Z/p^4\Z)^\times}}\sum_{\substack{S\in A(Np^4)^+\\p\mid\mid2\beta\\p^3\mid(\gamma p+2\beta b)}}\chi(b)a(S)W(\chi,2\beta p^{-1})e^{2\pi i \trace(S[
\begin{bmatrix} 1 & bp^{-2}   \\ &p^{-1} \end{bmatrix}]Z)}\\
=&p^{-k-2}\sum_{\substack{b\in(\Z/p^4\Z)^\times}}\sum_{\substack{S\in A(Np^4)^+\\p\nmid2\beta}}\chi(b)a(S[\begin{bmatrix}1&-bp^{-1}\\&p\end{bmatrix}])W(\chi,2\beta)e^{2\pi i \trace(SZ)}\\
=&p^{1-k}\sum_{\substack{S\in A(Np^4)^+\\p\nmid2\beta}}\sum_{\substack{b\in(\Z/p\Z)^\times}}\chi(b)a(S[\begin{bmatrix}1&-bp^{-1}\\&p\end{bmatrix}])W(\chi,2\beta)e^{2\pi i \trace(SZ)}.
\end{align*}
\noindent{\bf The case $l=12$:}
From \cite{JR3} we have
\begin{multline*}
(F|_k \mathcal{T}^{12}_\chi)(Z)=p^{-12}\sum_{S\in A(N)^+}a(S)\sum_{a\in(\Z/p^4\Z)^\times}\chi(a)\\\big(\sum_{\substack{y\in\Z/p^4\Z\\b,z\in(\Z/p^3\Z)^\times\\z\not\equiv1(p)}} \chi(bz(1-z))e^{2\pi i\trace(S \begin{bmatrix}a(y-b(1-z)p)p^{-4}&yp^{-3}\\yp^{-3} &a^{-1}(y+bp)p^{-2} \end{bmatrix})}\big)e^{2\pi i \trace(S[\begin{bmatrix} p&ap^{-2}\\&p^{-1} \end{bmatrix}]Z)}.
\end{multline*}
The inner sum is given by 
\begin{align*}
&\sum_{\substack{y\in\Z/p^4\Z\\b,z\in(\Z/p^3\Z)^\times\\z\not\equiv1(p)}} \chi(bz(1-z))e^{2\pi i \big(y(\alpha a+2\beta p +\gamma p^2 a^{-1})p^{-4}+b(\gamma a^{-1}p^2+\alpha(z-1)a)p^{-3}\big)}.\\
\intertext{The sum on the $y$ variable implies that this vanishes unless $p^4\mid \alpha a^2+2\beta ap +\gamma p^2$.  We assume this.  Continuing, the sum is now}
&p^4\sum\limits_{\substack{b_0\in(\Z/p\Z)^{\times}\\z\in(\Z/p^3\Z)^\times\\z\not\equiv1(p)}} \sum\limits_{b_1\in\Z/p^2\Z}\chi(b_0z(1-z))e^{2\pi i (b_0+pb_1)(\gamma a^{-1}p^2+\alpha(z-1)a)p^{-3}}.\\
\intertext{The sum on $b_1$ implies that the expression vanishes unless $p^2\mid\alpha$.  Hence the inner sum is given by}
&\begin{cases} p^{6}\sum\limits_{\substack{z\in(\Z/p^3\Z)^\times\\z\not\equiv1(p)}} \chi(z(1-z))W(\chi,\gamma a^{-1} +\alpha p^{-2}(z-1)a)& \text{if}\,p^4\mid (\alpha a^2+2\beta ap +\gamma p^2),\, p^2\mid\alpha,\\
0&\text{otherwise}.
\end{cases}\\
\end{align*}
Reasoning as in the previous cases we conclude that
\begin{multline*}
(F|_k\mathcal{T}^{12}_\chi)(Z)=\\p^{-1}\sum_{\substack{S\in A(Np^4)^+\\p\mid2\beta}}\sum_{\substack{a,z\in(\Z/p\Z)^\times\\z\not\equiv1(p)}} \chi(az(1-z))a(S[\begin{bmatrix}p^{-1}&-ap^{-2}\\&p\end{bmatrix}])W(\chi,az\alpha p^{-4}-2\beta p^{-1})e^{2\pi i \trace(SZ)}.\\
\end{multline*}
\noindent{\bf The case $l=13$:}
From \cite{JR3} we have
\begin{multline*}
(F_k\mathcal{T}^{13}_\chi)(Z)=\\p^{k-6}\sum_{S\in A(N)^+}a(S)\big(\sum_{\substack{a\in(\Z/p^2\Z)^\times\\b\in(\Z/p^3\Z)^\times\\x\in(\Z/p\Z)^\times}} \chi(bx)e^{2\pi i\trace(S \begin{bmatrix}b(1+x)p^{-1}&ap^{-2}\\
ap^{-2}&a^2b^{-1}p^{-3} \end{bmatrix})}\big)e^{2\pi i \trace(S[\begin{bmatrix} p^2&\\
&p^{-1}
\end{bmatrix}]Z)}.
\end{multline*}
The inner sum is given by
\begin{align*}
&\sum_{\substack{a\in(\Z/p^2\Z)^\times\\b\in(\Z/p^3\Z)^\times\\x\in(\Z/p\Z)^\times}} \chi(bx)e^{2\pi i \big(\gamma a^2 b^{-1}+\alpha b(x+1)p^2+2\beta a p\big)p^{-3}} \\
=&W(\chi,\alpha)\sum_{\substack{a\in(\Z/p^2\Z)^\times\\b\in(\Z/p^3\Z)^\times}} e^{2\pi i (\gamma a^2 p^{-2}b^{-1}+2\beta a p^{-1}+\alpha b)p^{-1}} \\
=&W(\chi,\alpha)\sum_{\substack{a\in(\Z/p^2\Z)^\times\\b_0\in(\Z/p\Z)^\times\\b_1\in\Z/p^2\Z}} e^{2\pi i (\gamma a^2 +2\beta a p+\alpha p^2)(b_0+pb_1)p^{-3}} \\
=&W(\chi,\alpha)\sum_{\substack{a\in(\Z/p^2\Z)^\times\\b_0\in(\Z/p\Z)^\times}} e^{2\pi i (\gamma a^2 +2\beta a p+\alpha p^2)b_0p^{-3}}\sum_{b_1\in\Z/p^2\Z} e^{2\pi i (\gamma a^2 +2\beta a p)b_1p^{-2}}\\
=&W(\chi,\alpha)\sum_{\substack{a\in(\Z/p^2\Z)^\times\\b_0\in(\Z/p\Z)^\times}} e^{2\pi i (\gamma a^2 +2\beta a p+\alpha p^2)b_0p^{-3}}\sum_{b_2,b_3\in\Z/p\Z} e^{2\pi i (\gamma a^2 +2\beta a p)(b_2+pb_3)p^{-2}}\\
=&W(\chi,\alpha)\sum_{\substack{a\in(\Z/p^2\Z)^\times\\b_0\in(\Z/p\Z)^\times}} e^{2\pi i (\gamma a^2 +2\beta a p+\alpha p^2)b_0p^{-3}}\sum_{b_2\in\Z/p\Z} e^{2\pi i (\gamma a^2 +2\beta a p)b_2p^{-2}}\sum_{b_3\in\Z/p\Z}e^{2\pi i (\gamma a^2)b_3p^{-1}}.\\
\intertext{This is zero unless $p\mid\gamma$.  Assume this.  Continuing the calculation we have}
=&pW(\chi,\alpha)\sum_{\substack{a\in(\Z/p^2\Z)^\times\\b_0\in(\Z/p\Z)^\times}} e^{2\pi i (\gamma a^2 +2\beta a p+\alpha p^2)b_0p^{-3}}\sum_{b_2\in\Z/p\Z} e^{2\pi i (\gamma p^{-1} a+2\beta  )b_2p^{-1}}\\
=&p^2W(\chi,\alpha)\sum_{\substack{a_0,b_0\in(\Z/p\Z)^\times\\a_1\in\Z/p\Z\\\gamma p^{-1} a_0\equiv-2\beta (p)}} e^{2\pi i (\gamma (a_0+pa_1)^2 +2\beta (a_0+pa_1)p+\alpha p^2)b_0p^{-3}}\\
=&p^2W(\chi,\alpha)\sum_{\substack{a_0,b_0\in(\Z/p\Z)^\times\\\gamma p^{-1} a_0\equiv-2\beta (p)}} e^{2\pi i (\gamma a_0^2 +2\beta a_0p+\alpha p^2)b_0p^{-3}}\sum_{a_1\in\Z/p\Z}e^{2\pi i (\gamma p^{-1}2a_0 +2\beta )a_1b_0p^{-1}}.\\
\intertext{Evidently, this is zero unless $p^2\mid\gamma$.  Assume this.  Then it follows that the sum is zero unless $p\mid 2\beta$.  We further assume this so that our sum is}
=&p^3W(\chi,\alpha)\sum_{\substack{a_0,b_0\in(\Z/p\Z)^\times}} e^{2\pi i (\gamma p^{-2}a_0^2 +2\beta p^{-1}a_0+\alpha)b_0p^{-1}}\\
=&p^3W(\chi,\alpha)\sum_{\substack{a_0\in(\Z/p\Z)^\times}} W(\mathbf{1},\gamma p^{-2}a_0^2 +2\beta p^{-1}a_0+\alpha).
\end{align*}
To calculate this, we notice first that the sum is zero if $p\mid\alpha$, so assume that $p\nmid \alpha$.  We assume first that $p^3\mid \gamma$.  In this case, 
\begin{align*}
\sum_{\substack{a_0\in(\Z/p\Z)^\times}} W(\mathbf{1},\gamma p^{-2}a_0^2 +2\beta p^{-1}a_0+\alpha)
=&-W(\mathbf{1},\alpha)+\sum_{\substack{a_0\in\Z/p\Z}} W(\mathbf{1},2\beta p^{-1}a_0+\alpha)\\
=&1+\sum_{\substack{a_0\in\Z/p\Z}} W(\mathbf{1},2\beta p^{-1}a_0+\alpha)\\
=&\begin{cases}1&\text{if}\quad p^2\nmid2\beta,\\
1-p&\text{if}\quad p^2\mid2\beta.\end{cases}
\end{align*}
Now, we assume that $p^2\mid\mid\gamma$.  In this case, the Gauss sum $W(\mathbf{1},\gamma p^{-2}a_0^2 +2\beta p^{-1}a_0+\alpha)$ is $p-1$ if $a_0$ is a root of the polynomial $x^2\gamma p^{-2} +2x\beta p^{-1}+\alpha $ modulo $p$ and -1 otherwise.  Thus, we need only count the solutions for a given $S$ modulo $p$.  The discriminant of the polynomial is $-4\det(S)p^{-2}$.  Hence, there are no solutions when $-4\det(S)p^{-2}$ is not a square mod $p$, one solution when $p\mid(-4\det(S)p^{-2})$, and two solutions when $-4\det(S)p^{-2}$ is a square mod $p$ and $p\nmid-4\det(S)p^{-2}$.  Thus, 
\begin{align*}
\sum_{\substack{a_0\in(\Z/p\Z)^\times}} W(\mathbf{1},\gamma p^{-2}a_0^2 +2\beta p^{-1}a_0+\alpha)
=&-W(\mathbf{1},\alpha)+\sum_{\substack{a_0\in\Z/p\Z}} W(\mathbf{1},\gamma p^{-2}a_0^2 +2\beta p^{-1}a_0+\alpha)\\
=&1+p\chi(-4\det(S)p^{-2}).\\
\end{align*}
To summarize,  we find the the inner sum is given by the formula
$$
p^3W(\chi,\alpha)(p\chi(-4\det(S)p^{-2})-W(\mathbf{1},\gamma p^{-2})).
$$

Now, consider the map $S\mapsto {}^tPSP$, for $P=\left[\begin{smallmatrix} p^2 &   \\ &p^{-1}\end{smallmatrix}\right]$, first interchanging the order of summation.  Evidently, this map is a bijection between $\{S\in A(N):p\mid2\beta,\,p^2\mid\gamma\}$ and $\{S\in A(Np^4):p^2\mid 2\beta\}$. Continuing the calculation, we have 
\begin{align*}
=&p^{k-3}\sum_{\substack{S\in A_(N)^+\\p^2\mid\gamma,\,p\mid 2\beta}}a(S)W(\chi,\alpha)(p\chi(-4\det(S)p^{-2})-W(\mathbf{1},\gamma p^{-2}))e^{2\pi i \trace(S[\begin{bmatrix}p^2&\\&p^{-1}\end{bmatrix}]Z)}\\
=&p^{k-3}\sum_{\substack{S\in A(Np^4)^+\\p^2\mid 2\beta}}a(S[\begin{bmatrix}p^{-2}&\\&p\end{bmatrix}])W(\chi,\alpha p^{-4})(p\chi(-4\det(S)p^{-4})-W(\mathbf{1},\gamma))e^{2\pi i \trace(SZ)}.
\end{align*}
\noindent{\bf The case $l=14$:}
From \cite{JR3} we have that
\begin{multline*}
(F|_k\mathcal{T}^{14}_\chi)(Z)=
p^{-6}\sum_{S\in A(N)^+}a(S)\big(\sum_{\substack{a\in(\Z/p^2\Z)^\times\\b\in(\Z/p\Z)^\times\\x\in\Z/p^4\Z}} \chi(b)e^{2\pi i \trace(S\begin{bmatrix}
bp^{-1}&ap^{-2}\\
ap^{-2}&x p^{-4}
\end{bmatrix})}e^{2\pi i \trace(S[\begin{bmatrix} p^2 & \\ & p^{-2} \end{bmatrix}]Z)}\big).
\end{multline*}
The inner sum is calculated as
\begin{align*}
\sum_{\substack{a\in(\Z/p^2\Z)^\times\\b\in(\Z/p\Z)^\times\\x\in\Z/p^4\Z}} \chi(b)
e^{2\pi i (x\gamma p^{-4}+\alpha b p^{-1}+2\beta a p^{-2} )}
=\begin{cases}p^{5}W(\chi,\alpha)W(\mathbf{1},2\beta p^{-1})&p^4\mid\gamma,\, p\mid2\beta,\\
0&\text{otherwise}.
\end{cases}\\
\end{align*}
Reasoning as in the previous cases we conclude that
\begin{align*}
(F|_k\mathcal{T}^{14}_\chi)(Z)
=&p^{-1}\sum_{\substack{S\in A(Np^4)^+\\p\mid2\beta}}a(S[\begin{bmatrix}p^{-2}&\\&p^2\end{bmatrix}])W(\chi,\alpha p^{-4})W(\mathbf{1},2\beta p^{-1})e^{2\pi i \trace(SZ)}.
\end{align*}

We now turn to the five assertions of the theorem. For each $S\in A(Np^4)^+$, we have that 
$W(\chi)a_\chi(S)=\sum_{l=1}^{14}a_l(S),
$
where $a_l(S)$ are given above.  \\

\noindent{\bf Proof of (i):} Assume that $S\in A(Np^4)^+$ is such that $p\nmid 2\beta$.  Then, $a_l(S)=0$ for all $l\neq 11$.
Hence, 
$$a_\chi(S)=p^{1-k}\chi(2\beta)\sum_{\substack{b\in(\Z/p\Z)^\times}}\chi(b)a(S[\begin{bmatrix}1&-bp^{-1}\\&p\end{bmatrix}]),$$
which proves statement (i).\\

\noindent{\bf Proof of (ii):} Assume that $S\in A(Np^4)^+$ is such that $p\mid\mid2\beta$ and $p^4\mid\mid\alpha$.  Then, $a_l(S)=0$ for $l\neq1,4,7,12,14$.  Hence,
\begin{align*}
a_\chi(S)=&p^{-1}\sum_{\substack{a,b\in(\Z/p\Z)^\times}}\chi\big(ab( -\gamma+2\beta p^{-1}a)\big)a(S[\begin{bmatrix}1&-(a+b)p^{-1}\\&1\end{bmatrix}])\\
 &+p^{k-2}\chi(-\alpha p^{-4})\sum_{\substack{x\in(\Z/p^2\Z)^\times\\x\alpha p^{-4}\equiv2\beta p^{-1}(p^2)}}a(S[\begin{bmatrix}p^{-1}&-xp^{-3}\\&1\end{bmatrix}]) \\
 &+p^{k-2}\chi(-\gamma)\sum_{\substack{a\in(\Z/p^2\Z)^\times\\2\beta p^{-1}a\equiv -\gamma(p^2)}}a(S[\begin{bmatrix}1&ap^{-2}\\&p^{-1}\end{bmatrix}])\\
 &+p^{-1}\sum_{\substack{a,z\in(\Z/p\Z)^\times\\z\not\equiv1(p)}} \chi\big(az(1-z)(az\alpha p^{-4}-2\beta p^{-1})\big)a(S[\begin{bmatrix}p^{-1}&-ap^{-2}\\&p\end{bmatrix}])\\
  &+p^{-1}\chi(\alpha p^{-4})W(\mathbf{1},2\beta p^{-1})a(S[\begin{bmatrix}p^{-2}&\\&p^2\end{bmatrix}]),
 \end{align*}
 which proves statement (ii).\\

 \noindent{\bf Proof of (iii):} Assume that $S\in A(Np^5)^+$ is such that $p\mid \mid 2\beta$.  Then, $a_l(S)=0$ for $l\neq1,7,12$.  Hence,
 \begin{align*}
 a_\chi(S)  =&p^{-1}\sum_{\substack{a,b\in(\Z/p\Z)^\times}}\chi(ab)\chi( -\gamma+2\beta p^{-1}a)a(S[\begin{bmatrix}1&-(a+b)p^{-1}\\&1\end{bmatrix}])\\
 &+p^{k-2}\chi(-\gamma)\sum_{\substack{a\in(\Z/p^2\Z)^\times\\2\beta p^{-1}a\equiv -\gamma(p^2)}}a(S[\begin{bmatrix}1&ap^{-2}\\&p^{-1}\end{bmatrix}])\\
 &+p^{-1}\sum_{\substack{a,z\in(\Z/p\Z)^\times\\z\not\equiv1(p)}} \chi(az(1-z))\chi(-2\beta p^{-1})a(S[\begin{bmatrix}p^{-1}&-ap^{-2}\\&p\end{bmatrix}])\\
   =&p^{-1}\sum_{\substack{a,b\in(\Z/p\Z)^\times}}\chi\big(ab( -\gamma+2\beta p^{-1}a)\big)a(S[\begin{bmatrix}1&-(a+b)p^{-1}\\&1\end{bmatrix}])\\
 &-p^{-1}\chi(2\beta p^{-1})\sum_{\substack{a\in(\Z/p\Z)^\times}}\chi(a)a(S[\begin{bmatrix}p^{-1}&-ap^{-2}\\&p\end{bmatrix}])\\
  &+p^{k-2}\chi(-\gamma)a(S[\begin{bmatrix}1&yp^{-2}\\&p^{-1}\end{bmatrix}]),
 \end{align*}
 where $y\in(\Z/p^2\Z)^\times$ is such that $2\beta p^{-1}y\equiv -\gamma(p^2)$, and we note that 
 $$
 \sum_{\substack{z\in(\Z/p\Z)^\times\\z\not\equiv1(p)}}\chi(z(1-z))=-\chi(-1),
 $$
 by Lemma \ref{jlemma}. This completes the proof of the third assertion.\\
 
 \noindent{\bf Proof of (iv):} Assume that $S\in A(Np^4)^+$ is such that $p^2\mid 2\beta$ and $p^4\mid\mid\alpha$.  Then $a_l(S)=0$ for $l=4,7,11$.  We define $W(\chi)b_\chi(S)=a_1(S)+a_2(S)+a_3(S)+a_5(S)+a_8(S)+a_{12}(S)+a_{13}(S)+a_{14}(S)$,  $W(\chi)c_\chi(S)=a_6(S)$, and $W(\chi)d_\chi(S)=a_9(S)+a_{10}(S)$.  Substituting the expressions for $a_l$ from above and applying the condition that $p^2\mid2\beta$, we have the desired expression for $b_\chi(S)$, $c_\chi(S)$, and $d_\chi(S)$.\\

 \noindent{\bf Proof of (v):} Assume that $S\in A(Np^5)^+$ is such that $p^2\mid 2\beta$.  Then $a_l(S)=0$ for all $l\neq1,2,3,8$.  Hence,
 \begin{align*}
 a_\chi(S)
 =&p^{-1}\chi(-\gamma)\sum_{\substack{a,b\in(\Z/p\Z)^\times}}\chi(ab)a(S[\begin{bmatrix}1&-(a+b)p^{-1}\\&1\end{bmatrix}])\\
 &-p^{k-3}\sum_{\substack{b,x\in(\Z/p\Z)^\times\\x\not\equiv1(p)}}\chi\big((1-x)(-\gamma x^{-1}+2\beta  bp^{-2})\big)a(S[\begin{bmatrix}p^{-1}&-bp^{-2}\\&1\end{bmatrix}])\\
&+p^{k-3}\chi(\gamma)(1-p+p\chi(-4\det(S)p^{-4}))a(S[\begin{bmatrix}p^{-1}&\\&1\end{bmatrix}])\\
&+p^{2k-4}\chi( \gamma)\sum\limits_{\substack{b\in(\Z/p^2\Z)^\times\\z\in(\Z/p\Z)^\times\\z\not\equiv1(p)\\p^2\mid(\alpha p^{-4}b^2-2\beta p^{-2}b+\gamma)}}\chi(z(1-z))a(S[\begin{bmatrix}p^{-1}&-bp^{-3}\\&p^{-1}\end{bmatrix}]).
 \end{align*}
 By Lemma \ref{jlemma}, we have the result.  This completes the proof of the theorem.
\end{proof}

\begin{proof}[Proof of Corollary \ref{corollary}]
Let $F$ be in $S_k(\Gamma^{\mathrm{para}}(1))=S_k(\SSp(4,\Z))$ and assume further that $F$ is in the Maass space.  Then by \cite{EZ} Theorem 2.2 and by formula (19) of \cite{Ma}, there exists a Jacobi form 
$$\phi(\tau,z)=\sum_{\substack{D,2\beta\in\Z, D\leq 0\\D\equiv (2\beta)^2(4)}}C_\phi(D)q^{((2\beta)^2-D)/4}\zeta^{2\beta}\,\in\, J_{k,1},$$
with $q=e^{2\pi i\tau}$ and $\zeta=e^{2\pi i z}$, such that the Fourier coefficients of $F$ are given by
\begin{equation}\label{maassfourier}
a(S)=\sum_{d\mid(\alpha,2\beta,\gamma)}d^{k-1}C_\phi(\frac{-4\det(S)}{d^2})=\sum_{d\mid(\alpha,2\beta,\gamma)}d^{k-1}C_\phi(\frac{D(S)}{d^2}),
\end{equation}
for $S\in A(1)^+$ and $D(S)=-4\det(S)$.  We show that $a_\chi(S)=0$ for every $S\in A(p^4)^+$ by considering cases (i)--(v) of Theorem \ref{maintheorem}.  Let $S\in A(p^4)^+$, as in \eqref{A(N)definition}.  To simplify the calculations, we note the following fact.  Let $q\neq p$ be a prime, $P$ be a matrix of the form $\left[\begin{smallmatrix}p^j&ap^k\\&p^l\end{smallmatrix}\right]$ with $a, j, k, l \in\Z$, and let $S[P]=\left[\begin{smallmatrix}\alpha'&\beta'\\\beta'&\gamma'\end{smallmatrix}\right]$; then $q^t\mid(\alpha,2\beta,\gamma)$ if and only if $q^t\mid(\alpha',2\beta',\gamma')$, for $t$ a non-negative integer. Hence if $f:\Z\rightarrow\C$ is a function then
$$
\sum_{d\mid(\alpha',2\beta',\gamma')}f(d)=\sum_{\substack{k\in\Z_{\geq0}\\p^k\mid(\alpha',2\beta',\gamma')}}\sum_{\substack{d\mid(\alpha,2\beta,\gamma)\\p\nmid d}}f(p^kd).
$$
\noindent{\bf Case (i)}:  We assume that $p\nmid 2\beta$.  Then by Theorem \ref{maintheorem}, we have
\begin{align*}
a_\chi(S)=&p^{1-k}\chi(2\beta)\sum_{\substack{b\in(\Z/p\Z)^\times}}\chi(b)a(S[\begin{bmatrix}1&-bp^{-1}\\&p\end{bmatrix}])\\
=&p^{1-k}\chi(2\beta)\sum_{\substack{b\in(\Z/p\Z)^\times}}\chi(b)a(\begin{bmatrix}\alpha&p\beta-b\alpha p^{-1}\\p\beta-b\alpha p^{-1}&p^2\gamma-2b\beta+b^2\alpha p^{-2}\end{bmatrix})\\
=&p^{1-k}\chi(2\beta)\sum_{\substack{b\in(\Z/p\Z)^\times}}\chi(b)\sum_{d\mid(\alpha,2\beta, \gamma)}d^{k-1}C_\phi(\frac{p^2D(S)}{d^2})\\
=&0.
\end{align*}
\noindent{\bf Case (ii)}: We assume that $p^4\mid\mid\alpha$ and $p\mid\mid 2\beta$.  Let $x\alpha p^{-4}\equiv2\beta p^{-1}(p^2)$ and $y2\beta p^{-1}\equiv -\gamma(p^2)$.  Then by Theorem \ref{maintheorem}, we have
\begin{align*}
a_\chi(S)&=p^{-1}\sum_{\substack{a,b\in(\Z/p\Z)^\times}}\chi\big(ab(2\beta p^{-1}a -\gamma)\big)a(S[\begin{bmatrix}1&-(a+b)p^{-1}\\&1\end{bmatrix}])\\\nonumber
 &+p^{-1}\sum_{\substack{a,z\in(\Z/p\Z)^\times\\z\not\equiv1(p)}} \chi\big(az(1-z)(az\alpha p^{-4}-2\beta p^{-1})\big)a(S[\begin{bmatrix}p^{-1}&-ap^{-2}\\&p\end{bmatrix}])\\\nonumber
&-p^{-1}\chi(\alpha p^{-4})a(S[\begin{bmatrix}p^{-2}&\\&p^2\end{bmatrix}])+p^{k-2}\chi(-\alpha p^{-4})a(S[\begin{bmatrix}p^{-1}&-xp^{-3}\\&1\end{bmatrix}]) \\
&+p^{k-2}\chi(-\gamma)a(S[\begin{bmatrix}1&yp^{-2}\\&p^{-1}\end{bmatrix}])\\
=&p^{-1}\sum_{\substack{a,b\in(\Z/p\Z)^\times}}\chi\big(ab(2\beta p^{-1}a -\gamma)\big)a(\begin{bmatrix}\alpha&\beta-\alpha p^{-1}(a+b)\\\beta-\alpha p^{-1}(a+b)&\gamma-2\beta p^{-1}(a+b)+\alpha p^{-2}(a+b)^2\end{bmatrix})\\
 &+p^{-1}\sum_{\substack{a,z\in(\Z/p\Z)^\times\\z\not\equiv1(p)}}\chi\big(az(1-z)(az\alpha p^{-4}-2\beta p^{-1})\big)a(\begin{bmatrix}\alpha p^{-2}&\beta-a\alpha p^{-3}\\\beta-a\alpha p^{-3}&p^2\gamma-2\beta p^{-1}a+\alpha p^{-4}a^2\end{bmatrix})\\
&-p^{-1}\chi(\alpha p^{-4})a(\begin{bmatrix}\alpha p^{-4}&\beta\\\beta&p^4\gamma\end{bmatrix})\\
&+p^{k-2}\chi(-\alpha p^{-4})a(\begin{bmatrix}\alpha p^{-2}&\beta p^{-1}-x\alpha p^{-4}\\\beta p^{-1}-x\alpha p^{-4}&\gamma-(2\beta p^{-1}-x\alpha p^{-4})xp^{-2}\end{bmatrix}) \\
&+p^{k-2}\chi(-\gamma)a(\begin{bmatrix}\alpha&\beta p^{-1}+y\alpha p^{-2}\\\beta p^{-1}+y\alpha p^{-2}&(\gamma+2\beta p^{-1}y)p^{-2}+y^2\alpha p^{-4}\end{bmatrix}).
\end{align*}
To further evaluate this Fourier coefficient, we will use formula \eqref{maassfourier} and the remarks from the beginning of the proof. We consider the first summand and notice that the $p\mid(\alpha',2\beta',\gamma')$ if and only if $\gamma\equiv2\beta p^{-1}(a+b)\bmod p$.  Therefore the first summand is
\begin{align*}
&p^{-1}\Big(\!\!\!\!\!\sum_{\substack{a,b\in(\Z/p\Z)^\times\\ \gamma\equiv2\beta p^{-1}(a+b)\bmod p}}\!\!\!\!\!\chi\big(ab(2\beta p^{-1}a -\gamma)\big)\Big)\cdot\sum_{\substack{ d\mid(\alpha,2\beta, \gamma)\\p\nmid d}}\big(d^{k-1}C_\phi(\frac{D(S)}{d^2})+(pd)^{k-1}C_\phi(\frac{D(S)}{(pd)^2})\big)\\
&+p^{-1}\Big(\sum_{\substack{a,b\in(\Z/p\Z)^\times\\ \gamma\not\equiv2\beta p^{-1}(a+b)\bmod p}}\chi\big(ab(2\beta p^{-1}a -\gamma)\big)\Big)\cdot\sum_{\substack{d\mid(\alpha,2\beta, \gamma)\\p\nmid d}}d^{k-1}C_\phi(\frac{D(S)}{d^2})\\
=&p^{-1}\Big(\!\!\!\!\!\sum_{\substack{a,b\in(\Z/p\Z)^\times\\ \gamma\equiv2\beta p^{-1}(a+b)\bmod p}}\!\!\!\!\!\chi\big(ab(2\beta p^{-1}a -\gamma)\big)\Big)\cdot\sum_{\substack{ d\mid(\alpha,2\beta, \gamma)\\p\nmid d}}(pd)^{k-1}C_\phi(\frac{D(S)}{(pd)^2})\\
=&p^{-1}\Big(\!\!\!\!\!\sum_{\substack{a\in(\Z/p\Z)^\times\\ \gamma\not\equiv2\beta p^{-1}a\bmod p}}\!\!\!\!\!\chi\big(-2\beta p^{-1}a)\big)\Big)\cdot\sum_{\substack{ d\mid(\alpha,2\beta, \gamma)\\p\nmid d}}(pd)^{k-1}C_\phi(\frac{D(S)}{(pd)^2})\\
=&-p^{-1}\chi(-\gamma)\sum_{\substack{ d\mid(\alpha,2\beta, \gamma)\\p\nmid d}}(pd)^{k-1}C_\phi(\frac{D(S)}{(pd)^2}).
\end{align*}
We consider the second summand and notice that $p\mid(\alpha',2\beta',\gamma')$ if and only if $2\beta p^{-1}\equiv \alpha p^{-4}a (p)$, and moreover $p^2\nmid(\alpha',2\beta',\gamma')$.  Therefore the second summand is 
\begin{align*}
&p^{-1}\Big(\!\!\!\!\!\sum_{\substack{a,z\in(\Z/p\Z)^\times\\z\not\equiv 1(p)\\2\beta p^{-1}\equiv \alpha p^{-4}a (p)}}\!\!\!\!\!\chi\big(az(1-z)(az\alpha p^{-4}-2\beta p^{-1})\big)\Big)\cdot\sum_{\substack{ d\mid(\alpha,2\beta, \gamma)\\p\nmid d}}\big(d^{k-1}C_\phi(\frac{D(S)}{d^2})+(pd)^{k-1}C_\phi(\frac{D(S)}{(pd)^2})\big)\\
&+p^{-1}\Big(\sum_{\substack{a,z\in(\Z/p\Z)^\times\\z\not\equiv 1(p)\\2\beta p^{-1}\not\equiv \alpha p^{-4}a (p)}} \chi\big(az(1-z)(az\alpha p^{-4}-2\beta p^{-1})\big)\Big)\cdot\sum_{\substack{d\mid(\alpha,2\beta, \gamma)\\p\nmid d}}d^{k-1}C_\phi(\frac{D(S)}{d^2})\\
=&p^{-1}\Big(\sum_{\substack{a,z\in(\Z/p\Z)^\times\\z\not\equiv 1(p)\\2\beta p^{-1}\equiv \alpha p^{-4}a (p)}}\chi\big(az(1-z)(az\alpha p^{-4}-2\beta p^{-1})\big)\Big)\cdot\sum_{\substack{ d\mid(\alpha,2\beta, \gamma)\\p\nmid d}}(pd)^{k-1}C_\phi(\frac{D(S)}{(pd)^2})\\
&+p^{-1}\Big(\sum_{\substack{a,z\in(\Z/p\Z)^\times\\z\not\equiv 1(p)}} \chi\big(az(1-z)(az\alpha p^{-4}-2\beta p^{-1})\big)\Big)\cdot\sum_{\substack{d\mid(\alpha,2\beta, \gamma)\\p\nmid d}}d^{k-1}C_\phi(\frac{D(S)}{d^2})\\
=&-p^{-1}\chi(-\alpha p^{-4})\sum_{\substack{ d\mid(\alpha,2\beta, \gamma)\\p\nmid d}}(pd)^{k-1}C_\phi(\frac{D(S)}{(pd)^2})\\
&+p^{-1}\chi(\alpha p^{-4})\sum_{\substack{d\mid(\alpha,2\beta, \gamma)\\p\nmid d}}d^{k-1}C_\phi(\frac{D(S)}{d^2}).
\end{align*} 
Here we have used Lemma \ref{jlemma} to calculate the character sum. The third summand is easily computed. Since $p\nmid \alpha p^{-4}$, it is given by
$$
-p^{-1}\chi(\alpha p^{-4})\sum_{\substack{d\mid(\alpha,2\beta, \gamma)\\p\nmid d}}d^{k-1}C_\phi(\frac{D(S)}{d^2}).
$$
For the fourth summand, we note that the conditions defining $x$ imply that $p\nmid 2(\beta p^{-1}-x\alpha p^{-4})$.  Hence, the fourth summand is
$$
p^{k-2}\chi(-\alpha p^{-4})\sum_{\substack{d\mid(\alpha,2\beta, \gamma)\\p\nmid d}}d^{k-1}C_\phi(\frac{p^{-2}D(S)}{d^2}).
$$
For the final summand, we also have $p\nmid 2(\beta p^{-1}+y\alpha p^{-2})$, so that the summand is
$$
p^{k-2}\chi(-\gamma)\sum_{\substack{d\mid(\alpha,2\beta, \gamma)\\p\nmid d}}d^{k-1}C_\phi(\frac{p^{-2}D(S)}{d^2}).
$$
Evidently, the calculated summands cancel so that $a_\chi(S)=0$.\\

\noindent{\bf Case (iii):} We assume that $p\mid\mid 2\beta$ and $p^5\mid\alpha$ and let $y\in(\Z/p^2\Z)^\times$ be such that $2\beta p^{-1}y\equiv -\gamma(p^2).$  Then by Theorem \ref{maintheorem}
\begin{align*}
a_\chi(S) =&p^{-1}\sum_{\substack{a,b\in(\Z/p\Z)^\times}}\chi\big(ab( 2\beta p^{-1}a-\gamma)\big)a(S[\begin{bmatrix}1&-(a+b)p^{-1}\\&1\end{bmatrix}])\\
 &+p^{k-2}\chi(-\gamma)a(S[\begin{bmatrix}1&yp^{-2}\\&p^{-1}\end{bmatrix}])-p^{-1}\chi(2\beta p^{-1})\sum_{\substack{a\in(\Z/p\Z)^\times}}\chi(a)a(S[\begin{bmatrix}p^{-1}&-ap^{-2}\\&p\end{bmatrix}])\\
  =&p^{-1}\sum_{\substack{a,b\in(\Z/p\Z)^\times}}\chi\big(ab( 2\beta p^{-1}a-\gamma)\big)a(\begin{bmatrix}\alpha&\beta-(a+b)\alpha p^{-1}\\\beta-(a+b)\alpha p^{-1}&\gamma-2\beta(a+b)p^{-1}+(a+b)^2\alpha p^{-2}\end{bmatrix})\\
  &+p^{k-2}\chi(-\gamma)a(\begin{bmatrix}\alpha&\beta p^{-1}+y\alpha p^{-2}\\\beta p^{-1}+y\alpha p^{-2}&(\gamma+y2\beta p^{-1})p^{-2}+y^2\alpha p^{-4}\end{bmatrix})\\
  &-p^{-1}\chi(2\beta p^{-1})\sum_{\substack{a\in(\Z/p\Z)^\times}}\chi(a)a(\begin{bmatrix}\alpha p^{-2}&\beta-a\alpha p^{-3}\\\beta-a\alpha p^{-3}&p^2\gamma-2a\beta p^{-1}+a^2\alpha p^{-4}\end{bmatrix}).
 \end{align*}
 The first summand is calculated as in the previous case to be
 $$
 -p^{-1}\chi(-\gamma)\sum_{\substack{ d\mid(\alpha,2\beta, \gamma)\\p\nmid d}}(pd)^{k-1}C_\phi(\frac{D(S)}{(pd)^2}).
 $$
   The second summand is also calculated as in the previous case to be 
 $$
 p^{k-2}\chi(-\gamma)\sum_{\substack{d\mid(\alpha,2\beta, \gamma)\\p\nmid d}}d^{k-1}C_\phi(\frac{p^{-2}D(S)}{d^2}).
 $$
 The third summand is zero since $p\nmid p^2\gamma-2a\beta p^{-1}+a^2\alpha p^{-4}$, and so the summation on $a$ causes the term to vanish.  Hence we see that $a_\chi(S)=0$.\\
 
 \noindent{\bf Case (iv):} In this case, we let $f_S(x)=\alpha p^{-4}x^2-2\beta p^{-2} x+\gamma$.
By assumption, $p^2\mid 2\beta$ and $p^4\mid\mid \alpha$.  As in the proof of case (iv) of Theorem \ref{maintheorem}, $a_\chi(S)=b_\chi(S)+c_\chi(S)+c_\chi(S)$.  Moreover, $W(\chi)b_\chi(S)=a_1(S)+a_2(S)+a_3(S)+a_5(S)+a_8(S)+a_{12}(S)+a_{13}(S)+a_{14}(S)$.  For convenience we will write $a'_l(S)$ for $W(\chi)^{-1}a_l(S)$.  We have
 \begin{align*}
a'_1(S) =&(1-p^{-1})\chi(\gamma)a(\begin{bmatrix}\alpha&\beta\\\beta&\gamma\end{bmatrix})-p^{-1}\chi( \gamma)\sum_{\substack{b\in(\Z/p\Z)^\times}}a(\begin{bmatrix}\alpha&\beta-b\alpha p^{-1}\\\beta-b\alpha p^{-1}&\gamma-2\beta bp^{-1}+b^2\alpha p^{-2}\end{bmatrix})\\\
a'_2(S)=&p^{k-3}\sum_{\substack{b,x,y\in(\Z/p\Z)^\times\\x,y\not\equiv1(p)}}\chi\big(y(1-x) (\alpha  p^{-4}(1-y)^{-1}(y-x)b^2 +2\beta  bp^{-2}-\gamma x^{-1})\big)\\
 &\qquad a(\begin{bmatrix}\alpha p^{-2}&\beta p^{-1}-b\alpha p^{-3}\\\beta p^{-1}-b\alpha p^{-3}&\gamma-2\beta p^{-2} b+\alpha p^{-4}b^2\end{bmatrix}),\\
a'_3(S)=&p^{k-3}\big(\chi(\gamma)+p\chi(-4\det(S)p^{-4}\gamma)-\chi( -\alpha p^{-4})W(\mathbf{1},2\beta p^{-2})\big)\\&\qquad a(\begin{bmatrix}\alpha p^{-2}&\beta p^{-1}\\\beta p^{-1}&\gamma\end{bmatrix}),\\
a'_5(S)=&p^{k-3}\chi(-\alpha p^{-4})\sum_{x\in\Z/p\Z}W(\mathbf{1},2\beta p^{-2}-x\alpha p^{-4})a(\begin{bmatrix}\alpha p^{-2}&\beta p^{-1}-x\alpha p^{-3}\\\beta p^{-1}-x\alpha p^{-3}&\gamma-2\beta p^{-2}x+\alpha p^{-4} x^2\end{bmatrix}),\\
a'_8(S)=&p^{2k-4}\!\!\!\sum\limits_{\substack{b\in(\Z/p^2\Z)^\times\\z\in(\Z/p\Z)^\times\\z\not\equiv1(p)\\f_S(b)\equiv0(p^2)}}\!\!\!\chi\big(z(1-z)( \gamma-z\alpha p^{-4}b^2)\big)a(\begin{bmatrix}\alpha p^{-2}&\beta p^{-2}-b\alpha p^{-4}\\\beta p^{-2}-b\alpha p^{-4}&p^{-2}(\gamma-2\beta p^{-2}b+\alpha p^{-4}b^2)\end{bmatrix}),\\
a'_{12}(S)=&-p^{-1}\chi(\alpha p^{-4})\sum_{\substack{a\in(\Z/p\Z)^\times}} a(\begin{bmatrix}\alpha p^{-2}&\beta-a\alpha p^{-3}\\\beta-a\alpha p^{-3}&p^2\gamma-2\beta p^{-1}a+\alpha p^{-4} a^2\end{bmatrix}),\\
a'_{13}(S)=&p^{k-3}\chi(\alpha p^{-4})(p\chi(-4\det(S)p^{-4})-W(\mathbf{1},\gamma))a(\begin{bmatrix}\alpha p^{-4}&\beta p^{-1}\\\beta p^{-1}&p^2\gamma\end{bmatrix}),\\
a'_{14}(S)=&(1-p^{-1})\chi(\alpha p^{-4})a(\begin{bmatrix}\alpha p^{-4}&\beta\\\beta&\gamma p^4\end{bmatrix}).
 \end{align*}
 We consider each summand of $b_\chi(S)$ separately.  First, $a'_1(S)$ vanishes.  This is clear if $p\mid \gamma$.    Assume that $p\nmid\gamma$.  Then $p\nmid \gamma-2\beta bp^{-1}+b^2\alpha p^{-2}$.  Hence, applying the formula \eqref{maassfourier}, we see that the two terms sum to zero.
 
 For $a'_2(S)$, we assume first that $p\mid \gamma$.  Then, $p\mid(\alpha',2\beta',\gamma')$ if and only if $2\beta p^{-2} \equiv \alpha p^{-4}b (p)$.  Notice that this also implies that in this case $p^2\mid\mid2\beta$.  Moreover, $p^2\nmid (\alpha',2\beta',\gamma')$.  Hence,
 \begin{align*}
a'_2(S)=&  p^{k-3}\Big(\sum_{\substack{b,x,y\in(\Z/p\Z)^\times\\x,y\not\equiv1(p)\\2\beta p^{-2}\not\equiv \alpha p^{-4}b(p)}}\chi\big((1-x)y( 2\beta  bp^{-2}+\alpha (1-y)^{-1}(y-x)b^2 p^{-4} )\big)\Big)\\
&\qquad\cdot\sum_{\substack{d\mid(\alpha, 2\beta,\gamma)\\p\nmid d}}d^{k-1}C_\phi(\frac{p^{-2}D(S)}{d^2})\\
  &+  p^{k-3}\Big(\sum_{\substack{b,x,y\in(\Z/p\Z)^\times\\x,y\not\equiv1(p)\\2\beta p^{-2}\equiv \alpha p^{-4}b(p)}}\chi\big((1-x)y(2\beta  bp^{-2}+\alpha (1-y)^{-1}(y-x)b^2 p^{-4} )\big)\Big)\\
  &\qquad\cdot\big(\sum_{\substack{d\mid(\alpha, 2\beta,\gamma)\\p\nmid d}}d^{k-1}C_\phi(\frac{p^{-2}D(S)}{d^2})+\sum_{\substack{d\mid(\alpha, 2\beta,\gamma)\\p\nmid d}}(pd)^{k-1}C_\phi(\frac{p^{-2}D(S)}{(pd)^2})\big)\\
  =&  p^{k-3}\Big(\sum_{\substack{b,x,y\in(\Z/p\Z)^\times\\x,y\not\equiv1(p)}}\chi\big((1-x)y(2\beta  bp^{-2}+\alpha (1-y)^{-1}(y-x)b^2 p^{-4} )\big)\Big)\\
  &\qquad\cdot\sum_{\substack{d\mid(\alpha, 2\beta,\gamma)\\p\nmid d}}d^{k-1}C_\phi(\frac{p^{-2}D(S)}{d^2})\\
  &+  p^{k-3}\Big(\sum_{\substack{b,x,y\in(\Z/p\Z)^\times\\x,y\not\equiv1(p)\\2\beta p^{-2}\equiv \alpha p^{-4}b(p)}}\chi\big((1-x)y( 2\beta  bp^{-2}+\alpha (1-y)^{-1}(y-x)b^2 p^{-4})\big)\Big)\\
  &\qquad\cdot\sum_{\substack{d\mid(\alpha, 2\beta,\gamma)\\p\nmid d}}(pd)^{k-1}C_\phi(\frac{p^{-2}D(S)}{(pd)^2})\\
  =&  p^{k-3}W(\mathbf{1},2\beta p^{-2})\Big(\sum_{\substack{x,y\in(\Z/p\Z)^\times\\x,y\not\equiv1(p)}}\chi\big((1-x)y(\alpha (1-y)^{-1}(y-x)p^{-4} )\big)\Big)\\
  &\qquad\cdot\sum_{\substack{d\mid(\alpha, 2\beta,\gamma)\\p\nmid d}}d^{k-1}C_\phi(\frac{p^{-2}D(S)}{d^2})\\
  &- p^{k-3}(p-2)\chi(-\alpha p^{-4})\sum_{\substack{b\in(\Z/p\Z)^\times\\2\beta p^{-2}\equiv \alpha p^{-4}b(p)}}\sum_{\substack{d\mid(\alpha, 2\beta,\gamma)\\p\nmid d}}(pd)^{k-1}C_\phi(\frac{p^{-2}D(S)}{(pd)^2})\\
   =&  p^{k-3} W(\mathbf{1},2\beta p^{-2})\chi(\alpha p^{-4} )\Big(\sum_{\substack{x,y\in(\Z/p\Z)^\times}}\chi\big((1-x)y(x+y)(1-(x+y))\big)\Big)\\&\qquad\cdot\sum_{\substack{d\mid(\alpha, 2\beta,\gamma)\\p\nmid d}}d^{k-1}C_\phi(\frac{p^{-2}D(S)}{d^2})\\
  &- p^{k-3}(p-2)\chi(-\alpha p^{-4})\Big(\sum_{\substack{b\in(\Z/p\Z)^\times\\2\beta p^{-2}\equiv\alpha p^{-4}b(p)}}1\Big)\sum_{\substack{d\mid(\alpha, 2\beta,\gamma)\\p\nmid d}}(pd)^{k-1}C_\phi(\frac{p^{-2}D(S)}{(pd)^2})\\
     =&  p^{-2} W(\mathbf{1},2\beta p^{-2})\Big(\chi(\alpha p^{-4} )+\chi(-\alpha p^{-4})\Big)\cdot\sum_{\substack{d\mid(\alpha, 2\beta,\gamma)\\p\nmid d}}(pd)^{k-1}C_\phi(\frac{D(S)}{(pd)^2})\\
  &- p^{-2}(p-2)\chi(-\alpha p^{-4})\Big(\sum_{\substack{b\in(\Z/p\Z)^\times\\2\beta p^{-2}\equiv\alpha p^{-4}b(p)}}1\Big)\sum_{\substack{d\mid(\alpha, 2\beta,\gamma)\\p\nmid d}}(p^2d)^{k-1}C_\phi(\frac{D(S)}{(p^2d)^2}).\\
 \end{align*}
Considering still $a'_2(S)$, we suppose that $p\nmid\gamma$.   Then, $p\mid (\alpha',2\beta',\gamma')$ if and only if $\gamma-2\beta p^{-2}b+\alpha p^{-4}b^2\equiv0(p)$.  Further, $p^2\mid(\alpha',2\beta',\gamma')$ if and only if $2\beta p^{-1}\equiv2b\alpha p^{-3}(p^2)$ and $\gamma-2\beta p^{-2}b+\alpha p^{-4}b^2\equiv0(p^2)$; together, these imply that $\gamma\equiv\alpha p^{-4}b^2(p^2)$ and that $p^6\mid\det(S)$.  We now consider cases depending upon the divisibility of $\det(S)$.  First assume that $p^6\nmid\det(S)$. Then we see that $p^2\nmid(\alpha',2\beta',\gamma')$ for all $b\in(\Z/p\Z)^\times$.  Hence, 
\begin{align*}
 a'_2(S)=&p^{k-3}\Big(\sum_{\substack{b,x,y\in(\Z/p\Z)^\times\\x,y\not\equiv1(p)\\f_S(b)\equiv0(p)}}\chi\big(y(1-x) (\alpha  p^{-4}(1-y)^{-1}(y-x)b^2 +2\beta  bp^{-2}-\gamma x^{-1})\big)\Big)\\
 &\cdot\big(\sum_{\substack{d\mid(\alpha, 2\beta,\gamma)\\p\nmid d}}d^{k-1}C_\phi(\frac{p^{-2}D(S)}{d^2})+\sum_{\substack{d\mid(\alpha, 2\beta,\gamma)\\p\nmid d}}(pd)^{k-1}C_\phi(\frac{p^{-2}D(S)}{(pd)^2})\big)\\
 &+ p^{k-3}\Big(\sum_{\substack{b,x,y\in(\Z/p\Z)^\times\\x,y\not\equiv1(p)\\f_S(b)\not\equiv0(p)}}\chi\big(y(1-x) (\alpha  p^{-4}(1-y)^{-1}(y-x)b^2 +2\beta  bp^{-2}-\gamma x^{-1})\big)\Big)\\
 &\qquad\cdot\sum_{\substack{d\mid(\alpha, 2\beta,\gamma)\\p\nmid d}}d^{k-1}C_\phi(\frac{p^{-2}D(S)}{d^2})\\
 =&p^{k-3}\Big(\sum_{\substack{b,x,y\in(\Z/p\Z)^\times\\x,y\not\equiv1(p)\\f_S(b)\equiv0(p)}}\chi\big(y(1-x) (\alpha  p^{-4}(1-y)^{-1}(y-x)b^2 +2\beta  bp^{-2}-\gamma x^{-1})\big)\Big)\\
 &\cdot\sum_{\substack{d\mid(\alpha, 2\beta,\gamma)\\p\nmid d}}(pd)^{k-1}C_\phi(\frac{p^{-2}D(S)}{(pd)^2})\\
 &+ p^{k-3}\Big(\sum_{\substack{b,x,y\in(\Z/p\Z)^\times\\x,y\not\equiv1(p)}}\chi\big(y(1-x) (\alpha  p^{-4}(1-y)^{-1}(y-x)b^2 +2\beta  bp^{-2}-\gamma x^{-1})\big)\Big)\\
 &\qquad\cdot\sum_{\substack{d\mid(\alpha, 2\beta,\gamma)\\p\nmid d}}d^{k-1}C_\phi(\frac{p^{-2}D(S)}{d^2})\\
  =&p^{k-3}\Big(\sum_{\substack{b,x,y\in(\Z/p\Z)^\times\\x,y\not\equiv1(p)\\f_S(b)\equiv0(p)}}\chi\big(y( \alpha p^{-4}(1-y)^{-1}b^2  -\gamma x^{-1})\big)\Big)\cdot\sum_{\substack{d\mid(\alpha, 2\beta,\gamma)\\p\nmid d}}(pd)^{k-1}C_\phi(\frac{p^{-2}D(S)}{(pd)^2})\\
 &+ p^{k-3}\Big(\sum_{\substack{b,x,y\in(\Z/p\Z)^\times\\x,y\not\equiv1(p)}}\chi\big(y(1-x) (\alpha  p^{-4}(1-y)^{-1}(y-x)b^2 +2\beta  bp^{-2}-\gamma x^{-1})\big)\Big)\\
 &\qquad\cdot\sum_{\substack{d\mid(\alpha, 2\beta,\gamma)\\p\nmid d}}d^{k-1}C_\phi(\frac{p^{-2}D(S)}{d^2}).
\end{align*}
 Still considering $a'_2(S)$, now with the assumptions that $p\nmid\gamma$ and $p^6\mid\det(S)$, we have
 \begin{align*}
 a'_2(S)=&p^{k-3}\Big(\sum_{\substack{b,x,y\in(\Z/p\Z)^\times\\x,y\not\equiv1(p)\\2\beta p^{-2}\equiv2\alpha p^{-4} b(p)}}\chi\big(y(1-x) (\alpha  p^{-4}(1-y)^{-1}(y-x)b^2 +2\beta  bp^{-2}-\gamma x^{-1})\big)\Big)\\
 &\cdot\big(\sum_{\substack{d\mid(\alpha, 2\beta,\gamma)\\p\nmid d}}d^{k-1}C_\phi(\frac{p^{-2}D(S)}{d^2})
 +\sum_{\substack{d\mid(\alpha, 2\beta,\gamma)\\p\nmid d}}(pd)^{k-1}C_\phi(\frac{p^{-2}D(S)}{(pd)^2})\\
 &\qquad+\sum_{\substack{d\mid(\alpha, 2\beta,\gamma)\\(p\nmid d}}(p^2d)^{k-1}C_\phi(\frac{p^{-2}D(S)}{(p^2d)^2})\big)\\
 &+p^{k-3}\Big(\sum_{\substack{b,x,y\in(\Z/p\Z)^\times\\x,y\not\equiv1(p)\\2\beta p^{-2}\not\equiv2\alpha p^{-4} b(p)\\f_S(b)\equiv0(p)}}\chi\big(y(1-x) (\alpha  p^{-4}(1-y)^{-1}(y-x)b^2 +2\beta  bp^{-2}-\gamma x^{-1})\big)\Big)\\
 &\qquad\cdot\big(\sum_{\substack{d\mid(\alpha, 2\beta,\gamma)\\p\nmid d}}d^{k-1}C_\phi(\frac{p^{-2}D(S)}{d^2})+\sum_{\substack{d\mid(\alpha, 2\beta,\gamma)\\p\nmid d}}(pd)^{k-1}C_\phi(\frac{p^{-2}D(S)}{(pd)^2})\big)\\
  &+p^{k-3}\Big(\sum_{\substack{b,x,y\in(\Z/p\Z)^\times\\x,y\not\equiv1(p)\\2\beta p^{-2}\not\equiv2\alpha p^{-4} b(p)\\f_S(b)\not\equiv0(p)}}\chi\big(y(1-x) (\alpha  p^{-4}(1-y)^{-1}(y-x)b^2 +2\beta  bp^{-2}-\gamma x^{-1})\big)\Big)\\
 &\qquad\cdot\sum_{\substack{d\mid(\alpha, 2\beta,\gamma)\\p\nmid d}}d^{k-1}C_\phi(\frac{p^{-2}D(S)}{d^2})\\
 = &p^{k-3}\Big(\sum_{\substack{b,x,y\in(\Z/p\Z)^\times\\x,y\not\equiv1(p)\\2\beta p^{-2}\equiv2\alpha p^{-4} b(p)}}\chi\big(y(1-x) (\alpha  p^{-4}(1-y)^{-1}(y-x)b^2 +2\beta  bp^{-2}-\gamma x^{-1})\big)\Big)\\
 &\qquad\cdot\sum_{\substack{d\mid(\alpha, 2\beta,\gamma)\\(p\nmid d}}(p^2d)^{k-1}C_\phi(\frac{p^{-2}D(S)}{(p^2d)^2})\\
 &+p^{k-3}\Big(\sum_{\substack{b,x,y\in(\Z/p\Z)^\times\\x,y\not\equiv1(p)\\f_S(b)\equiv0(p)}}\chi\big(y(1-x) (\alpha  p^{-4}(1-y)^{-1}(y-x)b^2 +2\beta  bp^{-2}-\gamma x^{-1})\big)\Big)\\
 &\qquad\cdot\sum_{\substack{d\mid(\alpha, 2\beta,\gamma)\\p\nmid d}}(pd)^{k-1}C_\phi(\frac{p^{-2}D(S)}{(pd)^2})\\
  &+p^{k-3}\Big(\sum_{\substack{b,x,y\in(\Z/p\Z)^\times\\x,y\not\equiv1(p)}}\chi\big(y(1-x) (\alpha  p^{-4}(1-y)^{-1}(y-x)b^2 +2\beta  bp^{-2}-\gamma x^{-1})\big)\Big)\\
 &\qquad\cdot\sum_{\substack{d\mid(\alpha, 2\beta,\gamma)\\p\nmid d}}d^{k-1}C_\phi(\frac{p^{-2}D(S)}{d^2}).
 \end{align*}
 We now consider $a'_3(S)$.  We see that $p\mid(\alpha',2\beta',\gamma')$ if and only if $p\mid\gamma$.  Further $p^2\mid(\alpha',2\beta',\gamma')$ if and only if $p^3\mid2\beta$ and $p^2\mid\gamma$; in this case, we also have that $p^6\mid4\det(S)$. Note that $p^3\nmid\alpha'$.  Hence, if $p\nmid\gamma$, 
  \begin{align*}
 a'_3(S)=&p^{k-3}\big(\chi(\gamma)+p\chi(D(S)p^{-4}\gamma)-\chi(-\alpha p^{-4})W(\mathbf{1},2\beta p^{-2})\big)\sum_{\substack{d\mid(\alpha, 2\beta,\gamma)\\p\nmid d}}d^{k-1}C_\phi(\frac{p^{-2}D(S)}{d^2}).
 \end{align*}
 If $p\mid\gamma$ but $p^2\nmid\gamma$ or $p^3\nmid2\beta$,
 \begin{align*}
a'_3(S)=&-p^{k-3}\chi( -\alpha p^{-4})W(\mathbf{1},2\beta p^{-2})\big( \!\!\sum_{\substack{d\mid(\alpha, 2\beta,\gamma)\\p\nmid d}}d^{k-1}C_\phi(\frac{p^{-2}D(S)}{d^2})+\!\!\!\sum_{\substack{d\mid(\alpha, 2\beta,\gamma)\\p\nmid d}}(pd)^{k-1}C_\phi(\frac{p^{-2}D(S)}{(pd)^2})\big).
 \end{align*}
  If $p^2\mid\gamma$ and $p^3\mid2\beta$, then 
   \begin{align*}
 a'_3(S)= &-p^{k-3}\chi( -\alpha p^{-4})(p-1)\big( \sum_{\substack{d\mid(\alpha, 2\beta,\gamma)\\p\nmid d}}d^{k-1}C_\phi(\frac{p^{-2}D(S)}{d^2})\\
 & \qquad+\sum_{\substack{d\mid(\alpha, 2\beta,\gamma)\\p\nmid d}}(pd)^{k-1}C_\phi(\frac{p^{-2}D(S)}{(pd)^2})+\sum_{\substack{d\mid(\alpha, 2\beta,\gamma)\\p\nmid d}}(p^2d)^{k-1}C_\phi(\frac{p^{-2}D(S)}{(p^2d)^2})\big).
 \end{align*}
 We consider $a'_5(S)$. We have that $p\mid(\alpha',2\beta',\gamma')$ if and only if $p\mid(\gamma-2\beta p^{-2}x+\alpha p^{-4}x^2)$.  Further, $p^2\mid(\alpha',2\beta',\gamma')$ if and only if $2\beta p^{-2}\equiv 2x\alpha p^{-4} (p)$ and $p^6\mid D(S)$.  Finally, $p^3\nmid \alpha'$.  Hence, in the case that $p^6\nmid D(S)$, we have 
 \begin{align*}
 a'_5(S)=&p^{k-3}\chi(-\alpha p^{-4})\Big(\sum_{\substack{x\in\Z/p\Z\\f_S(x)\equiv0(p)}}W(\mathbf{1},2\beta p^{-2}-x\alpha p^{-4})\Big)\sum_{\substack{d\mid(\alpha, 2\beta,\gamma)\\p\nmid d}}\big(d^{k-1}C_\phi(\frac{p^{-2}D(S)}{d^2})\\
 &\qquad+(pd)^{k-1}C_\phi(\frac{p^{-2}D(S)}{(pd)^2})\big)\\
 &+ p^{k-3}\chi(-\alpha p^{-4})\Big(\sum_{\substack{x\in\Z/p\Z\\f_S(x)\not\equiv0(p)}}W(\mathbf{1},2\beta p^{-2}-x\alpha p^{-4})\Big)\sum_{\substack{d\mid(\alpha, 2\beta,\gamma)\\p\nmid d}}d^{k-1}C_\phi(\frac{p^{-2}D(S)}{d^2})\\
 =&p^{k-3}\chi(-\alpha p^{-4})\Big(\sum_{\substack{x\in\Z/p\Z\\f_S(x)\equiv0(p)}}W(\mathbf{1},2\beta p^{-2}-x\alpha p^{-4})\Big)\sum_{\substack{d\mid(\alpha, 2\beta,\gamma)\\p\nmid d}}(pd)^{k-1}C_\phi(\frac{p^{-2}D(S)}{(pd)^2})\\
 &+ p^{k-3}\chi(-\alpha p^{-4})\Big(\sum_{\substack{x\in\Z/p\Z}}W(\mathbf{1},2\beta p^{-2}-x\alpha p^{-4})\Big)\sum_{\substack{d\mid(\alpha, 2\beta,\gamma)\\p\nmid d}}d^{k-1}C_\phi(\frac{p^{-2}D(S)}{d^2})\\
 =&p^{k-3}\chi(-\alpha p^{-4})\Big(\sum_{\substack{x\in\Z/p\Z\\f_S(x)\equiv0(p)}}W(\mathbf{1},2\beta p^{-2}-x\alpha p^{-4})\Big)\sum_{\substack{d\mid(\alpha, 2\beta,\gamma)\\p\nmid d}}(pd)^{k-1}C_\phi(\frac{p^{-2}D(S)}{(pd)^2}).
 \end{align*}
 In the case that $p^6\mid D(S)$, we have 
 \begin{align*}
 a'_5(S)=&p^{k-3}\chi(-\alpha p^{-4})\Big(\sum_{\substack{x\in\Z/p\Z\\2\beta p^{-2}\equiv2x\alpha p^{-4}}}W(\mathbf{1},2\beta p^{-2}-x\alpha p^{-4})\Big)\sum_{\substack{d\mid(\alpha, 2\beta,\gamma)\\p\nmid d}}\big(d^{k-1}C_\phi(\frac{p^{-2}D(S)}{d^2})\\
 &\qquad+(pd)^{k-1}C_\phi(\frac{p^{-2}D(S)}{(pd)^2})+(p^2d)^{k-1}C_\phi(\frac{p^{-2}D(S)}{(p^2d)^2})\big)\\
 &+p^{k-3}\chi(-\alpha p^{-4})\Big(\sum_{\substack{x\in\Z/p\Z\\2\beta p^{-2}\not\equiv2x\alpha p^{-4}}}W(\mathbf{1},2\beta p^{-2}-x\alpha p^{-4})\Big)\sum_{\substack{d\mid(\alpha, 2\beta,\gamma)\\p\nmid d}}d^{k-1}C_\phi(\frac{p^{-2}D(S)}{d^2})\\
 =&p^{k-3}\chi(-\alpha p^{-4})\Big(\sum_{\substack{x\in\Z/p\Z\\2\beta p^{-2}\equiv2x\alpha p^{-4}}}W(\mathbf{1},x\alpha p^{-4})\Big)\sum_{\substack{d\mid(\alpha, 2\beta,\gamma)\\p\nmid d}}\big((pd)^{k-1}C_\phi(\frac{p^{-2}D(S)}{(pd)^2})\\
& \qquad+(p^2d)^{k-1}C_\phi(\frac{p^{-2}D(S)}{(p^2d)^2})\big)\\
 &+p^{k-3}\chi(-\alpha p^{-4})\Big(\sum_{\substack{x\in\Z/p\Z}}W(\mathbf{1},2\beta p^{-2}-x\alpha p^{-4})\Big)\sum_{\substack{d\mid(\alpha, 2\beta,\gamma)\\p\nmid d}}d^{k-1}C_\phi(\frac{p^{-2}D(S)}{d^2})\\
  =&p^{k-3}\chi(-\alpha p^{-4})\Big(\sum_{\substack{x\in\Z/p\Z\\2\beta p^{-2}\equiv2x\alpha p^{-4}}}W(\mathbf{1},x\alpha p^{-4})\Big)\sum_{\substack{d\mid(\alpha, 2\beta,\gamma)\\p\nmid d}}\big((pd)^{k-1}C_\phi(\frac{p^{-2}D(S)}{(pd)^2})\\
& \qquad+(p^2d)^{k-1}C_\phi(\frac{p^{-2}D(S)}{(p^2d)^2})\big).
 \end{align*}
 \noindent
 We now consider $a'_8(S)$.  We first notice that the term vanishes unless there are $b\in (\Z/p^2\Z)^\times$, such that $f_S(b)$ vanishes modulo $p^2$.  By Lemma \ref{quadeqlemma}, we see that such $b$ exist if and only if  
 1) $p^4\mid\mid D(S)$ and $D(S)p^{-4}$ is a square modulo $p^2$ or 2)  $p^6\mid D(S)$.  Thus, if $p^5\mid\mid D(S)$, $a'_8(S)$ vanishes.  
 
 Assume 1) holds, and let $s\in(\Z/p^2\Z)^\times$ be a square root of $D(S)p^{-4}$.  We see by Lemma \ref{ccondlemma} that $b=b_{\pm}$ where $b_\pm:= (2\beta p^{-2}\pm s)(2\alpha p^{-4})^{-1}$  and that $p\nmid(\alpha',2\beta',\gamma')$. We have 
\begin{align*} 
 a'_8(S)=&p^{2k-4}\Big(\sum\limits_{\substack{b\in(\Z/p^2\Z)^\times\\z\in(\Z/p\Z)^\times\\b=b_\pm\\z\not\equiv 1(p)}}\!\!\!\chi\big(z(1-z)(\gamma  -z\alpha p^{-4}b^2\big)\Big)\sum_{\substack{d\mid(\alpha, 2\beta,\gamma)\\p\nmid d}}d^{k-1}C_\phi(\frac{p^{-4}D(S)}{d^2}).
\end{align*}
Now assume that $p^6\mid\mid D(S)$.  Then, $a_8(S)=0$ unless $p^2\mid\mid2\beta$.  Assume this.  Then by Lemmas \ref{quadeqlemma} and \ref{ccondlemma}, $\{b=2\beta p^{-2}(2\alpha p^{-4})^{-1}+py\,:\,y\in\Z/p\Z\}\subset(\Z/p^2\Z)^\times$ are solutions of the polynomial $f_S(x)$ modulo $p^2$.   Further, we have that $p\mid(\alpha',2\beta',\gamma')$ if and only if $D(S)p^{-6}$ is a square modulo $p$ and $(2\alpha p^{-4}y)^2\equiv D(S)p^{-6}(p)$. Moreover, $p^2\nmid(\alpha',2\beta',\gamma')$.  Hence,
\begin{align*} 
 a'_8(S)=&p^{2k-4}\Big(\sum\limits_{\substack{y\in\Z/p\Z\\z\in(\Z/p\Z)^\times\\z\not\equiv 1(p)}}\!\!\!\chi\big(z(1-z)(\gamma  -z\alpha p^{-4}(2\beta p^{-2}(2\alpha p^{-4})^{-1}+py)^2)\big)\Big)\\
 &\cdot\sum_{\substack{d\mid(\alpha, 2\beta,\gamma)\\p\nmid d}}d^{k-1}C_\phi(\frac{p^{-4}D(S)}{d^2})\\
 +&\frac{(1+\chi(D(S)p^{-6}))}{2} p^{2k-4}\Big(\!\!\!\!\!\sum\limits_{\substack{y\in\Z/p\Z\\z\in(\Z/p\Z)^\times\\(2\alpha p^{-4}y)^2\equiv D(S)p^{-6}(p)\\z\not\equiv1(p)}} \!\!\!\!\!\chi\big(z(1-z)(\gamma-z\alpha p^{-4}(2\beta p^{-2}(2\alpha p^{-4})^{-1}+py)^2)\big)\Big)\\
 &\cdot\sum_{\substack{d\mid(\alpha, 2\beta,\gamma)\\p\nmid d}}(pd)^{k-1}C_\phi(\frac{p^{-4}D(S)}{(pd)^2})\big)\\
  =&p^{2k-3}\Big(\!\!\!\sum\limits_{\substack{z\in(\Z/p\Z)^\times\\z\not\equiv 1(p)}}\!\!\!\chi\big(z(1-z)(\gamma -4^{-1}(\alpha p^{-4})^{-1}(2\beta p^{-2})^2z)\big)\Big)\sum_{\substack{d\mid(\alpha, 2\beta,\gamma)\\p\nmid d}}d^{k-1}C_\phi(\frac{p^{-4}D(S)}{d^2})\\
 &+(1+\chi(D(S)p^{-6})) p^{2k-4}\Big(\sum\limits_{\substack{z\in(\Z/p\Z)^\times\\z\not\equiv1(p)}}\!\!\!
 \chi\big(z(1-z)(\gamma-4^{-1}(\alpha p^{-4})^{-1}(2\beta p^{-2})^2z)\big)\Big)\\
 &\cdot\sum_{\substack{d\mid(\alpha, 2\beta,\gamma)\\p\nmid d}}(pd)^{k-1}C_\phi(\frac{p^{-4}D(S)}{(pd)^2})\\
  =&p^{2k-3}\Big(\!\!\!\sum\limits_{\substack{z\in(\Z/p\Z)^\times\\z\not\equiv 1(p)}}\!\!\!\chi(z \alpha p^{-4})\Big)\sum_{\substack{d\mid(\alpha, 2\beta,\gamma)\\p\nmid d}}d^{k-1}C_\phi(\frac{p^{-4}D(S)}{d^2})\\
 &+(1+\chi(D(S)p^{-6})) p^{2k-4}\Big(\sum\limits_{\substack{z\in(\Z/p\Z)^\times\\z\not\equiv1(p)}}\!\!\!
 \chi(z \alpha p^{-4})\Big)\sum_{\substack{d\mid(\alpha, 2\beta,\gamma)\\p\nmid d}}(pd)^{k-1}C_\phi(\frac{p^{-4}D(S)}{(pd)^2})\\
   =&-p^{2k-3}\chi(\alpha p^{-4})\sum_{\substack{d\mid(\alpha, 2\beta,\gamma)\\p\nmid d}}d^{k-1}C_\phi(\frac{p^{-4}D(S)}{d^2})\\
 &-p^{2k-4}(1+\chi(D(S)p^{-6})) \chi (\alpha p^{-4})\sum_{\substack{d\mid(\alpha, 2\beta,\gamma)\\p\nmid d}}(pd)^{k-1}C_\phi(\frac{p^{-4}D(S)}{(pd)^2}).
 \end{align*}
 We now assume that $p^7\mid\mid D(S)$.  Then, $a'_8(S)=0$ unless $p^2\mid\mid2\beta$.  Assume this. Then by Lemmas \ref{quadeqlemma} and \ref{ccondlemma} $\{b=2\beta p^{-2}(2\alpha p^{-4})^{-1}+py\,:\,y\in\Z/p\Z\}\subset(\Z/p^2\Z)^\times$ are solutions of $f_S(x)$ modulo $p^2$.  Further, we have that $p\mid(\alpha',2\beta',\gamma')$ if and only if $b\equiv 2\beta p^{-2}(2\alpha p^{-4})^{-1}(p^2)$.  Moreover, $p^2\nmid(\alpha',2\beta',\gamma')$.  Hence,
 \begin{align*}
 a'_8(S)=&p^{2k-4}\Big(\!\!\!\!\!\sum\limits_{\substack{y\in\Z/p\Z\\z\in(\Z/p\Z)^\times\\z\not\equiv 1(p)}}\!\!\!\!\!\chi\big(z(1-z)(\gamma -z\alpha p^{-4}(2\beta p^{-2}(2\alpha p^{-4})^{-1}+py)^2)\big)\Big)\sum_{\substack{d\mid(\alpha, 2\beta,\gamma)\\p\nmid d}}d^{k-1}C_\phi(\frac{p^{-4}D(S)}{d^2})\\
 &+ p^{2k-4}\Big(\!\!\!\!\!\sum\limits_{\substack{\\z\in(\Z/p\Z)^\times\\z\not\equiv1(p)}}\!\!\!\!\!
 \chi\big(z(1-z)(\gamma -z\alpha p^{-4}(2\beta p^{-2}(2\alpha p^{-4})^{-1})^2)\big)\Big)\sum_{\substack{d\mid(\alpha, 2\beta,\gamma)\\p\nmid d}}(pd)^{k-1}C_\phi(\frac{p^{-4}D(S)}{(pd)^2})\\
 =&-p^{2k-3}\chi( \alpha p^{-4})\sum_{\substack{d\mid(\alpha, 2\beta,\gamma)\\p\nmid d}}d^{k-1}C_\phi(\frac{p^{-4}D(S)}{d^2})- p^{2k-4}\chi(\alpha p^{-4})\sum_{\substack{d\mid(\alpha, 2\beta,\gamma)\\p\nmid d}}(pd)^{k-1}C_\phi(\frac{p^{-4}D(S)}{(pd)^2}).
\end{align*}
Finally, assume that $p^8\mid D(S)$.  Then, $a'_8(S)=0$ unless $p^2\mid\mid2\beta$.  Assume this. Then by Lemmas  \ref{quadeqlemma} and \ref{ccondlemma} $\{b=2\beta p^{-2}(2\alpha p^{-4})^{-1}+py\,:\,y\in\Z/p\Z\}\subset(\Z/p\Z)^\times$ are solutions of the polynomial $f_S(x)$ modulo $p^2$.  Further, we have that $p\mid(\alpha',2\beta',\gamma')$ if and only if $b\equiv 2\beta p^{-2}(2\alpha p^{-4})^{-1}(p^2)$ if and only if $p^2\mid(\alpha',2\beta',\gamma')$.   Moreover, $p^3\nmid(\alpha',2\beta',\gamma')$.  Hence,
\begin{align*}
a'_8(S)=&p^{2k-4}\Big(\!\!\!\sum\limits_{\substack{y\in\Z/p\Z\\z\in(\Z/p\Z)^\times\\z\not\equiv 1(p)}}\!\!\!\chi\big(z(1-z)(\gamma -z\alpha p^{-4}(2\beta p^{-2}(2\alpha p^{-4})^{-1}+py)^2)\big)\Big)\sum_{\substack{d\mid(\alpha, 2\beta,\gamma)\\p\nmid d}}d^{k-1}C_\phi(\frac{p^{-4}D(S)}{d^2})\\
 &+ p^{2k-4}\Big(\sum\limits_{\substack{\\z\in(\Z/p\Z)^\times\\z\not\equiv1(p)}}\!\!\!
 \chi\big(z(1-z) (\gamma-z\alpha p^{-4}(2\beta p^{-2}(2\alpha p^{-4})^{-1})^2)\big)\Big)\\
 &\cdot\sum_{\substack{d\mid(\alpha, 2\beta,\gamma)\\p\nmid d}}\big((pd)^{k-1}C_\phi(\frac{p^{-4}D(S)}{(pd)^2})+(p^2d)^{k-1}C_\phi(\frac{p^{-4}D(S)}{(p^2d)^2})\big)\\
 =&-p^{2k-3}\chi(\alpha p^{-4})\sum_{\substack{d\mid(\alpha, 2\beta,\gamma)\\p\nmid d}}d^{k-1}C_\phi(\frac{p^{-4}D(S)}{d^2})\\
 &-p^{2k-4}\chi(\alpha p^{-4})\sum_{\substack{d\mid(\alpha, 2\beta,\gamma)\\p\nmid d}}\big((pd)^{k-1}C_\phi(\frac{p^{-4}D(S)}{(pd)^2})+(p^2d)^{k-1}C_\phi(\frac{p^{-4}D(S)}{(p^2d)^2})\big).
 \end{align*}
  For $a'_{12}(S)$, we see that $p\nmid \gamma'$ and hence
  \begin{align*}
  a'_{12}(S)=&-(1-p^{-1})\chi(\alpha p^{-4})\sum_{\substack{d\mid(\alpha, 2\beta,\gamma)\\p\nmid d}}d^{k-1}C_\phi(\frac{D(S)}{d^2}).
  \end{align*}
  For $a'_{13}(S)$, we see that $p\nmid\alpha'$ and hence
  \begin{align*}
 a'_{13}(S)= &p^{k-3}\chi(\alpha p^{-4})\big(p\chi(D(S)p^{-4})-W(\mathbf{1},\gamma)\big)\sum_{\substack{d\mid(\alpha, 2\beta,\gamma)\\p\nmid d}} d^{k-1}C_\phi(\frac{p^{-2}D(S)}{d^2}).
  \end{align*}
Finally, for $a'_{14}(S)$, we see that $p\nmid\alpha'$ 
\begin{align*}
a'_{14}(S)=&(1-p^{-1})\chi(\alpha p^{-4})\sum_{\substack{d\mid(\alpha, 2\beta,\gamma)\\p\nmid d}}d^{k-1}C_\phi(\frac{D(S)}{d^2}).
\end{align*}
Next, $c_\chi(S)$ vanishes unless $p^5\mid D(S)$ and $\chi(\alpha p^{-4})=\chi(\gamma)$.  Assume these conditions.  Then, 
\begin{align*}
c_\chi(S)=&a'_6(S)=p^{2k-4}\chi(\alpha p^{-4})W(\mathbf{1},D(S)p^{-5})a([\begin{bmatrix}\alpha p^{-4}&2\beta p^{-2}\\2\beta p^{-2}&\gamma\end{bmatrix}])\\
=&p^{2k-4}\chi(\alpha p^{-4})W(\mathbf{1},D(S)p^{-5})\sum_{\substack{d\mid(\alpha,2\beta,\gamma)\\p\nmid d}}d^{k-1}C_\phi(\frac{p^{-4}D(S)}{d^2}).
\end{align*}
Now, $d_\chi(S)=a'_9(S)+a'_{10}(S)$.  Consider first $a'_9(S)$, which vanishes unless $p^6\mid D(S)$ and there exists an $a\in (\Z/p\Z)^\times$ satisfying $2\alpha p^{-4}a\equiv 2\beta p^{-2}$.  Assume this.  Then, 
\begin{align*}
a'_9(S)=&p^{3k-5}\chi(\alpha p^{-4})\chi(D(S)p^{-6})a(S[\begin{bmatrix}p^{-2}&-ap^{-3}\\&p^{-1}\end{bmatrix}])\\
=&p^{3k-5}\chi(\alpha p^{-4})\chi(D(S)p^{-6})a(S[\begin{bmatrix}\alpha p^{-4}&(\beta p^{-2}-a\alpha p^{-4})p^{-1}\\(\beta p^{-2}-a\alpha p^{-4})p^{-1}&p^{-2}(\gamma-2a\beta p^{-2}+a^2\alpha p^{-4})\end{bmatrix}])\\
=&p^{3k-5}\chi(\alpha p^{-4})\chi(D(S)p^{-6})\sum_{\substack{d\mid(\alpha,2\beta,\gamma)\\p\nmid d}}d^{k-1}C_\phi(\frac{p^{-6}D(S)}{d^2}).
\end{align*}
Consider $a'_{10}(S)$, which vanishes unless $p^8\mid D(S)$ and $p^2\mid\mid 2\beta$.  Assume this.  Then
\begin{align*}
a'_{10}(S)=&p^{4k-6}\chi(\alpha p^{-4})\sum_{\substack{d\mid(\alpha,2\beta,\gamma)\\p\nmid d}}d^{k-1}C_\phi(\frac{p^{-8}D(S)}{d^2}).
\end{align*}

We now show that $a_\chi(S)$ vanishes.  First suppose that $p^4\mid\mid D(S)$.  Then $c_\chi(S)=d_\chi(S)=0$.  Notice that the terms $a'_{14}(S)$ and $a'_{12}(S)$ of $b_\chi(S)$ cancel one another, so we consider the five remaining terms.  

Assume first that $p\mid\gamma$, and hence that $p^2\mid\mid 2\beta$.  Then, 
\begin{align*}
a_\chi(S)&=b_\chi(S)=a'_2(S)+a'_3(S)+a'_5(S)+a'_8(S)+a'_{13}(S)\\
&=   -p^{-2} \Big(\chi(\alpha p^{-4} )+\chi(-\alpha p^{-4})\Big)\cdot\sum_{\substack{d\mid(\alpha, 2\beta,\gamma)\\p\nmid d}}(pd)^{k-1}C_\phi(\frac{D(S)}{(pd)^2})\\
  &- p^{-2}(p-2)\chi(-\alpha p^{-4})\sum_{\substack{d\mid(\alpha, 2\beta,\gamma)\\p\nmid d}}(p^2d)^{k-1}C_\phi(\frac{D(S)}{(p^2d)^2})\\
   &+p^{-2}\chi( -\alpha p^{-4})\big( \!\!\sum_{\substack{d\mid(\alpha, 2\beta,\gamma)\\p\nmid d}}(pd)^{k-1}C_\phi(\frac{D(S)}{(pd)^2})+\!\!\!\sum_{\substack{d\mid(\alpha, 2\beta,\gamma)\\p\nmid d}}(p^2d)^{k-1}C_\phi(\frac{D(S)}{(p^2d)^2})\big)\\
      &+p^{-2}\chi(-\alpha p^{-4})\Big(\sum_{\substack{x\in\Z/p\Z\\f_S(x)\equiv0(p)}}W(\mathbf{1},2\beta p^{-2}-x\alpha p^{-4})\Big)\sum_{\substack{d\mid(\alpha, 2\beta,\gamma)\\p\nmid d}}(p^2d)^{k-1}C_\phi(\frac{D(S)}{(p^2d)^2})\\
&+p^{-2}\Big(\sum\limits_{\substack{b\in(\Z/p^2\Z)^\times\\z\in(\Z/p\Z)^\times\\b=b_\pm\\z\not\equiv 1(p)}}\!\!\!\chi\big(z(1-z)(\gamma  -z\alpha p^{-4}b^2\big)\Big)\sum_{\substack{d\mid(\alpha, 2\beta,\gamma)\\p\nmid d}}(p^2d)^{k-1}C_\phi(\frac{D(S)}{(p^2d)^2})\\
 &+p^{-2}\chi(\alpha p^{-4})\sum_{\substack{d\mid(\alpha, 2\beta,\gamma)\\p\nmid d}} (pd)^{k-1}C_\phi(\frac{D(S)}{(pd)^2})\\
  &=p^{-2}\Big(-(p-3)\chi(-\alpha p^{-4})+\chi(-\alpha p^{-4})\big(\sum_{\substack{x\in\Z/p\Z\\f_S(x)\equiv0(p)}}W(\mathbf{1},2\beta p^{-2}-x\alpha p^{-4})\big)\\
  &+\sum\limits_{\substack{b\in(\Z/p^2\Z)^\times\\z\in(\Z/p\Z)^\times\\b=b_\pm\\z\not\equiv 1(p)}}\!\!\!\chi\big(z(1-z)(\gamma  -z\alpha p^{-4}b^2\big)\Big)\cdot\sum_{\substack{d\mid(\alpha, 2\beta,\gamma)\\p\nmid d}}(p^2d)^{k-1}C_\phi(\frac{D(S)}{(p^2d)^2})\\
  &=p^{-2}\chi(-\alpha p^{-4})\Big(-(p-3)+(p-2)+\sum\limits_{\substack{z\in(\Z/p\Z)^\times\\z\not\equiv 1(p)}}\!\!\!\chi(1-z)\Big)\cdot\sum_{\substack{d\mid(\alpha, 2\beta,\gamma)\\p\nmid d}}(p^2d)^{k-1}C_\phi(\frac{D(S)}{(p^2d)^2})\\
  &=0.
\end{align*}
We now consider the case that $p\nmid\gamma$.  We will further assume that $p^{-4}D(S)$ is a square modulo $p$, and hence by Hensel's lemma is a square in $\Z_p$.  Let $s\in \Z_p$ be a square root of $p^{-4}D(S)$. We have
\begin{align*}
a_\chi(S)=&b_\chi(S)=a'_2(S)+a'_3(S)+a'_5(S)+a'_8(S)+a'_{13}(S)\\
=&p^{k-3}\Big(\sum_{\substack{b,x,y\in(\Z/p\Z)^\times\\x,y\not\equiv1(p)\\f_S(b)\equiv0(p)}}\chi\big(y( \alpha p^{-4}(1-y)^{-1}b^2  -\gamma x^{-1})\big)\Big)\cdot\sum_{\substack{d\mid(\alpha, 2\beta,\gamma)\\p\nmid d}}(pd)^{k-1}C_\phi(\frac{p^{-2}D(S)}{(pd)^2})\\
 &+ p^{k-3}\Big(\sum_{\substack{b,x,y\in(\Z/p\Z)^\times\\x,y\not\equiv1(p)}}\chi\big(y(1-x) (\alpha  p^{-4}(1-y)^{-1}(y-x)b^2 +2\beta  bp^{-2}-\gamma x^{-1})\big)\Big)\\
 &\qquad\cdot\sum_{\substack{d\mid(\alpha, 2\beta,\gamma)\\p\nmid d}}d^{k-1}C_\phi(\frac{p^{-2}D(S)}{d^2})\\
 &+p^{k-3}\big(\chi(\gamma)(p+1)-\chi(-\alpha p^{-4})W(\mathbf{1},2\beta p^{-2})\big)\sum_{\substack{d\mid(\alpha, 2\beta,\gamma)\\p\nmid d}}d^{k-1}C_\phi(\frac{p^{-2}D(S)}{d^2})\\
 &+p^{k-3}\chi(-\alpha p^{-4})\Big(\sum_{\substack{x\in\Z/p\Z\\f_S(x)\equiv0(p)}}W(\mathbf{1},2\beta p^{-2}-x\alpha p^{-4})\Big)\sum_{\substack{d\mid(\alpha, 2\beta,\gamma)\\p\nmid d}}(pd)^{k-1}C_\phi(\frac{p^{-2}D(S)}{(pd)^2})\\
 &+p^{2k-4}\Big(\sum\limits_{\substack{b\in(\Z/p^2\Z)^\times\\z\in(\Z/p\Z)^\times\\b=b_\pm\\z\not\equiv 1(p)}}\!\!\!\chi\big(z(1-z)(\gamma  -z\alpha p^{-4}b^2\big)\Big)\sum_{\substack{d\mid(\alpha, 2\beta,\gamma)\\p\nmid d}}d^{k-1}C_\phi(\frac{p^{-4}D(S)}{d^2})\\
 &+p^{k-3}\chi(\alpha p^{-4})(p+1)\sum_{\substack{d\mid(\alpha, 2\beta,\gamma)\\p\nmid d}} d^{k-1}C_\phi(\frac{p^{-2}D(S)}{d^2})\\
 =&p^{-2}\Big(\sum_{\substack{b,x,y\in(\Z/p\Z)^\times\\x,y\not\equiv1(p)\\f_S(b)\equiv0(p)}}\chi\big(y( \alpha p^{-4}(1-y)^{-1}b^2  -\gamma x^{-1})\big)+\sum\limits_{\substack{b\in(\Z/p^2\Z)^\times\\z\in(\Z/p\Z)^\times\\b=b_\pm\\z\not\equiv 1(p)}}\!\!\!\chi\big(z(1-z)(\gamma  -z\alpha p^{-4}b^2)\big)\\
 &\qquad-2\chi(-\alpha p^{-4})\Big)\cdot\sum_{\substack{d\mid(\alpha, 2\beta,\gamma)\\p\nmid d}}(p^2d)^{k-1}C_\phi(\frac{D(S)}{(p^2d)^2})\\
 &+p^{-2}\Big(\sum_{\substack{b,x,y\in(\Z/p\Z)^\times\\x,y\not\equiv1(p)}}\chi\big(y(1-x) (\alpha  p^{-4}(1-y)^{-1}(y-x)b^2 +2\beta  bp^{-2}-\gamma x^{-1})\big)\\
 &+\chi(\gamma)(p+1)-\chi(-\alpha p^{-4})W(\mathbf{1},2\beta p^{-2})+\chi(\alpha p^{-4})(p+1)\Big)\cdot\sum_{\substack{d\mid(\alpha, 2\beta,\gamma)\\p\nmid d}}(pd)^{k-1}C_\phi(\frac{D(S)}{(pd)^2}).
\end{align*}
Let
\begin{align*}
A&=\sum_{\substack{b,x,y\in(\Z/p\Z)^\times\\x,y\not\equiv1(p)\\f_S(b)\equiv0(p)}}\chi\big(y( \alpha p^{-4}(1-y)^{-1}b^2  -\gamma x^{-1})\big)+\sum\limits_{\substack{b\in(\Z/p^2\Z)^\times\\z\in(\Z/p\Z)^\times\\b=b_\pm\\z\not\equiv 1(p)}}\!\!\!\chi\big(z(1-z)(\gamma  -z\alpha p^{-4}b^2)\big)\\
 &\qquad-2\chi(-\alpha p^{-4}).
\end{align*}
We will show that $A$ vanishes.  Indeed, 
\begin{align*}
A=&-2\chi(-\alpha p^{-4})+\sum_{\substack{x\in\Z/p\Z\\b,y\in(\Z/p\Z)^\times\\y\not\equiv1(p)\\f_S(b)\equiv0(p)}}\chi\big(y( \alpha p^{-4}(1-y)^{-1}b^2  -\gamma x)\big)-\sum_{\substack{b,y\in(\Z/p\Z)^\times\\y\not\equiv1(p)\\f_S(b)\equiv0(p)}}\chi\big(y( \alpha p^{-4}(1-y)^{-1}b^2)\big)\\
&-\sum_{\substack{b,y\in(\Z/p\Z)^\times\\y\not\equiv1(p)\\f_S(b)\equiv0(p)}}\chi\big(y( \alpha p^{-4}(1-y)^{-1}b^2  -\gamma )\big)+\sum\limits_{\substack{b\in(\Z/p^2\Z)^\times\\z\in(\Z/p\Z)^\times\\b=b_\pm\\z\not\equiv 1(p)}}\!\!\!\chi\big(z(1-z)(\gamma  -z\alpha p^{-4}b^2)\big)\\
=&-2\chi(-\alpha p^{-4})-2\chi(\alpha p^{-4})\sum_{\substack{y\in(\Z/p\Z)^\times\\y\not\equiv1(p)}}\chi(y(1-y)) \\
&-\sum_{\substack{b,y\in(\Z/p\Z)^\times\\y\not\equiv1(p)\\f_S(b)\equiv0(p)}}\chi\big(y(1-y)(\gamma-y \alpha p^{-4}b^2  )\big)+\sum\limits_{\substack{b\in(\Z/p^2\Z)^\times\\z\in(\Z/p\Z)^\times\\b=b_\pm\\z\not\equiv 1(p)}}\!\!\!\chi\big(z(1-z)(\gamma  -z\alpha p^{-4}b^2)\big)\\
=&\,0.
\end{align*}
Here we have used Lemma \ref{jlemma} for the last equality and two sequential changes of variable on $y$ for the penultimate equality.  Now, applying Lemma \ref{mmlemma} we have that $a_\chi(S)=0$.  To complete the calculation for $p^4\mid\mid D(S)$ we consider the case that  $p\nmid\gamma$ and $D(S)$ is not a square modulo $p$. Since the polynomial $f_S(x)$ has no roots modulo $p$, several of the terms in $a_\chi(S)$ vanish.  Therefore, again applying Lemma \ref{mmlemma}, we have
\begin{align*}
a_\chi(S)=&b_\chi(S)=a'_2(S)+a'_3(S)+a'_5(S)+a'_8(S)+a'_{13}(S)\\
=& p^{k-3}\Big(\sum_{\substack{b,x,y\in(\Z/p\Z)^\times\\x,y\not\equiv1(p)}}\chi\big(y(1-x) (\alpha  p^{-4}(1-y)^{-1}(y-x)b^2 +2\beta  bp^{-2}-\gamma x^{-1})\big)\Big)\cdot\\
 &\qquad \sum_{\substack{d\mid(\alpha, 2\beta,\gamma)\\p\nmid d}}d^{k-1}C_\phi(\frac{p^{-2}D(S)}{d^2})\\
 &+p^{k-3}\big(\chi(\gamma)-p\chi(\gamma)-\chi(-\alpha p^{-4})W(\mathbf{1},2\beta p^{-2})\big)\sum_{\substack{d\mid(\alpha, 2\beta,\gamma)\\p\nmid d}}d^{k-1}C_\phi(\frac{p^{-2}D(S)}{d^2})\\
 &+p^{k-3}\chi(\alpha p^{-4})(1-p)\sum_{\substack{d\mid(\alpha, 2\beta,\gamma)\\p\nmid d}} d^{k-1}C_\phi(\frac{p^{-2}D(S)}{d^2})\\
=&0.
\end{align*}

Next, we consider the case when $p^5\mid\mid D(S)$.  There are two possibilities,  $p\mid\gamma$ and $p^3\mid 2\beta$ or $p\nmid\gamma$ and $p^2\mid\mid2\beta$.  We consider first the case when $p\mid\gamma$ so that $p^3\mid 2\beta$. Notice that in this case, $p^2\nmid\gamma$.
We have $a_\chi(S)=b_\chi(S)$ so that 
\begin{align*} 
a_\chi(S)=&  p^{-2} (p-1)\Big(\chi(\alpha p^{-4} )+\chi(-\alpha p^{-4})\Big)\cdot\sum_{\substack{d\mid(\alpha, 2\beta,\gamma)\\p\nmid d}}(pd)^{k-1}C_\phi(\frac{D(S)}{(pd)^2})\\
&-p^{-2}(p-1)\chi( -\alpha p^{-4})\big( \!\!\sum_{\substack{d\mid(\alpha, 2\beta,\gamma)\\p\nmid d}}(pd)^{k-1}C_\phi(\frac{D(S)}{(pd)^2})+\!\!\!\sum_{\substack{d\mid(\alpha, 2\beta,\gamma)\\p\nmid d}}(p^2d)^{k-1}C_\phi(\frac{D(S)}{(p^2d)^2})\big)\\
&+ p^{-2}\chi(-\alpha p^{-4})(p-1)\sum_{\substack{d\mid(\alpha, 2\beta,\gamma)\\p\nmid d}}(p^2d)^{k-1}C_\phi(\frac{D(S)}{(p^2d)^2})\\
&-p^{-2}(p-1)\chi(\alpha p^{-4})\sum_{\substack{d\mid(\alpha, 2\beta,\gamma)\\p\nmid d}} (pd)^{k-1}C_\phi(\frac{D(S)}{(pd)^2})\\
=&0.
\end{align*}
Now assume that $p\nmid\gamma$ so that $p^2\mid\mid2\beta$.  Then we must have that  $\chi(\alpha p^{-4})=\chi(\gamma)$.  Hence, $a_\chi(S)=b_\chi(S)+c_\chi(S)$ so that
\begin{align*}
a_\chi(S)=&p^{-2}\Big(\sum_{\substack{b,x,y\in(\Z/p\Z)^\times\\x,y\not\equiv1(p)\\f_S(b)\equiv0(p)}}\chi\big(y( \alpha p^{-4}(1-y)^{-1}b^2  -\gamma x^{-1})\big)\Big)\cdot\sum_{\substack{d\mid(\alpha, 2\beta,\gamma)\\p\nmid d}}(p^2d)^{k-1}C_\phi(\frac{D(S)}{(p^2d)^2})\\
 &+ p^{-2}\Big(\sum_{\substack{b,x,y\in(\Z/p\Z)^\times\\x,y\not\equiv1(p)}}\chi\big(y(1-x) (\alpha  p^{-4}(1-y)^{-1}(y-x)b^2 +2\beta  bp^{-2}-\gamma x^{-1})\big)\Big)\\
 &\qquad\cdot\sum_{\substack{d\mid(\alpha, 2\beta,\gamma)\\p\nmid d}}(pd)^{k-1}C_\phi(\frac{p^{-2}D(S)}{(pd)^2})\\
 &+p^{-2}\big(\chi(\gamma)+\chi(-\alpha p^{-4})\big)\sum_{\substack{d\mid(\alpha, 2\beta,\gamma)\\p\nmid d}}(pd)^{k-1}C_\phi(\frac{D(S)}{(pd)^2})\\
 &+p^{-2}\chi(-\alpha p^{-4})\Big(\sum_{\substack{x\in\Z/p\Z\\f_S(x)\equiv0(p)}}W(\mathbf{1},2\beta p^{-2}-x\alpha p^{-4})\Big)\sum_{\substack{d\mid(\alpha, 2\beta,\gamma)\\p\nmid d}}(p^2d)^{k-1}C_\phi(\frac{D(S)}{(p^2d)^2})\\
 &+p^{-2}\chi(\alpha p^{-4})\sum_{\substack{d\mid(\alpha, 2\beta,\gamma)\\p\nmid d}} (pd)^{k-1}C_\phi(\frac{D(S)}{(pd)^2})\\
 &-p^{-2}\chi(\alpha p^{-4})\sum_{\substack{d\mid(\alpha,2\beta,\gamma)\\p\nmid d}}(p^2)d^{k-1}C_\phi(\frac{D(S)}{(p^2d)^2})\\
 =&p^{-2}\Big(\chi(\gamma)+\chi(-\gamma)\Big)\cdot\sum_{\substack{d\mid(\alpha, 2\beta,\gamma)\\p\nmid d}}(p^2d)^{k-1}C_\phi(\frac{D(S)}{(p^2d)^2})\\
 &+ p^{-2}\Big(-2\chi(\gamma)-\chi(-\alpha p^{-4})\Big)\cdot\sum_{\substack{d\mid(\alpha, 2\beta,\gamma)\\p\nmid d}}(pd)^{k-1}C_\phi(\frac{p^{-2}D(S)}{(pd)^2})\\
 &+p^{-2}\big(\chi(\gamma)+\chi(-\alpha p^{-4})\big)\sum_{\substack{d\mid(\alpha, 2\beta,\gamma)\\p\nmid d}}(pd)^{k-1}C_\phi(\frac{D(S)}{(pd)^2})\\
 &-p^{-2}\chi(-\alpha p^{-4})\sum_{\substack{d\mid(\alpha, 2\beta,\gamma)\\p\nmid d}}(p^2d)^{k-1}C_\phi(\frac{D(S)}{(p^2d)^2})\\
 &+p^{-2}\chi(\alpha p^{-4})\sum_{\substack{d\mid(\alpha, 2\beta,\gamma)\\p\nmid d}} (pd)^{k-1}C_\phi(\frac{D(S)}{(pd)^2})\\
 &-p^{-2}\chi(\alpha p^{-4})\sum_{\substack{d\mid(\alpha,2\beta,\gamma)\\p\nmid d}}(p^2)d^{k-1}C_\phi(\frac{D(S)}{(p^2d)^2})\\
 =&0.
 \end{align*}
 Here we have used Lemmas \ref{mslemma} and \ref{mmlemma} to evaluate the character sums.
 
We now assume that $p^6\mid D(S)$ and $p^8\nmid D(S)$.  Assume first that $p\mid \gamma$.  In this case, we see that in fact, $p^2\mid \gamma$ and $p^3\mid 2\beta$.  Hence,
\begin{align*}
a_\chi(S)=&b_\chi(S)+c_\chi(S)+d_\chi(S)\\
=&  p^{-2} (p-1)\Big(\chi(\alpha p^{-4} )+\chi(-\alpha p^{-4})\Big)\cdot\sum_{\substack{d\mid(\alpha, 2\beta,\gamma)\\p\nmid d}}(pd)^{k-1}C_\phi(\frac{D(S)}{(pd)^2})\\
   &-p^{-2}\chi( -\alpha p^{-4})(p-1)\big( \sum_{\substack{d\mid(\alpha, 2\beta,\gamma)\\p\nmid d}}(pd)^{k-1}C_\phi(\frac{D(S)}{(pd)^2})\\
 & \qquad+\sum_{\substack{d\mid(\alpha, 2\beta,\gamma)\\p\nmid d}}(p^2d)^{k-1}C_\phi(\frac{D(S)}{(p^2d)^2})+\sum_{\substack{d\mid(\alpha, 2\beta,\gamma)\\p\nmid d}}(p^3d)^{k-1}C_\phi(\frac{D(S)}{(p^3d)^2})\big)\\
&+p^{-2}(p-1)\chi(-\alpha p^{-4})\sum_{\substack{d\mid(\alpha, 2\beta,\gamma)\\p\nmid d}}\big((p^2d)^{k-1}C_\phi(\frac{D(S)}{(p^2d)^2})+(p^3d)^{k-1}C_\phi(\frac{D(S)}{(p^3d)^2})\big)\\
&-p^{-2}(p-1)\chi(\alpha p^{-4})\sum_{\substack{d\mid(\alpha, 2\beta,\gamma)\\p\nmid d}} (pd)^{k-1}C_\phi(\frac{D(S)}{(pd)^2})\\
=&0.
\end{align*}
Now suppose that $p\nmid \gamma$ so that $p^2\mid\mid 2\beta$ and $\chi(\alpha p^{-4})=\chi(\gamma)$.  Further we assume that $p^6\mid\mid D(S)$.  Then,
\begin{align*}
a_\chi(S)=&b_\chi(S)+c_\chi(S)+d_\chi(S)\\
= &p^{-2}\Big(\sum_{\substack{b,x,y\in(\Z/p\Z)^\times\\x,y\not\equiv1(p)\\2\beta p^{-2}\equiv2\alpha p^{-4} b(p)}}\chi\big(y(1-x) (\alpha  p^{-4}(1-y)^{-1}(y-x)b^2 +2\beta  bp^{-2}-\gamma x^{-1})\big)\Big)\\
 &\qquad\cdot\Big(\sum_{\substack{d\mid(\alpha, 2\beta,\gamma)\\(p\nmid d}}(p^3d)^{k-1}C_\phi(\frac{D(S)}{(p^3d)^2})+\sum_{\substack{d\mid(\alpha, 2\beta,\gamma)\\p\nmid d}}(p^2d)^{k-1}C_\phi(\frac{D(S)}{(p^2d)^2})\Big)\\
  &+p^{-2}\Big(\sum_{\substack{b,x,y\in(\Z/p\Z)^\times\\x,y\not\equiv1(p)}}\chi\big(y(1-x) (\alpha  p^{-4}(1-y)^{-1}(y-x)b^2 +2\beta  bp^{-2}-\gamma x^{-1})\big)\Big)\\
 &\qquad\cdot\sum_{\substack{d\mid(\alpha, 2\beta,\gamma)\\p\nmid d}}(pd)^{k-1}C_\phi(\frac{D(S)}{(pd)^2})\\
&+p^{-2}\big(\chi(\gamma)+\chi(-\alpha p^{-4})\big)\sum_{\substack{d\mid(\alpha, 2\beta,\gamma)\\p\nmid d}}(pd)^{k-1}C_\phi(\frac{D(S)}{(pd)^2})\\
 &-p^{-2}\chi(-\alpha p^{-4})\sum_{\substack{d\mid(\alpha, 2\beta,\gamma)\\p\nmid d}}\big((p^2d)^{k-1}C_\phi(\frac{D(S)}{(p^2d)^2})+(p^3d)^{k-1}C_\phi(\frac{D(S)}{(p^3d)^2})\big)\\
&-p^{-1}\chi(\alpha p^{-4})\sum_{\substack{d\mid(\alpha, 2\beta,\gamma)\\p\nmid d}}(p^2d)^{k-1}C_\phi(\frac{D(S)}{(p^2d)^2})\\
 &-p^{-2}(1+\chi(D(S)p^{-6})) \chi (\alpha p^{-4})\sum_{\substack{d\mid(\alpha, 2\beta,\gamma)\\p\nmid d}}(p^3d)^{k-1}C_\phi(\frac{D(S)}{(p^3d)^2})\\
 &+p^{-2}\chi(\alpha p^{-4})\sum_{\substack{d\mid(\alpha, 2\beta,\gamma)\\p\nmid d}} (pd)^{k-1}C_\phi(\frac{D(S)}{(pd)^2})\\
 &+p^{-2}(p-1)\chi(\alpha p^{-4})\sum_{\substack{d\mid(\alpha,2\beta,\gamma)\\p\nmid d}}(p^2d)^{k-1}C_\phi(\frac{D(S)}{(p^2d)^2})\\
&+p^{-2}\chi(\alpha p^{-4})\chi(D(S)p^{-6})\sum_{\substack{d\mid(\alpha,2\beta,\gamma)\\p\nmid d}}(p^3d)^{k-1}C_\phi(\frac{D(S)}{(p^3d)^2})\\
= &p^{-2}\Big(\chi(\gamma)+\chi(-\gamma))\cdot\Big(\sum_{\substack{d\mid(\alpha, 2\beta,\gamma)\\(p\nmid d}}(p^3d)^{k-1}C_\phi(\frac{D(S)}{(p^3d)^2})+\sum_{\substack{d\mid(\alpha, 2\beta,\gamma)\\p\nmid d}}(p^2d)^{k-1}C_\phi(\frac{D(S)}{(p^2d)^2})\Big)\\
  &-p^{-2}\Big(2\chi(\gamma)+\chi(-\alpha p^{-4})\Big)\cdot\sum_{\substack{d\mid(\alpha, 2\beta,\gamma)\\p\nmid d}}(pd)^{k-1}C_\phi(\frac{D(S)}{(pd)^2})\\
&+p^{-2}\big(\chi(\gamma)+\chi(-\alpha p^{-4})\big)\sum_{\substack{d\mid(\alpha, 2\beta,\gamma)\\p\nmid d}}(pd)^{k-1}C_\phi(\frac{D(S)}{(pd)^2})\\
 &-p^{-2}\chi(-\alpha p^{-4})\sum_{\substack{d\mid(\alpha, 2\beta,\gamma)\\p\nmid d}}\big((p^2d)^{k-1}C_\phi(\frac{D(S)}{(p^2d)^2})+(p^3d)^{k-1}C_\phi(\frac{D(S)}{(p^3d)^2})\big)\\
&-p^{-1}\chi(\alpha p^{-4})\sum_{\substack{d\mid(\alpha, 2\beta,\gamma)\\p\nmid d}}(p^2d)^{k-1}C_\phi(\frac{D(S)}{(p^2d)^2})\\
 &-p^{-2}(1+\chi(D(S)p^{-6})) \chi (\alpha p^{-4})\sum_{\substack{d\mid(\alpha, 2\beta,\gamma)\\p\nmid d}}(p^3d)^{k-1}C_\phi(\frac{D(S)}{(p^3d)^2})\\
 &+p^{-2}\chi(\alpha p^{-4})\sum_{\substack{d\mid(\alpha, 2\beta,\gamma)\\p\nmid d}} (pd)^{k-1}C_\phi(\frac{D(S)}{(pd)^2})\\
 &+p^{-2}(p-1)\chi(\alpha p^{-4})\sum_{\substack{d\mid(\alpha,2\beta,\gamma)\\p\nmid d}}(p^2d)^{k-1}C_\phi(\frac{D(S)}{(p^2d)^2})\\
&+p^{-2}\chi(\alpha p^{-4})\chi(D(S)p^{-6})\sum_{\substack{d\mid(\alpha,2\beta,\gamma)\\p\nmid d}}(p^3d)^{k-1}C_\phi(\frac{D(S)}{(p^3d)^2})\\
 =&0.
 \end{align*}
 Here, we have used Lemmas \ref{mslemma} and \ref{mmlemma} to evaluate the character sum.
We are still in the case that $p\nmid \gamma$ so that $p^2\mid\mid 2\beta$ and $\chi(\alpha p^{-4})=\chi(\gamma)$.  Further we assume that $p^7\mid\mid D(S)$. Then,
\begin{align*}
a_\chi(S)&=b_\chi(S)+c_\chi(S)+d_\chi(S)\\
= &p^{-2}\Big(\sum_{\substack{b,x,y\in(\Z/p\Z)^\times\\x,y\not\equiv1(p)\\2\beta p^{-2}\equiv2\alpha p^{-4} b(p)}}\chi\big(y(1-x) (\alpha  p^{-4}(1-y)^{-1}(y-x)b^2 +2\beta  bp^{-2}-\gamma x^{-1})\big)\Big)\\
 &\qquad\cdot\Big(\sum_{\substack{d\mid(\alpha, 2\beta,\gamma)\\(p\nmid d}}(p^3d)^{k-1}C_\phi(\frac{D(S)}{(p^3d)^2})+\sum_{\substack{d\mid(\alpha, 2\beta,\gamma)\\p\nmid d}}(p^2d)^{k-1}C_\phi(\frac{D(S)}{(p^2d)^2})\Big)\\
  &+p^{-2}\Big(\sum_{\substack{b,x,y\in(\Z/p\Z)^\times\\x,y\not\equiv1(p)}}\chi\big(y(1-x) (\alpha  p^{-4}(1-y)^{-1}(y-x)b^2 +2\beta  bp^{-2}-\gamma x^{-1})\big)\Big)\\
 &\qquad\cdot\sum_{\substack{d\mid(\alpha, 2\beta,\gamma)\\p\nmid d}}(pd)^{k-1}C_\phi(\frac{D(S)}{(pd)^2})\\
 &+p^{-2}\big(\chi(\gamma)+\chi(-\alpha p^{-4})\big)\sum_{\substack{d\mid(\alpha, 2\beta,\gamma)\\p\nmid d}}(pd)^{k-1}C_\phi(\frac{D(S)}{(pd)^2})\\
 &-p^{-2}\chi(-\alpha p^{-4})\sum_{\substack{d\mid(\alpha, 2\beta,\gamma)\\p\nmid d}}\big((p^2d)^{k-1}C_\phi(\frac{D(S)}{(p^2d)^2})+(p^2d)^{k-1}C_\phi(\frac{p^{-2}D(S)}{(p^2d)^2})\big)\\
&-p^{-1}\chi( \alpha p^{-4})\sum_{\substack{d\mid(\alpha, 2\beta,\gamma)\\p\nmid d}}(p^2d)^{k-1}C_\phi(\frac{D(S)}{(p^2d)^2})- p^{-2}\chi(\alpha p^{-4})\sum_{\substack{d\mid(\alpha, 2\beta,\gamma)\\p\nmid d}}(p^3d)^{k-1}C_\phi(\frac{D(S)}{(p^3d)^2})\\
 &+p^{-2}\chi(\alpha p^{-4})\sum_{\substack{d\mid(\alpha, 2\beta,\gamma)\\p\nmid d}} (pd)^{k-1}C_\phi(\frac{D(S)}{(pd)^2})\\
 &+p^{-2}(p-1)\chi(\alpha p^{-4})\sum_{\substack{d\mid(\alpha,2\beta,\gamma)\\p\nmid d}}(p^2d)^{k-1}C_\phi(\frac{D(S)}{(p^2d)^2})\\
 =&p^{-2}\Big(\chi(\gamma)+\chi(-\gamma)\Big)\cdot\Big(\sum_{\substack{d\mid(\alpha, 2\beta,\gamma)\\(p\nmid d}}(p^3d)^{k-1}C_\phi(\frac{D(S)}{(p^3d)^2})+\sum_{\substack{d\mid(\alpha, 2\beta,\gamma)\\p\nmid d}}(p^2d)^{k-1}C_\phi(\frac{D(S)}{(p^2d)^2})\Big)\\
  &-p^{-2}\Big(2\chi(\gamma)+\chi(-\alpha p^{-4})\Big)\cdot\sum_{\substack{d\mid(\alpha, 2\beta,\gamma)\\p\nmid d}}(pd)^{k-1}C_\phi(\frac{D(S)}{(pd)^2})\\
 &+p^{-2}\big(\chi(\gamma)+\chi(-\alpha p^{-4})\big)\sum_{\substack{d\mid(\alpha, 2\beta,\gamma)\\p\nmid d}}(pd)^{k-1}C_\phi(\frac{D(S)}{(pd)^2})\\
 &-p^{-2}\chi(-\alpha p^{-4})\sum_{\substack{d\mid(\alpha, 2\beta,\gamma)\\p\nmid d}}\big((p^2d)^{k-1}C_\phi(\frac{D(S)}{(p^2d)^2})+(p^3d)^{k-1}C_\phi(\frac{D(S)}{(p^3d)^2})\big)\\
&-p^{-1}\chi( \alpha p^{-4})\sum_{\substack{d\mid(\alpha, 2\beta,\gamma)\\p\nmid d}}(p^2d)^{k-1}C_\phi(\frac{D(S)}{(p^2d)^2})- p^{-2}\chi(\alpha p^{-4})\sum_{\substack{d\mid(\alpha, 2\beta,\gamma)\\p\nmid d}}(p^3d)^{k-1}C_\phi(\frac{D(S)}{(p^3d)^2})\\
 &+p^{-2}\chi(\alpha p^{-4})\sum_{\substack{d\mid(\alpha, 2\beta,\gamma)\\p\nmid d}} (pd)^{k-1}C_\phi(\frac{D(S)}{(pd)^2})
+p^{-2}(p-1)\chi(\alpha p^{-4})\sum_{\substack{d\mid(\alpha,2\beta,\gamma)\\p\nmid d}}(p^2d)^{k-1}C_\phi(\frac{D(S)}{(p^2d)^2})\\
=&0.
\end{align*}
Here we have again used Lemmas \ref{mslemma} and \ref{mmlemma} in evaluating the character sums.  We now consider the case where $p^8\mid D(S)$.  We first assume that $p\mid\gamma$.  In this case, we have that $p^3\mid2\beta$ and in fact $p^2\mid\gamma$.
\begin{align*}
a_\chi(S)&=b_\chi(S)+c_\chi(S)+d_\chi(S)\\
 =& p^{-2} (p-1)\Big(\chi(\alpha p^{-4} )+\chi(-\alpha p^{-4})\Big)\cdot\sum_{\substack{d\mid(\alpha, 2\beta,\gamma)\\p\nmid d}}(pd)^{k-1}C_\phi(\frac{D(S)}{(pd)^2})\\
  &-p^{-2}\chi( -\alpha p^{-4})(p-1) \sum_{\substack{d\mid(\alpha, 2\beta,\gamma)\\p\nmid d}}\big((pd)^{k-1}C_\phi(\frac{D(S)}{(pd)^2})+(p^2d)^{k-1}C_\phi(\frac{D(S)}{(p^2d)^2})\\
  &\qquad+(p^3d)^{k-1}C_\phi(\frac{D(S)}{(p^3d)^2})\big)\\
 &+p^{-2}(p-1)\chi(-\alpha p^{-4})\sum_{\substack{d\mid(\alpha, 2\beta,\gamma)\\p\nmid d}}\big((p^2d)^{k-1}C_\phi(\frac{D(S)}{(p^2d)^2})+(p^3d)^{k-1}C_\phi(\frac{D(S)}{(p^3d)^2})\big)\\
&-p^{-2}(p-1)\chi(\alpha p^{-4})\sum_{\substack{d\mid(\alpha, 2\beta,\gamma)\\p\nmid d}} (pd)^{k-1}C_\phi(\frac{D(S)}{(pd)^2})\\
=&0.
\end{align*}
 Finally, consider the case where $p^8\mid D(S)$ and $p\nmid\gamma$.  Then
 \begin{align*}
 a_\chi(S)&=b_\chi(S)+c_\chi(S)+d_\chi(S)\\
 =&p^{-2}\Big(\sum_{\substack{b,x,y\in(\Z/p\Z)^\times\\x,y\not\equiv1(p)\\2\beta p^{-2}\equiv2\alpha p^{-4} b(p)}}\chi\big(y(1-x) (\alpha  p^{-4}(1-y)^{-1}(y-x)b^2 +2\beta  bp^{-2}-\gamma x^{-1})\big)\Big)\\
 &\qquad\cdot\Big(\sum_{\substack{d\mid(\alpha, 2\beta,\gamma)\\(p\nmid d}}(p^3d)^{k-1}C_\phi(\frac{D(S)}{(p^3d)^2})
+\sum_{\substack{d\mid(\alpha, 2\beta,\gamma)\\p\nmid d}}(p^2d)^{k-1}C_\phi(\frac{D(S)}{(p^2d)^2})\Big)\\
  &+p^{-2}\Big(\sum_{\substack{b,x,y\in(\Z/p\Z)^\times\\x,y\not\equiv1(p)}}\chi\big(y(1-x) (\alpha  p^{-4}(1-y)^{-1}(y-x)b^2 +2\beta  bp^{-2}-\gamma x^{-1})\big)\Big)\\
 &\qquad\cdot\sum_{\substack{d\mid(\alpha, 2\beta,\gamma)\\p\nmid d}}(pd)^{k-1}C_\phi(\frac{D(S)}{(pd)^2})\\
 &+p^{-2}\big(\chi(\gamma)+\chi(-\alpha p^{-4})\big)\sum_{\substack{d\mid(\alpha, 2\beta,\gamma)\\p\nmid d}}(pd)^{k-1}C_\phi(\frac{D(S)}{(pd)^2})\\
   &-p^{-2}\chi(-\alpha p^{-4})\sum_{\substack{d\mid(\alpha, 2\beta,\gamma)\\p\nmid d}}\big((p^2d)^{k-1}C_\phi(\frac{D(S)}{(p^2d)^2})+(p^3d)^{k-1}C_\phi(\frac{D(S)}{(p^3d)^2})\big)\\
&-p^{-1}\chi(\alpha p^{-4})\sum_{\substack{d\mid(\alpha, 2\beta,\gamma)\\p\nmid d}}(p^2d)^{k-1}C_\phi(\frac{D(S)}{(p^2d)^2})\\
 &-p^{-2}\chi(\alpha p^{-4})\sum_{\substack{d\mid(\alpha, 2\beta,\gamma)\\p\nmid d}}\big((p^3d)^{k-1}C_\phi(\frac{D(S)}{(p^3d)^2})+(p^4d)^{k-1}C_\phi(\frac{D(S)}{(p^4d)^2})\big)\\
 &+p^{-2}\chi(\alpha p^{-4})\sum_{\substack{d\mid(\alpha, 2\beta,\gamma)\\p\nmid d}} (pd)^{k-1}C_\phi(\frac{D(S)}{(pd)^2})\\
 &+p^{-2}(p-1)\chi(\alpha p^{-4})\sum_{\substack{d\mid(\alpha,2\beta,\gamma)\\p\nmid d}}(p^2d)^{k-1}C_\phi(\frac{D(S)}{(p^2d)^2})\\
&+p^{-2}\chi(\alpha p^{-4})\sum_{\substack{d\mid(\alpha,2\beta,\gamma)\\p\nmid d}}(p^4d)^{k-1}C_\phi(\frac{D(S)}{(p^4d)^2})\\
=&p^{-2}\Big(\chi(\gamma)+\chi(-\gamma)\Big)\cdot\Big(\sum_{\substack{d\mid(\alpha, 2\beta,\gamma)\\(p\nmid d}}(p^3d)^{k-1}C_\phi(\frac{D(S)}{(p^3d)^2})+\sum_{\substack{d\mid(\alpha, 2\beta,\gamma)\\p\nmid d}}(p^2d)^{k-1}C_\phi(\frac{D(S)}{(p^2d)^2})\Big)\\
  &-p^{-2}\Big(2\chi(\gamma)+\chi(-\alpha p^{-4})\Big)\qquad\cdot\sum_{\substack{d\mid(\alpha, 2\beta,\gamma)\\p\nmid d}}(pd)^{k-1}C_\phi(\frac{D(S)}{(pd)^2})\\
 &+p^{-2}\big(\chi(\gamma)+\chi(-\alpha p^{-4})\big)\sum_{\substack{d\mid(\alpha, 2\beta,\gamma)\\p\nmid d}}(pd)^{k-1}C_\phi(\frac{D(S)}{(pd)^2})\\
   &-p^{-2}\chi(-\alpha p^{-4})\sum_{\substack{d\mid(\alpha, 2\beta,\gamma)\\p\nmid d}}\big((p^2d)^{k-1}C_\phi(\frac{D(S)}{(p^2d)^2})+(p^3d)^{k-1}C_\phi(\frac{D(S)}{(p^3d)^2})\big)\\
&-p^{-1}\chi(\alpha p^{-4})\sum_{\substack{d\mid(\alpha, 2\beta,\gamma)\\p\nmid d}}(p^2d)^{k-1}C_\phi(\frac{D(S)}{(p^2d)^2})\\
 &-p^{-2}\chi(\alpha p^{-4})\sum_{\substack{d\mid(\alpha, 2\beta,\gamma)\\p\nmid d}}\big((p^3d)^{k-1}C_\phi(\frac{D(S)}{(p^3d)^2})+(p^4d)^{k-1}C_\phi(\frac{D(S)}{(p^4d)^2})\big)\\
 &+p^{-2}\chi(\alpha p^{-4})\sum_{\substack{d\mid(\alpha, 2\beta,\gamma)\\p\nmid d}} (pd)^{k-1}C_\phi(\frac{D(S)}{(pd)^2})\\
 &+p^{-2}(p-1)\chi(\alpha p^{-4})\sum_{\substack{d\mid(\alpha,2\beta,\gamma)\\p\nmid d}}(p^2d)^{k-1}C_\phi(\frac{D(S)}{(p^2d)^2})\\
&+p^{-2}\chi(\alpha p^{-4})\sum_{\substack{d\mid(\alpha,2\beta,\gamma)\\p\nmid d}}(p^4d)^{k-1}C_\phi(\frac{D(S)}{(p^4d)^2})\\
=&0.
\end{align*}
 Hence, $a_\chi(S)=0$.\\
 
 \noindent{\bf Case (v):}  By assumption, $p^2\mid 2\beta$ and $p^5\mid \alpha$.  Then by Theorem \ref{maintheorem} we have 
 \begin{align*}
a_\chi(S) =&(1-p^{-1})\chi(\gamma)a(S)-p^{-1}\chi( \gamma)\sum_{\substack{b\in(\Z/p\Z)^\times}}a(S[\begin{bmatrix}1&-bp^{-1}\\&1\end{bmatrix}])\\
 &-p^{k-3}\sum_{b\in(\Z/p\Z)^\times}\sum_{\substack{x\in(\Z/p\Z)^\times\\x\not\equiv1(p)}}\chi\big((1-x) (2\beta  bp^{-2}-\gamma x^{-1})\big)a(S[\begin{bmatrix}p^{-1}&-bp^{-2}\\&1\end{bmatrix}])\\
&+p^{k-3}\chi(\gamma)(1-p+p\chi(-4\det(S)p^{-4}))a(S[\begin{bmatrix}p^{-1}&\\&1\end{bmatrix}])\\
&-p^{2k-4}\chi( -\gamma)\sum\limits_{\substack{b\in(\Z/p^2\Z)^\times\\p^2\mid(\alpha p^{-4}b^2-2\beta p^{-2}b+\gamma)}}a(S[\begin{bmatrix}p^{-1}&-bp^{-3}\\&p^{-1}\end{bmatrix}])\\
=&(1-p^{-1})\chi(\gamma)a(\begin{bmatrix}\alpha&\beta\\\beta&\gamma\end{bmatrix})-p^{-1}\chi( \gamma)\sum_{\substack{b\in(\Z/p\Z)^\times}}a(\begin{bmatrix}\alpha&\beta-b\alpha p^{-1}\\\beta-b\alpha p^{-1}&\gamma-2\beta bp^{-1}+b^2\alpha p^{-2}\end{bmatrix})\\
 &-p^{k-3}\sum_{\substack{b,x\in(\Z/p\Z)^\times\\x\not\equiv1(p)}}\chi\big((1-x) (2\beta  bp^{-2}-\gamma x^{-1})\big)a(\begin{bmatrix}\alpha p^{-2}&\beta p^{-1}-b\alpha p^{-3}\\\beta p^{-1}-b\alpha p^{-3}&\gamma-b(2\beta p^{-2})+b^2\alpha p^{-4}\end{bmatrix})\\
&+p^{k-3}\chi(\gamma)(1-p+p\chi(-4\det(S)p^{-4}))a(\begin{bmatrix}\alpha p^{-2}&\beta p^{-1}\\\beta p^{-1}&\gamma\end{bmatrix})\\
&-p^{2k-4}\chi( -\gamma)\sum\limits_{\substack{b\in(\Z/p^2\Z)^\times\\p^2\mid(\alpha p^{-4}b^2-2\beta p^{-2}b+\gamma)}}a(\begin{bmatrix}\alpha p^{-2}&\beta p^{-2}-b\alpha p^{-4}\\\beta p^{-2}-b\alpha p^{-4}&p^{-2}(\gamma-2\beta p^{-2} b+b^2\alpha p^{-4})\end{bmatrix})\\
\end{align*}
We consider the first and second summands together.  They both clearly vanish in the case that $p\mid\gamma$.  Assume that $p\nmid\gamma$.  Then $p\nmid \gamma-2\beta bp^{-1}+b^2\alpha p^{-2}$.  Hence, applying the formula \eqref{maassfourier}, we see that the two terms sum to zero.
 For the third summand, we consider two cases.  First, assume that $p\mid \gamma$.  We claim that the summand is zero.  Evidently, $\chi( -\gamma x^{-1}+2\beta b p^{-2})=0$ if $p^3\mid2\beta$, so assume that $p^2\mid\mid 2\beta$.  This implies that $p\nmid \gamma-b(2\beta p^{-2})+b^2\alpha p^{-4}$.  Hence in this case, the third term is 
$$
-p^{k-3}\Big(\sum_{\substack{b,x\in(\Z/p\Z)^\times\\x\not\equiv1(p)}}\chi\big((1-x) 2\beta  bp^{-2}\big)\Big)\cdot\sum_{\substack{d\mid(\alpha, 2\beta,\gamma)\\p\nmid d}}d^{k-1}C_\phi(\frac{p^{-2}D(S)}{d^2}).
$$
Summing over $b$, we see that this is zero.  Now consider the case when $p\nmid\gamma$.  Then, $p\mid(\alpha',2\beta',\gamma')$ if and only if $\gamma\equiv 2\beta p^{-2} b (p)$, and moreover $p^2\nmid (\alpha',2\beta',\gamma')$.  Hence the third summand is 
\begin{align*}
&-p^{k-3}\Big(\sum_{\substack{b,x\in(\Z/p\Z)^\times\\x\not\equiv1(p)\\\gamma\equiv 2\beta p^{-2} b (p)}}\chi(-\gamma x)\Big)\cdot \sum_{\substack{d\mid(\alpha, 2\beta,\gamma)\\p\nmid d}}(d^{k-1}C_\phi(\frac{p^{-2}D(S)}{d^2})+(pd)^{k-1}C_\phi(\frac{p^{-2}D(S)}{(pd)^2}))\\
&-p^{k-3}\Big(\sum_{\substack{b,x\in(\Z/p\Z)^\times\\x\not\equiv1(p)\\\gamma\not\equiv 2\beta p^{-2} b (p)}}\chi\big((1-x)( 2\beta  bp^{-2}-\gamma x^{-1})\big)\Big)\cdot\sum_{\substack{d\mid(\alpha, 2\beta,\gamma)\\p\nmid d}}d^{k-1}C_\phi(\frac{p^{-2}D(S)}{d^2})\\
=&-p^{k-3}\Big(\sum_{\substack{b,x\in(\Z/p\Z)^\times\\x\not\equiv1(p)\\\gamma\equiv 2\beta p^{-2} b (p)}}\chi(-\gamma x) \Big)\cdot\sum_{\substack{d\mid(\alpha, 2\beta,\gamma)\\p\nmid d}}(pd)^{k-1}C_\phi (\frac{p^{-2}D(S)}{(pd)^2})\\
&-p^{k-3}\Big(\sum_{\substack{b,x\in(\Z/p\Z)^\times\\x\not\equiv1(p)}}\chi\big((1-x)( 2\beta  bp^{-2}-\gamma x^{-1})\big)\Big)\cdot\sum_{\substack{d\mid(\alpha, 2\beta,\gamma)\\p\nmid d}}d^{k-1}C_\phi(\frac{p^{-2}D(S)}{d^2})\\
=&p^{k-3}\chi(-\gamma)\sum_{\substack{b\in(\Z/p\Z)^\times\\\gamma\equiv 2\beta p^{-2} b (p)}}\sum_{\substack{d\mid(\alpha, 2\beta,\gamma)\\p\nmid d}}(pd)^{k-1}C_\phi (\frac{p^{-2}D(S)}{(pd)^2})\\
&+p^{k-3}\chi(\gamma)W(\mathbf{1},2\beta p^{-2})\sum_{\substack{d\mid(\alpha, 2\beta,\gamma)\\p\nmid d}}d^{k-1}C_\phi(\frac{p^{-2}D(S)}{d^2}).
\end{align*}
For the fourth summand, we see that the term is zero if $p\mid \gamma$, so assume that $p\nmid\gamma$.  Then, the term is 
\begin{align*}
&p^{k-3}\chi(\gamma)(1-p+p\chi(-4\det(S)p^{-4}))\sum_{\substack{d\mid(\alpha, 2\beta,\gamma)\\p\nmid d}}d^{k-1}C_\phi(\frac{p^{-2}D(S)}{d^2}).
\end{align*}
For the fifth summand, we observe first that if $p\mid\gamma$ or $p^3\mid 2\beta$ the sum vanishes. So assume, that $p\nmid\gamma$ and $p^2\mid\mid 2\beta$.  Then, we see that there is a unique solution to the polynomial modulo $p^2$, so that the fifth summand is
\begin{align*}
&-p^{2k-4}\chi( -\gamma)\sum_{\substack{d\mid(\alpha, 2\beta,\gamma)\\p\nmid d}}d^{k-1}C_\phi(\frac{p^{-4}D(S)}{d^2}).
\end{align*}
Adding the five summands together, we see that $a_\chi(S)=0$.
\end{proof}

\section{Lemmas}
\label{lemmassec}
Continuing with the notation of Section \ref{notation} for $p$ and $\chi$, we collect several technical lemmas.   Let $A,B,C \in \Z$ and set $D=B^2-4AC$. In this section we will often  use the following polynomial identity:
\begin{equation}
\label{quadideq}
4A(AX^2+BX+C)=(2AX+B)^2-D.
\end{equation}
The following lemmas are applied in the proofs of Theorem \ref{maintheorem} and Corollary \ref{corollary}, typically with $A=\alpha p^{-4}$, $B=-2\beta p^{-2}$, and $C=\gamma$.

\begin{lemma}
\label{p22lemma}
Let $s_1,s_2 \in \Z$ with $(s_1,p)=(s_2,p)=1$. If $s_1^2 \equiv s_2^2\ (p^2)$, then $s_1 \equiv \pm s_2\ (p^2)$. 
\end{lemma}
\begin{proof}
Assume that $s_1 \not\equiv s_2\ (p^2)$ and $s_1 \not\equiv -s_2\ (p^2)$; we will obtain a contradiction.  We have $(s_1-s_2)(s_1+s_2) \equiv 0\ (p^2)$. It follows from  our assumption that $s_1 \equiv s_2\ (p)$ and $s_1 \equiv -s_2\ (p)$. Therefore, $2s_1 \equiv 0\ (p)$. This contradicts $(s_1,p) = 1$. 
\end{proof}

\begin{lemma}
\label{a2lemma}
Let $A,B,C \in \Z$, and assume that $(A,p)=1$. Set $f(X) = AX^2+BX+C$. Let $D =B^2-4AC$. For $r \in \Z$ and $k \in \{1,2\}$ we say that $r$ satisfies $R(k)$ if 
$$
\text{R(k):\quad $2Ar +B \equiv 0 \, (p^{k-1})$ and $f(r) \equiv 0\, (p^{k}).$}
$$
Let $r_1,r_2 \in \Z$ and $k \in \{1,2\}$. Assume that $r_1\equiv r_2\, (p)$. Then $r_1$ satisfies $R(k)$ if and only if $r_2$ satisfies $R(k)$. Thus, the concept of an element of $\Z/p\Z$ satisfying $R(k)$ is well-defined. 
\renewcommand{\theenumi}{\roman{enumi}}%
\begin{enumerate}
\item Assume that $(D,p)=1$ and  $D$ is not a square mod $p$. Then no element of $\Z/p\Z$ satisfies $R(k)$ for $k \in \{1,2\}$.
\item Assume that $(D,p)=1$ and $D$ is a square mod $p$, so that there exists $s \in \Z$ such that $D \equiv s^2\, (p)$. The set of elements of $\Z/p \Z$ that satisfy $R(1)$ is $\{(-B+s)(2A)^{-1},(-B-s)(2A)^{-1}\}$. No element of $ \Z/p\Z$  satisfies $R(2)$.
\item Assume that $p||D$. Then the set of elements of $\Z/p\Z$ that satisfy $R(1)$ is $\{-B(2A)^{-1}\}$. No element of $\Z/p\Z$ satisfies  $R(2)$.
\item Assume that $p^2|D$. Then the set of elements of $\Z/p\Z$ that satisfy $R(1)$ is $\{-B(2A)^{-1}\}$, and the set of elements of $\Z/p\Z$ that satisfy $R(2)$ is $\{-B(2A)^{-1}\}$. 
\end{enumerate}
\end{lemma}
\begin{proof}
Suppose that $r_1$ satisfies $R(k)$. Write $r_2=r_1+np$ for some $n \in \Z$. Then $2Ar_2+B \equiv 2Ar_1+B +2Anp \equiv 0\, (p^{k-1})$. Also, using \eqref{quadideq}, a computation shows that $4Af(r_2) \equiv 0\, (p^k)$. Hence, $r_2$ satisfies $R(k)$. 

(i) This follows from \eqref{quadideq}. 

(ii) The identity \eqref{quadideq} implies that the elements of $\{(-B+s)(2A)^{-1},(-B-s)(2A)^{-1}\} \subset \Z/p\Z$ satisfy $R(1)$. 
Conversely, assume that $r \in \Z$ satisfies $R(1)$. Then \eqref{quadideq} implies that $(2Ar+B)^2 \equiv D \equiv s^2\, (p)$; hence, the image of $r$ in $\Z/p\Z$ is in $ \{(-B+s)(2A)^{-1},(-B-s)(2A)^{-1}\}\subset \Z/p\Z $. Assume that $r \in \Z$ satisfies $R(2)$. Then by \eqref{quadideq} we have $(2Ar+B)^2 \equiv D \equiv s^2\, (p)$. Therefore, $2Ar+B  \equiv \pm s\, (p)$. By $R(2)$, we also have $2Ar+B \equiv 0\, (p)$. Hence, $s \equiv 0\, (p)$, a contradiction. 

(iii) The identity \eqref{quadideq} and the assumption $p||D$ imply that the element of $\{-B(2A)^{-1}\} \subset \Z/p\Z$ satisfies $R(1)$. Conversely, suppose that $r \in \Z$ satisfies $R(1)$. By \eqref{quadideq} we have $(2Ar+B)^2 \equiv 0\, (p)$, so that $r = -B(2A)^{-1}$ in $\Z/p\Z$. 
Suppose that there exists $r \in \Z$ that satisfies $R(2)$; we will obtain a contradiction. By \eqref{quadideq} we have $(2Ar+B)^2 \equiv D \equiv 0\, (p^2)$. Therefore, since $D \equiv 0\, (p)$,  $2Ar+B \equiv 0\, (p)$. This implies that $(2Ar+B)^2 \equiv 0\, (p^2)$. Hence, $D \equiv 0\, (p^2)$, contradicting $p||D$. 

(iv) This is similar to previous cases.
\end{proof}

\begin{lemma}
\label{quadeqlemma}
Let $A,B,C \in \Z$, and assume that $(A,p)=1$. Set $f(X) = AX^2+BX+C$. Let $D =B^2-4AC$. Let $S$ be the set of $r \in \Z/p^2\Z$ such that $f(r) \equiv 0\ (p^2)$. 
\renewcommand{\theenumi}{\roman{enumi}}%
\begin{enumerate}
\item Assume that $(D,p)=1$. Then $S$ is non-empty if and only if $D$ is a square mod $p^2$. Moreover, if $D$ is a square in $\Z/p^2\Z$, and $s \in \Z$ is such that $D \equiv s^2\ (p^2)$, then $S=\{(-B+s)(2A)^{-1}, (-B-s)(2A)^{-1} \}$.
\item If $p||D$, then $S$ is empty.
\item If $p^2|D$, then $S=\{-B(2A)^{-1}+py: y \in \Z/p\Z\}$. 
\end{enumerate}
\end{lemma}
\begin{proof}
(i) Assume that $r \in \Z$ is such that $f(r) \equiv 0\ (p^2)$. By \eqref{quadideq}, 
$(2Ar+B)^2 \equiv D\ (p^2)$. Thus, $D$ is a square in $\Z/p^2\Z$. Conversely, assume that $D$ is a square in $\Z/p^2\Z$, and let $s \in \Z$ be such that $D \equiv s^2\ (p^2)$. By \eqref{quadideq} if $r \in \Z$ is such that $r \equiv (-B\pm s)(2A)^{-1}\ (p^2)$, then $f(r)\equiv 0\ (p^2)$. It follows that $(-B\pm s)(2A)^{-1} \in S$. Finally, assume that $r' \in S$. Again, we have $(2Ar'+B)^2 \equiv D\ (p^2)$. By Lemma \ref{p22lemma} we have $2Ar'+B\equiv \pm s\ (p^2)$, i.e., $r'\equiv(-B\pm s)(2A)^{-1}\ (p^2)$. Thus, $S$ is as specified.

(ii) Assume that $S$ is non-empty; we will obtain a contradiction. Let $r \in S$.
By \eqref{quadideq}, $(2Ar+B)^2 \equiv D\ (p^2)$. Since $p|D$, it follows that $2Ar+B\equiv 0\ (p)$. This implies, in turn, that $p^2 |D$, a contradiction. 

(iii) In this case we have the identity $f(X) \equiv (2AX+B)^2\ (p^2)$. It follows that if $r \in \{-B(2A)^{-1}+py: y \in \Z/p\Z\}$, then $f(r) \equiv 0\ (p^2)$, i.e., $r \in S$. Conversely, suppose that $r \in S$. By the just mentioned identity, $(r+B(2A)^{-1})^2\equiv 0\ (p^2)$. This implies that $r+B(2A)^{-1} \equiv 0\ (p)$, so that $r \in \{-B(2A)^{-1}+py: y \in \Z/p\Z\}$. 
\end{proof}

\begin{lemma}
\label{ccondlemma}
Let the notation be as in Lemma \ref{quadeqlemma}. Let $r \in \Z$, and assume that the reduction mod $p^2$ of $r$ is in $S$, so that $f(r) \equiv 0\ (p^2)$; by Lemma \ref{quadeqlemma}, $(D,p)=1$, or $p^2|D$. For $k$ a positive integer, define the statement $R(k)$ as:
$$
\text{R(k): \quad $(2A)r \equiv -B\ (p^k)$ and $f(r) \equiv 0\ (p^{k+2}).$}
$$
Then:
\renewcommand{\theenumi}{\roman{enumi}}%
\begin{enumerate}
\item If $(p,D)=1$, then $R(1)$ does not hold.
\item Assume that $p^2||D$. Then R(1) holds if and only if  $r \equiv -B(2A)^{-1}+py\ (p^2)$ for some $y \in \Z$ with $(2A y)^2 \equiv Dp^{-2}\ (p)$. Also, R(2) does not hold. 
\item Assume that $p^3|| D$. Then R(1) holds if and only if $r \equiv -B(2A)^{-1}\ (p^2)$. Moreover, R(2) does not hold. 
\item Assume that $p^4 |D$. Then R(1) holds if and only if $r \equiv -B(2A)^{-1} (p^2)$. Also, R(2) holds if and only if $r \equiv -B(2A)^{-1} (p^2)$. 
\end{enumerate}
\end{lemma}
\begin{proof}
(i) Assume that $R(1)$ holds; we will obtain a contradiction. By \eqref{quadideq}, $(2Ar+B)^2 \equiv D\ (p^3)$; this implies $(2Ar+B)^2 \equiv D\ (p)$. By $R(1)$, $2Ar+B\equiv 0\ (p)$. Hence, $D \equiv 0\ (p)$, a contradiction. 

(ii) Assume that $p^2||D$. By (3) of Lemma \ref{quadeqlemma}, there exists $y,z \in \Z$ such that $r \equiv -B(2A)^{-1} +py+p^2z\ (p^3)$. By \eqref{quadideq}, $f(r) \equiv 0\ (p^3)  \iff(2A y)^2 \equiv Dp^{-2}\ (p)$. The first assertion of (2) follows from this equivalence. Assume that $R(2)$ holds; we will obtain a contradiction. Since $R(2)$ holds we have $(2A)r \equiv -B\ (p^2)$. On the other hand, $R(2)$ implies $R(1)$, so that by the first assertion of (ii) we have $r \equiv -B(2A)^{-1}+py\ (p^2)$ for some $y \in \Z$ with $(2A y)^2 \equiv Dp^{-2}\ (p)$; in particular, $(y,p)=1$ since $p^2||D$. This is a contradiction. 

(iii) Assume that $p^3||D$. Assume that $R(1)$ holds. Since $R(1)$ holds by \eqref{quadideq} we have $(2Ar+B)^2 \equiv D \equiv 0\ (p^3)$. Therefore, $r \equiv -B(2A)^{-1} (p^2)$. Assume that $r \equiv -B(2A)^{-1} (p^2)$. Then $4Af(r)\equiv (2Ar+B)^2-D  \equiv 0\ (p^3)$. Hence, R(1) holds. Assume that $R(2$) holds; we will obtain a contradiction. Now  $0 \equiv 4Af(r)\equiv (2Ar+B)^2 - D \equiv -D\ (p^4)$. This implies that $p^4 |D$, a contradiction. 

(iv) Assume that $p^4|D$.  Assume that $R(1)$ holds. By \eqref{quadideq},   $0\equiv 4A f(r)\equiv (2Ar+B)^2\ (p^3)$. It follows that $r \equiv -B(2A)^{-1} (p^2)$. Assume that $r \equiv -B(2A)^{-1} (p^2)$. Then \eqref{quadideq} implies that $f(r)\equiv 0\ (p^4)$, i.e., $R(2)$ holds. Finally, since $R(2)$ implies $R(1)$, the proof is complete. 
\end{proof}

\begin{lemma}
\label{mslemma}
Let $A,B,C \in \Z$. Assume that $(A,p)=1$ and $B^2-4AC \equiv 0\, (p)$. Then
\begin{align*}
&\sum_{\substack{b,x,y\in(\Z/p\Z)^\times\\x,y\not\equiv1(p)\\ b \equiv -B(2A)^{-1}\, (p)}}
\chi\big(y(1-x)( A (1-y)^{-1}(y-x)b^2 -B b-C x^{-1} )\big)\\
=&\sum_{\substack{b,x,y\in(\Z/p\Z)^\times\\x,y\not\equiv1(p)\\f_S(b)\equiv0(p)}}\chi\big(y(A(1-y)^{-1}b^2  -Cx^{-1})\big)=\chi(C)+\chi(-C).
\end{align*}
\end{lemma}
\begin{proof}
To prove the first equality, substitute $-Bb\equiv Ab^2+C(p)$. To complete the proof, we calculate
\begin{align*}
&\sum_{\substack{b,x,y\in(\Z/p\Z)^\times\\x,y\not\equiv1(p)\\ b \equiv -B(2A)^{-1}\, (p)}}
\chi\big(y(1-x)(A (1-y)^{-1}(y-x)b^2 -B b-C x^{-1} )\big)\\
&=\sum_{\substack{b,x,y\in(\Z/p\Z)^\times\\x,y\not\equiv1(p)}}
\chi\big(y(1-x)(A(1-y)^{-1}(y-x)(-B(2A)^{-1})^2)+B(-B(2A)^{-1}) -C x^{-1})\big)\\
&=\sum_{\substack{x,y\in(\Z/p\Z)^\times\\x,y\not\equiv1(p)}}
\chi\big(y(1-x) C(-x^{-1}+2+(1-y)^{-1}(y-x) )\big)\\
&=\sum_{\substack{x,y\in(\Z/p\Z)^\times\\x,y\not\equiv1(p)}}
\chi\big(y(1-x)C(x-1)(x+y-1)x^{-1}(1-y)^{-1}\big)\\
&=\chi(-C) \sum_{\substack{y\in(\Z/p\Z)^\times\\ y\not\equiv1(p)}}
\chi(y(1-y)) \sum_{\substack{x\in(\Z/p\Z)^\times\\ x\not\equiv1(p)}}  \chi(x) \chi(x+y-1)  \\
&=\chi(-C) \sum_{\substack{y\in(\Z/p\Z)^\times\\ y\not\equiv1(p)}}
\chi(y(1-y)) \big(-\chi(y)+ \sum_{\substack{x\in \Z/p\Z}}  \chi(x) \chi(x+y-1)  \big) \\
&=\chi(-C) \sum_{\substack{y\in(\Z/p\Z)^\times\\ y\not\equiv1(p)}}
\chi(y(1-y)) (-\chi(y)-1 ) \\
&=-\chi(-C) \sum_{\substack{y\in(\Z/p\Z)^\times\\ y\not\equiv1(p)}}
\chi(1-y) -\chi(-C) \sum_{\substack{y\in(\Z/p\Z)^\times\\ y\not\equiv1(p)}}
\chi(y(1-y))\\
&=\chi(-C)  +\chi(C).
\end{align*}
Here we have used Lemma \ref{jlemma}.
\end{proof}

\begin{lemma}
\label{mmlemma}
Let $A,B,C \in \Z$. Define
$$
M=\sum_{\substack{b,x,y\in(\Z/p\Z)^\times\\x,y\not\equiv1(p)}}
 \chi\big(y(1-x)(A (1-y)^{-1}(y-x)b^2-B b  -C x^{-1})\big).
$$
\renewcommand{\theenumi}{\roman{enumi}}%
\begin{enumerate}
\item If $A \not \equiv 0\, (p)$, $B \not \equiv 0 \, (p)$, $C \not\equiv 0\, (p) $ and $B^2-4AC$ is not a square mod $p$, then 
$$M=(p-1)\chi(C)-\chi(-A)+(p-1)\chi(A).$$
\item If $A \not \equiv 0\, (p)$, $B \not \equiv 0 \, (p)$, $C \not\equiv 0\, (p) $, $B^2-4AC  \not\equiv 0\, (p) $, and $B^2-4AC$ is a square mod $p$, then
$$M=-(p+1)\chi(C)-\chi(-A)-(p+1)\chi(A).$$
\item If $A \not \equiv 0\, (p)$, $B \not \equiv 0\, (p)$, $C \not \equiv 0\, (p)$ and $B^2-4AC \equiv 0\, (p)$, then
$$M=-2\chi(C)-\chi(-A).$$
\item If $A \not \equiv 0\, (p)$, $B \not \equiv  0\, (p)$ and $C\equiv  0\, (p)$, then
$$M=-\chi(-A)-\chi(A).$$
\item If $A \not \equiv 0\, (p)$, $B  \equiv 0\, (p)$, $C \not\equiv  0\, (p)$, and $B^2-4AC$ is a square mod $p$, then
$$M=-2\chi(C)-(p+1)\chi(A).$$
\item If $A \not \equiv 0\, (p)$, $B  \equiv 0\, (p)$, $C \not\equiv  0\, (p)$, and $B^2-4AC$ is not a square mod $p$, then
$$
M=(p-1)\chi(A).
$$
\item If $A \equiv 0\, (p)$, $B \not \equiv 0\, (p)$ and $C \not \equiv 0\, (p)$, then
$$
M=-\chi(C).
$$
\item If $A \not \equiv 0\, (p)$, $B \equiv 0\, (p)$ and $C \equiv 0\, (p)$, then
$$M=(p-1)\chi(-A)+(p-1)\chi(A).$$
\item If $A \equiv 0\, (p)$, $B \not \equiv 0\, (p)$ and $C \equiv 0\, (p)$, then
$$M=0.$$
\end{enumerate}

\end{lemma}
\begin{proof}
We have:
\begin{align*}
M&=\sum_{\substack{b,x,y\in(\Z/p\Z)^\times\\x,y\not\equiv1(p)}}
\chi((1-x)y) \chi( -C x^{-1}-B b+A(1-y)^{-1}(y-x)b^2  )\\
&=\sum_{\substack{b,x,y\in(\Z/p\Z)^\times\\x,y\not\equiv1(p)}}
\chi(x(1-x)y(1-y)) \\
&\hspace{1.5in}\cdot \chi(-C (1-y)-B(1-y)(xb)+A (y-x)x^{-1}(xb)^2 )\\
\intertext{(change variables $b\mapsto bx^{-1}$:)}
&=\sum_{\substack{b,x,y\in(\Z/p\Z)^\times\\x,y\not\equiv1(p)}}
\chi(x(1-x)y(1-y)) \chi(-C(1-y)-B(1-y)b+A(y-x) x^{-1}b^2 )\\
\intertext{(change variables $y \mapsto 1-y:$)}
&=\sum_{\substack{b,x,y\in(\Z/p\Z)^\times\\x,y\not\equiv1(p)}}
\chi(x(1-x)y(1-y)) \chi(-C y -B yb+A (1-y-x) x^{-1}b^2 )\\
&=\sum_{\substack{b,x,y\in(\Z/p\Z)^\times\\x,y\not\equiv1(p)}}
\chi(x^{-1}-1)\chi(y(1-y)) \chi(A (1-y)b^2  x^{-1} -C y -Byb-Ab^2 ))\\
\intertext{(change variables $x \mapsto x^{-1}$:)}
&=\sum_{\substack{b,x,y\in(\Z/p\Z)^\times\\x,y\not\equiv1(p)}}
\chi(y(1-y))\chi(x-1) \chi(A (1-y)b^2  x -C y -Byb-A b^2 )\\
&=
\Big( -
\sum_{\substack{y\in(\Z/p\Z)^\times\\y\not\equiv1(p)}} \chi(y(1-y))\chi(-1)  \big(-\chi(-C y ) + \sum_{\substack{b \in \Z/p\Z}}
\chi(-C y -Byb-A b^2 )\big)\\
&\qquad+\sum_{\substack{b,y\in(\Z/p\Z)^\times\\y\not\equiv1(p)}}\chi(y(1-y))\sum_{\substack{x\in \Z/p\Z}}
\chi(x-1) \chi(A (1-y)b^2  x -C y -Byb-A b^2 ) \Big)\\
&=
\Big(
 \chi(C  )\big( -1+\sum_{\substack{y\in \Z/p\Z}} \chi(1-y)\big) \\
&\qquad -\chi(-1) \sum_{\substack{y\in(\Z/p\Z)^\times\\y\not\equiv1(p)}} \chi(y(1-y)) \sum_{\substack{b \in \Z/p\Z}}
\chi(-C y -B yb-A b^2 ) \\
&\qquad\qquad+\sum_{\substack{b,y\in(\Z/p\Z)^\times\\y\not\equiv1(p)}}\chi(y(1-y))\sum_{\substack{x\in \Z/p\Z}}
\chi(x-1) \chi(A (1-y)b^2  x -C y -B yb-A b^2 ) \Big)\\
M&=
\Big(
 -\chi(C  )
-\chi(-1) \sum_{\substack{y\in(\Z/p\Z)^\times\\y\not\equiv1(p)}} \chi(y(1-y)) \sum_{\substack{b \in \Z/p\Z}}
\chi(-C y -B yb-A b^2 ) \\
&\qquad\qquad+\sum_{\substack{b,y\in(\Z/p\Z)^\times\\y\not\equiv1(p)}}\chi(y(1-y))\sum_{\substack{x\in \Z/p\Z}}
\chi(x-1) \chi(A (1-y)b^2  x -C y -B yb-A b^2 ) \Big).
\end{align*}

(i) In this case,
\begin{align*}
M&=
\Big(
 -\chi(C  ) -\chi(-1) \chi\big( (4AC)B^{-2}(1-(4AC)B^{-2}) \big)(p-1)\chi(-A)\\
&\qquad\qquad-\chi(-1) \sum_{\substack{y\in(\Z/p\Z)^\times\\y\not\equiv1(p)\\ y \not\equiv (4AC)B^{-2}(p)}} \chi(y(1-y)) (-\chi(-A))\\
&\qquad\qquad+\sum_{\substack{b,y\in(\Z/p\Z)^\times\\y\not\equiv1(p)}}\chi(y(1-y))(-\chi(A(1-y))) \Big)\\
&=
\Big((p-1)\chi(C)+\chi(A) \sum_{\substack{y\in(\Z/p\Z)^\times\\y\not\equiv1(p)}} \chi(y(1-y))-\chi(A)(p-1)\sum_{\substack{y\in(\Z/p\Z)^\times\\y\not\equiv1(p)}}\chi(y) \Big)\\
&=
\Big((p-1)\chi(C)-\chi(-A)+(p-1)\chi(A) \Big)\\
&=(p-1)C)-\chi(-A)+(p-1)\chi(A).
\end{align*}
(ii) In this case,
\begin{align*}
M=&\Big(
 -\chi(C  )
-\chi(-1) \chi\big( (4AC)B^{-2}(1-(4AC)B^{-2}) \big)(p-1)\chi(-A)\\
&\qquad\qquad-\chi(-1) \sum_{\substack{y\in(\Z/p\Z)^\times\\y\not\equiv1(p)\\ y \not\equiv (4AC)B^{-2}(p)}} \chi(y(1-y)) (-\chi(-A))\\
&\qquad\qquad+(p-1)\chi(A)\sum_{\substack{b,y\in(\Z/p\Z)^\times\\y\not\equiv1(p)\\Ab^2+Bb+C\equiv0 (p)}}\chi(y)-\chi(A)\sum_{\substack{b,y\in(\Z/p\Z)^\times\\y\not\equiv1(p)\\Ab^2+Bb+C\not\equiv0 (p)}}\chi(y) \Big)\\
=&\Big(
 -(p+1)\chi(C)
-\chi(-A)-2(p-1)\chi(A)+(p-3)\chi(A) \Big)\\
=&-(p+1)\chi(C)-\chi(-A)-(p+1)\chi(A).
\end{align*}
(iii) In this case
\begin{align*}
M=&
\Big(
 -\chi(C  )
+\chi(A) \sum_{\substack{y\in(\Z/p\Z)^\times\\y\not\equiv1(p)}} \chi(y(1-y)) \\
&\qquad\qquad+(p-1)\chi(A)\sum_{\substack{b,y\in(\Z/p\Z)^\times\\y\not\equiv1(p)\\Ab^2+Bb+C\equiv0 (p)}}\chi(y)-\chi(A)\sum_{\substack{b,y\in(\Z/p\Z)^\times\\y\not\equiv1(p)\\Ab^2+Bb+C\not\equiv0 (p)}}\chi(y) \Big)\\
=&
\Big(
 -\chi(C  )
-\chi(-A)-(p-1)\chi(A)+(p-2)\chi(A) \Big)\\
=&-\chi(C)-\chi(-A)-\chi(A).
\end{align*}
(iv) In this case
\begin{align*}
M&=
\Big(\chi(A) \sum_{\substack{y\in(\Z/p\Z)^\times\\y\not\equiv1(p)}} \chi(y(1-y)) \\
&\qquad\qquad+(p-1)\chi(A)\sum_{\substack{b,y\in(\Z/p\Z)^\times\\y\not\equiv1(p)\\Ab^2+Bb+C\equiv0 (p)}}\chi(y)-\chi(A)\sum_{\substack{b,y\in(\Z/p\Z)^\times\\y\not\equiv1(p)\\Ab^2+Bb+C\not\equiv0 (p)}}\chi(y) \Big)\\
&=
\Big(-\chi(-A) -(p-1)\chi(A)+(p-2)\chi(A) \Big)\\
&=-\chi(-A)-\chi(A).
\end{align*}
(v) In this case
\begin{align*}
M&=
\Big(
 -\chi(C  )
+\chi(A) \sum_{\substack{y\in(\Z/p\Z)^\times\\y\not\equiv1(p)}} \chi(y(1-y))  \\
&\qquad\qquad+(p-1)\chi(A)\sum_{\substack{b,y\in(\Z/p\Z)^\times\\y\not\equiv1(p)\\Ab^2+Bb+C\equiv0 (p)}}\chi(y)-\chi(A)\sum_{\substack{b,y\in(\Z/p\Z)^\times\\y\not\equiv1(p)\\Ab^2+Bb+C\not\equiv0 (p)}}\chi(y) \Big)\\
&=
\Big(
 -\chi(C  )
-\chi(-A) -2(p-1)\chi(A)+(p-3)\chi(A))\\
&=-\chi(C)-\chi(-A)-(p+1)\chi(A).
\end{align*}
(vi) In this case
\begin{align*}
M&=
\Big(
 -\chi(C  )
-\chi(-A)  -\chi(A)\sum_{\substack{b,y\in(\Z/p\Z)^\times\\y\not\equiv1(p)}}\chi(y) \Big)\\
&=-\chi(C)-\chi(-A)+(p-1)\chi(A).
\end{align*}
(vii) In this case
\begin{align*}
M&=-\chi(C).
\end{align*}
(viii) In this case
\begin{align*}
M&=
\Big(
-(p-1)\chi(A) \sum_{\substack{y\in(\Z/p\Z)^\times\\y\not\equiv1(p)}} \chi(y(1-y)) -\chi(A)\sum_{\substack{b,y\in(\Z/p\Z)^\times\\y\not\equiv1(p)}}\chi(y) \Big)\\
&=(p-1)\chi(-A)+(p-1)\chi(A).
\end{align*}
(ix) It is clear that in this case, $M=0$.
\end{proof}

\end{document}